\documentclass[10pt,times,twoside]{article}

\usepackage{graphicx,latexsym,euscript,makeidx,color,bm}
\usepackage{amsmath,amsfonts,amssymb,amsthm,thmtools,mathtools,mathrsfs,enumerate}
\usepackage[colorlinks,linkcolor=blue,anchorcolor=green,citecolor=red]{hyperref}

\usepackage{geometry}
\def\5n{\negthinspace \negthinspace \negthinspace \negthinspace \negthinspace }
\def\4n{\negthinspace \negthinspace \negthinspace \negthinspace }
\def\3n{\negthinspace \negthinspace \negthinspace }
\def\2n{\negthinspace \negthinspace }
\def\1n{\negthinspace }
     
 \def\sB{\mathscr{B}}

\def\dbE{\mathbb{E}}   \def\cE{{\cal E}}  
\def\dbF{\mathbb{F}} \def\sF{\mathscr{F}}    
\def\dbG{\mathbb{G}}   \def\cG{{\cal G}}      
   \def\cH{{\cal H}}  
   \def\cI{{\cal I}}  
     
   \def\cK{{\cal K}}

\def\dbN{\mathbb{N}} \def\sN{\mathscr{N}}    
     
\def\dbP{\mathbb{P}}     
\def\dbQ{\mathbb{Q}}   \def\cQ{{\cal Q}}  
\def\dbR{\mathbb{R}}     
   \def\cS{{\cal S}}

   \def\cV{{\cal V}}

           \def\Th{\Theta}

\def\ms{\medskip}

\def\ra{\rightarrow}      
 
\def\no{\noindent}        \def\q{\quad}                      
    \def\qq{\qquad}                    
    \def\hb{\hbox}

  \def\deq{\triangleq}               
            \def\({\Big (}
                  \def\){\Big )}
\def\leq{\leqslant}       \def\geq{\geqslant}
          \def\[{\Big[}
\def\cl{\overline}           \def\]{\Big]}
\def\h{\widehat}          \def\tr{\hbox{\rm tr$\,$}}         \def\cd{\cdot}
\def\wt{\widetilde}              \def\cds{\cdots}
\def\ti{\tilde}           \def\cl{\overline}                 
        \def\ts{\times}                    
\def\pa{\partial}

\def\a{\alpha}           \def\g{\gamma}      \def\o{\omega}
\def\b{\beta}            \def\d{\delta}   \def\F{\Phi}     \def\p{\phi}
\def\z{\zeta}         \def\Th{\Theta}  \def\th{\theta}    \def\si{\sigma}
\def\e{\varepsilon}   \def\L{\Lambda}  \def\l{\lambda}        
              \def\k{\kappa}

\theoremstyle{plain}

\theoremstyle{rmk}

\makeatletter

\@addtoreset{equation}{section}
\makeatother

\newtheorem{theorem}{Theorem}[section]
\newtheorem{definition}[theorem]{Definition}

\newtheorem{lemma}[theorem]{Lemma}
\newtheorem{remark}[theorem]{Remark}
\newtheorem{example}[theorem]{Example}

\newtheorem{assumption}{Assumption}

\makeatletter

\@addtoreset{equation}{section}
\makeatother

\allowdisplaybreaks[4]

\begin{document}

\title{\bf
	Maximum Principle of  Stochastic Optimal Control Problems with Model Uncertainty\thanks{This paper is supported by	National Key R\&D Program of China (Grant No.  2022YFA1006102),	National Natural Science Foundation of China (Grant Nos. 72171133 and 12471418), 	Natural Science Foundation of Shandong Province (Grant Nos. ZR2024MA039 and ZR2022MA029), Guangdong Basic and Applied Basic Research Foundation (Grant No.  2025B1515020091), Shenzhen Fundamental Research General Program   (Grant No.  JCYJ20230807093309021).}}

\author{Tao Hao\thanks{School of Statistics and Mathematics, Shandong University of Finance and Economics, Jinan 250014, China (Email: {\tt taohao@sdufe.edu.cn}).}~,~~~
%
Jiaqiang Wen\thanks{Department of Mathematics  and SUSTech International Center for Mathematics, Southern University of Science and Technology, Shenzhen 518055, China (Email: {\tt wenjq@sustech.edu.cn}).}~,~~~
Jie Xiong\thanks{Department of Mathematics and SUSTech International Center for Mathematics, Southern University of Science and Technology, Shenzhen 518055,  China
 (Email: {\tt xiongj@sustech.edu.cn}).}
%
}

\date{}
\maketitle

\no\bf Abstract. \rm
This paper is concerned with the maximum principle of  stochastic optimal control problems, where the coefficients of  the state equation and the cost functional are uncertain, and the system is generally under Markovian regime switching. Firstly, the $L^\b$-solutions of forward-backward stochastic differential equations  with regime switching are given. Secondly,  we obtain the variational inequality by making use of   the continuity of solutions to variational equations with respect to  the  uncertainty parameter $\th$. Thirdly, utilizing the linearization and weak convergence techniques, we prove the necessary stochastic maximum principle and provide sufficient conditions for the stochastic optimal control. Finally, as an application, a risk-minimizing portfolio selection problem is studied.

\ms

\no\bf Key words:\rm\   Forward-backward stochastic differential equation, maximum principle, regime switching, model uncertainty, partial information


\ms

\no\bf AMS subject classifications. \rm 93E20, 60H10, 60H30.

\section{Introduction}\label{Introduction}

       Let $T>0$ and let $(\Omega,\sF, \dbF, \dbP)$ be a complete filtered probability space on which a standard  2-dimensional  Brownian motion
       $(W(\cd), G(\cd))$ and a continuous-time  Markov chain $\a(\cd)$  with a finite state space $\cI=\{1,2,\cd\cd\cd, I\}$
         are defined. $(W(\cd),G(\cd))$ and $\a(\cd)$ are independent.
         The generator of  Markov chain $\a(\cd)$ is denoted by $\L=(\l_{ij})_{i,j\in \cI}$ satisfying $\l_{ij}\geq0$, for $i\neq j\in \cI$  and $\sum_{j=1}^I\l_{ij}=0$, for every $i\in \cI$. For each $s>0$, we set
          $$\begin{aligned}
             \sF^\a_s&=\si\{\a(r): 0\leq r\leq  s\}\vee \sN,\q  \sF_s^{W,\a}=\si\{W(r),\a(r): 0\leq r\leq  s\}\vee \sN,\\
             \q  \cG_s&=\si\{ G(r) : 0\leq r\leq  s\}\vee \sN,\q \sF_s=\si\{W(r),G(r), \a(r): 0\leq r\leq  s\}\vee \sN,
          \end{aligned}$$
          where $\sN$ is the set of all $\dbP$-null subsets. Denote $\dbG=\{\cG_s\}_{s\in[0,T]} $.  $\dbF^\a,  \dbF^{W,\a}, \dbF$ can be
          understood similarly.
       By  $\Th$ we denote the set of all market states,
      which is  a locally compact, complete separable space with distance $d$, and $\cQ$  the set of all possible probability distributions of $\th\in\Th$.
       For $i,j\in \cI$,
      let $M_{ij}(t), i,j=1,\cds,I$ be a purely discontinuous and square integrable martingale with respect to $\sF^\a_t$.
      For a collection of $\dbF^\a$-progressively measurable functions $K(t) =( K_{ij}(t))_{i,j\in\cI}, t\geq0,$
               we denote $K(t)\bullet dM(t)=\sum_{i,j=1,i\neq j}^IK_{ij}(t)dM_{ij}(t)$.
   By $\dbE[\cd]$ we denote the expectation under the probability $\dbP$. Let $V$ be a nonempty convex subset of $\dbR^l$.

                In present work, we consider  the following controlled forward-backward stochastic differential equation (FBSDE, for short): for $t\in[0,T]$,

\begin{equation}\label{4.1}
 \left\{  \begin{aligned}
               dX^v_\th(t) & =  b_\th(t,X^v_\th(t),v(t),\a(t-))dt + \si_\theta(t,X^v_\th(t),v(t),\a(t-))dW(t) \\
                              & \q +\bar{\si}_\th(t,X^v_\th(t),v(t),\a(t-)) d\cl W^v_{\th}(t)
                           +\b_\th(t,X^v_\th(t-),v(t),\a(t-)) \bullet dM(t),\\
                  dY^v_\th(t) & = -f_\th(t,X^v_\th(t),Y^v_\th(t),Z^v_\th(t),\bar{Z}^v_\th(t), \sum_{i,j=1,i\neq j}^IK^v_{\th,ij}(t)
                                     \l_{ij}\mathbf{1}_{\{\a(t-)=i\}},v(t),\a(t-))dt\\
                               &\q  +Z^v_\th(t)dW(t) + \bar{Z}^v_\th(t)dG(t) + K^v_\th (t)\bullet dM(t),\\
                   X^v_\th(0)& = x, \q     Y^v_\th(T) = \F_\th(X^v_\th(T), \a(T)),
                \end{aligned}  \right.
               \end{equation}
               where $x\in \dbR$; $v(\cd)$ is a control; $\th$ comes from $\Th$;  the assumption on
               $(b_\th,\si_\th,\bar{\si}_\th,\b_\th,f_\th, \F_\th)$ refers to \autoref{assumption3};
               the process
               $\cl {W}^v_{\th}(\cd)$ satisfies

               \begin{equation}\label{4.2}
                 \left\{ \begin{aligned}
                  d\cl {W}^v_{\th}(t)&=-h_\th(t,X^v_\th(t), \a(t-))dt+dG(t),\q t\in[0,T], \\
                   \cl {W}^v_{\th}(t) &=0,
                 \end{aligned}\right.
               \end{equation}
               with $h_\th(\cd)$ being a continuous bounded map (see \autoref{assumption3} for more details).
             Note that  the process $\cl {W}^v_{\th}(\cd)$ depends on the control $v(\cd)$ and the parameter $\th$.
               Suppose that the state process $(X^v_\th(\cd),Y^v_\th(\cd),Z^v_\th(\cd),\bar{Z}^v_\th(\cd),K^v_\th(\cd) )$ cannot be directly observed.
                Instead, we can observe a related process $G(\cd)$. Hence, the control is $\dbG$-adapted, as described below.
\begin{definition}
A process $v:\Omega\ts [0,T]\ra V$ is called to be an admissible control if $v(t)$ is $\cG_t$-adapted, and satisfies $\sup_{t\in[0,T]}\dbE[|v(t)|^4]<\infty$. The set of admissible controls is denoted by $\cV_{ad}$.
\end{definition}
            Additionally,  thanks to Girsanov theorem, we can define an equivalent probability by
            \begin{equation}\label{4.6-111}
             d \overline{\dbP}^v_\th:=R^v_\th(t)d\dbP \q \text{on}\q  \mathcal{F}_t,
             \end{equation}
              for each $t\in[0,T]$, where $R^v_\th(\cd)$ is the solution to the following stochastic differential equation (SDE):

             \begin{align}\label{4.7-1}
               \begin{aligned}
                 dR^v_\theta(t)=R^v_\theta(t)h_\theta(t,X^v_\theta(t), \alpha(t-))dG(t),\  t\in[0,T],\q \text{and}\q
                 R^v_\theta(0)=1.
               \end{aligned}
             \end{align}
              Clearly,  $(W(\cd),\cl W^v_\th(\cd))$ is a $2$-dimensional standard Brownian motion under the  probability   $\overline{\dbP}^v_\th$.
              By $\dbE^{\bar{\dbP}^v_\th}[\cd]$ we denote the expectation under the probability $\bar{\dbP}^v_\th$.
              %
              Under \autoref{assumption3} below, the above equations (\ref{4.1}), (\ref{4.2}) and (\ref{4.7-1}) admit unique solutions.
              This implies that for each $\th\in \Th$ (or, say, for each market state), similar to   classical optimal control problems
              (see \cite{Wang-Wu-Xiong-13,Xiong-Zhou-2007,Zhang-Xiong-Liu-2018}),
               one can consider the cost functional
               \begin{equation*}
               J_\th(v(\cd)):=\dbE^{\bar{\dbP}^v_\th}[\Psi_\th(X_\th^v(T))+\L_\th(Y^v_\th(0))],
               \end{equation*}
               where the assumption on  $(\Psi_\th(\cd), \L_\th(\cd))$ {are given in} \autoref{assumption4}.
               However, the parameter $\th$ is different
               for different market states.
               Obviously, it is unreasonable to only study the cost functional $J_\th(v(\cd))$ for some given $\th$.
               Inspired by the spirit of Hu and Wang \cite{Hu-Wang-20}, we prefer to consider the following robust cost functional
              \begin{align}\label{4.2222}
               J(v(\cd))=\sup_{Q\in \cQ}\int_{\Th} J_\th(v(\cd)) Q(d\th)
              \end{align}
              instead of $J_\th(v(\cd))$ for some given $\th$. In this setting,
              our optimal control problem can be  described as follows.

\ms

               {\bf Problem (PRUC)}. Find a  $\cG_t$-adapted optimal control $\h{v}(\cdot)$ such that
               \begin{align} \nonumber
               J(\h v(\cd))=\inf_{v(\cd)\in \cV_{ad} } J(v(\cd)),
               \end{align}
               subject to (\ref{4.1}) and (\ref{4.2}).


\subsection{Motivation}

 Let us explain the motivation for studying the above optimal control problem by using an example from finance: specifically, finding risk-minimizing portfolios under model uncertainty. Note that the risk is characterized in terms of $g$-expectation (see Peng \cite{Peng-2004}). A similar model, which does not incorporate regime switching, is examined by Xiong and Zhou \cite{Xiong-Zhou-2007}, as well as by Wang, Wu, and Xiong \cite{Wang-Wu-Xiong-13} and Zhang, Xiong, and Liu \cite{Zhang-Xiong-Liu-2018}.

\begin{example}\label{ex 1.1}
Assume that a market consists of a bond and a stock. Let $v(\cd)$ be the amount invested in the stock.
Furthermore,  suppose that the share market is roughly divided into either a bull or  bear market. Let $\th=1$ and 2 denote a bull market and a bear market, respectively, which is characterized as  $\th\in\Th=\{1,2\}$.  Let
$\cQ=\{Q^\l: \l\in[0,1]\},$
where $Q^\l$ is the  probability  such that $Q^\l(\{1\})=\l$ and $Q^\l(\{2\})=1-\l.$
Let $(\Omega,\sF, \dbF, \overline{\dbP}^v_\th)$ be a probability space, where  the probability $\overline{\dbP}^v_\th$ depends on the process $v(\cd)$ and the parameter $\th$.
Let $(W(\cd),\cl W^v_\th(\cd))$ be a 2-dimensional standard Brownian motion
and the process  $\a(\cd)$ with a finite state space $\cI=\{1,2,\cd\cd\cd, I\}$   a continuous-time Markov chain
 under the probability $\overline{\dbP}^v_\th$.
%

\ms

Assume that the prices of  the  bond and the stock  are $S_0(t)$ and $S_1(t)$, respectively, which are driven by the following equations:
\begin{equation*}
dS_{0}(t)=S_{0}(t)r(t)dt \q\hb{and}\q
dS_{1}(t)=S_{1}(t)\{\mu(t)dt+\si(t)d\overline{W}^v_\th(t)\}\q
(\hb{\hb{\rm under}}\  \overline{\dbP}_\th^v), \q t\in[0,T],
\end{equation*}
where $r(t)$ is the interest rate;
$\mu(t)$ is the appreciation rate with $\mu(t)>r(t)$ for $t\in[0,T]$; $\si(t)\neq0$ is the volatility.
Under a
self-financed portfolio, the wealth process of an agent, starting with an initial
wealth $x_0$, satisfies the following  equation:
\begin{equation}\label{1.2}
\left\{
\begin{aligned}
 x_\th^v(0)&=x_0,\\
 dx^v_\th(t)&=\{r(t)x_\th^v(t)+(\mu(t)-r(t))v(t)\}dt+\si(t)v(t)d\overline{W}_\th^v(t)-\b(t)\bullet dM(t), \q t\in[0,T],
\end{aligned}
\right.
\end{equation}
where  {$v(\cd)$ represents the investment strategy}; $\b(t)>0,\ t\in[0,T];$  $M_{ij}(\cd),\  i,j\in \cI$ is the canonical martingale of the  Markov chain  $\a(\cd)$;
$\b(t)\bullet dM(t):=\sum_{i,j=1,i\neq j}^I\b_{ij}(t)dM_{ij}(t)$  represents the  cumulative cost caused by regime switching jumps up to time $t$.
Define $G(t):= \frac{1}{\si(t)}\log S_{1}(t)$. It is an observation process since the policymaker can obtain information from the stock price, which satisfies
\begin{equation}\label{1.3}
dG(t)=\frac{1}{\si(t)}\Big\{\mu(t)-\frac{1}{2}\si^2(t)\Big\}dt+d\overline W^v_\th(t)\q \hb{and}\q
G(0)=0.
\end{equation}
Note that,  in a real market, it is reasonable to assume the amount $v(t)$ being adapted to the filtration $\cG_t:= \si\{G(r): 0\leq r\leq  t\}$.

For any $\xi\in L^2(\sF_T; \dbR^m)$, by $(y^v_\th(\cd),z^v_\th(\cd), \bar{z}^v_\th(\cd), k^v_\th(\cd))$ we denote the solution to the following backward stochastic differential equation (BSDE, for short) with regime switching
\begin{equation}\label{1.4}
\left\{
\begin{aligned}
dy^v_\th(t) & = -g_\th(t,y^v_\th(t),z^v_\th(t), \bar{z}^v_\th(t),\sum_{i,j=1,i\neq j}^Ik^v_{\th,ij}(t)
                                     \l_{ij}\mathbf{1}_{\{\a(t-)=i\}})dt\\
                               &\q  +z^v_\th(t)dW(t) + \bar{z}_\th^v(t)dG(t) + k^v_\th (t)\bullet dM(t),\\
  y^v_\th(t)&=\xi,
  \end{aligned}
\right.
\end{equation}
where $g_\th(\cd)$ is a suitable map. Furthermore,
if $g_\th(s,y,0,$ $0,0)=0, (s,y)\in[0,T]\ts \dbR$,
     following the spirit of Peng \cite{Peng-2004}, we
    can define the   $g$-expectation:
    $$
     \cE_{g,\th}^v[\xi]:=y^v_\th(0).
     $$
     $\cE_{g,\th}^v[\cd]$ is a nonlinear expectation, which has almost the properties of classical expectation $\dbE[\cd]$, except
     for linearity.
      Taking $\xi=-x^v_\th(T)$, the $g$-expectation $\cE_{g,\th}^v[-x^v_\th(T)]$ could be regarded as   a risk measure of $x^v_\th(T)$, see Rosazza and Gianin \cite{Rosazza-Gianin-2006}.
Next, we consider  the following cost functional
    \begin{equation*}
    \begin{aligned}
     J_g(v(\cd))&=\max\bigg\{\cE_{g,1}^v[-x^v_1(T)], \cE_{g,2}^v[-x^v_2(T)]\bigg\},
     \end{aligned}
     \end{equation*}
 which characterizes the maximum risk under all  market  states.
The objective is to minimize $J_g(v(\cd))$, which is called a robust optimal control problem (see Hu and Wang \cite{Hu-Wang-20}).
Moreover,  according to the definition of $g$-expectation,   the above cost functional  can be reformulated as
   \begin{equation*}
    \begin{aligned}
     J_g(v(\cd))&=\max\Big\{y_1^v(0), y_2^v(0)\Big\}=  \sup_{\l\in[0,1]}\Big\{ \l y^v_1(0)+(1-\l) y^v_2(0) \Big\}=\sup_{Q^\l\in \cQ}\int_\Th y_\th^v(0) Q^\l(d\th).
     \end{aligned}
     \end{equation*}
Thereby, we  deduce a new framework that is considering a risk-minimizing portfolio selection problem under model
uncertainty.

{\rm\bf  Problem}. Find a $\cG_t$-adapted process $\widehat{v}(t)$ such that
\begin{equation*}
J_g(\widehat{v}(\cd))=\inf\limits_{v(\cd) }J_g(v(\cd))
=\inf\limits_{v(\cd) }\sup_{Q^\l\in \cQ}\int_\Th y_\th^v(0) Q^\l(d\th),
\end{equation*}
subject to (\ref{1.2})-(\ref{1.4}).
\end{example}

\subsection{Literature Review}
Classical stochastic control theory appeared with the advent of stochastic analysis and has developed rapidly in recent decades due to its wide range of applications (see Yong and Zhou \cite{Yong-Zhou-99}). One of the primary approaches to solving stochastic optimal control problems is Pontryagin's maximum principle, which renders these problems more tractable. The fundamental idea of the stochastic maximum principle is to establish a set of necessary and sufficient conditions that must be met by any optimal control.
In the seminal paper \cite{Peng-1990}, Peng formulated a global maximum principle for stochastic optimal control problems, building on the nonlinear backward stochastic differential equations introduced by Pardoux and Peng \cite{Pardoux-Peng-1990}. Since then, many researchers have extensively studied the stochastic maximum principle using systems of forward-backward stochastic differential equations (FBSDEs). For further details, we refer  readers to \cite{Hu-Ji-Xu-2022, Hu-Ji-Xue-2018, Hu-Wang-20, Wu-2013} and the references therein.
%

Recently, there has been a dramatic increase in interest in studying the stochastic maximum principle for optimal control problems involving random jumps, such as Poisson jumps (see Tang and Li \cite{TL}) or regime-switching jumps. These approaches are of practical importance in various fields, including economics, financial management, science, and engineering.
For instance, one might encounter two market regimes in financial markets: one representing a bull market with rising prices, and the other representing a bear market with falling prices. This scenario can be modeled as a regime-switching model, where market parameters depend on modes that switch among a finite number of regimes.
More recently, the applications of the stochastic maximum principle to optimal control problems with regime-switching systems or Poisson jumps have been extensively developed. For example, Zhang, Elliott, and Siu \cite{Zhang-Elliott-Siu-2012} studied a stochastic maximum principle for a Markovian regime-switching jump-diffusion model and its applications in finance. Li and Zheng \cite{Li-Zheng-15} considered the weak necessary and sufficient stochastic maximum principle for Markovian regime-switching systems.  Zhang, Sun, and Xiong \cite{Zhang-Sun-Xiong-2018} established a general stochastic maximum principle for mean-field Markovian regime-switching systems. Zhang, Xiong, and Liu \cite{Zhang-Xiong-Liu-2018} examined a stochastic maximum principle for partially observed forward-backward stochastic differential equations (FBSDEs) with jumps and regime switching. Sun, Kemajou-Brown, and Menoukeu-Pamen \cite{Sun-Kemajou-Menoukeu-2018} investigated a risk-sensitive maximum principle for a Markov regime-switching system. Dong, Nie, and Wu \cite{Dong-Nie-Wu-2022} derived the maximum principle for mean-field stochastic control problems with jumps. Wen et al. \cite{Wen-Li-Xiong-Zhang-2022} explored related stochastic linear-quadratic optimal control problems involving Markovian regime-switching systems and their applications in finance.
For additional important works, we refer readers to Donnelly and Heunis \cite{Donnelly-Heunis-2012}, {{\O}ksendal and Sulem } \cite{Oksendal-Sulem-2005}, Song, Stockbridge, and Zhu \cite{Song-Stockbridge-Zhu-2011}, Zhang and Zhou \cite{Zhang-Zhou-2009}, and the references therein.


From \autoref{ex 1.1}, we know that in a real market, it is more reasonable and general to allow for model uncertainty alongside the Markov chain, as these uncertainties can significantly impact interest rates, stock prices, and volatilities.
{For example,
Maenhout \cite{Maenhout-2004}  proposed a new approach to the dynamic portfolio and consumption of  investors in the context where they are concerned about model uncertainty.
Maenhout \cite{Maenhout-2006} gave the optimal portfolio choice under model uncertainty and stochastic premia. {\O}ksendal, and  Sulem \cite{Oksendal-Sulem-2012} analyzed forward-backward stochastic differential games and stochastic control under model uncertainty.
Yi et al. \cite{Yi-2013} discussed  a robust optimal reinsurance and investment problem  for an  insurer  worried about model uncertainty.
Menoukeu-Pamen \cite{Menoukeu-2017} investigated  a problem of robust utility maximization under a relative entropy penalty and gave optimal
investment of an insurance firm under model uncertainty.
}
However, there are few cases where both model uncertainty and regime switching work are considered simultaneously.
In this paper, we aim to study the maximum principle for a Markovian regime-switching system under model uncertainty. To achieve a general result, we prefer to focus on the broader scenario where the state equation is governed by a forward-backward stochastic differential equation (FBSDE) system with partial information, as there are numerous scenarios involving partial information in financial models. For relevant optimal control problems with partial information, we refer  readers to \cite{Menoukeu-2017, Tang-1998, Wang-Wu-2009, Wang-Wu-Xiong-13, Xiong-Zhou-2007}, etc.

\subsection{The Contribution of This Paper}
            We now present our main contributions and difficulties in detail.

\begin{itemize}

\item [(i)]
             The optimal control problem with partial information and regime switching under model uncertainty ({\bf Problem (PRUC)}) is
formulated. The critical approaches in the maximum principle investigation are linearization and weak convergence techniques.
First, notice that the cost functional $J(v(\cd))$ involves the probability $\bar{\dbP}^v_\th$.
To study  {\bf Problem (PRUC)}, making use of Bayes' formula, we rewrite the cost functional (\ref{4.2222}) as a new form associated with the probability $\mathbb{P}$ (see (\ref{4.10})), which does not depend on $v$ and $\th$.
Second, for the sake of obtaining the variational inequality (\autoref{le 5.3}),  we need to show the continuity of solutions to the variational equation of FBSDE  (\ref{4.1}) with respect to $\th$ (\autoref{le 6.3-0}).
The variational inequality is established based on the above result (\autoref{le 5.3}).
Third, the primary material to develop the necessary maximum principle, i.e., adjoint equation,  is introduced.  Fubini's theorem proves the stochastic maximum principle by investigating the boundness and continuity of its solutions with respect to $\th$.
Finally,  sufficient conditions for  optimal control $\h v$ are obtained. Its proof mainly depends on some convex assumptions on the Hamiltonian.

\item [(ii)] We propose   risk-minimizing portfolio selection problems in a general framework. Applying the above obtained theoretical results for  a special risk-minimizing portfolio selection problem, its explicit solution  is given.
           %
\end{itemize}
 Since Markovian regime switching jumps and partial information appear in our model, compared with
                  Hu and Wang \cite{Hu-Wang-20},  the difficulties of the present work  mainly come from the boundness
                  and continuity of  solutions  to    variational equations.

\begin{itemize}
	
%

\item [(iii)]
        On the one hand, since partial information is considered in our model, the original cost functional (\ref{4.2222}) depends on the
                   probability $\bar{\dbP}^v_\th$. It is difficult to obtain the variation of the cost functional  (\ref{4.2222}) directly,  since the probability  $\bar{\dbP}^v_\th$ changes when the control $v$ changes. In order to avoid this obstacle, we apply
                   Bayes's formula to rewrite the above cost functional as a new cost functional (\ref{4.10}), which depends on the probability $\mathbb{P}$, but not $\bar{\dbP}^v_\th$. Following this line, \textbf{Problem} (\textbf{PRUC}) is equivalent to minimizing
                   the new optimization problem (\ref{4.3})-(\ref{4.7-1})-(\ref{4.10}).

\item [(iv)] On the other hand, since the cost functional in Hu and Wang \cite{Hu-Wang-20} only involves $Y^v_\th(0)$, relatively speaking,
                    the variation of the cost functional follows directly.
                   But   the cost functional (\ref{4.10}) in the present work contains  $R^v_\th(T), X_\th^v(T), Y^v_\th(0)$, the calculation of
                    variation of the cost functional (\ref{4.10}) is more involved. For example, the boundness and  continuity
                    of the triple $(R^v_\th(T), X_\th^v(T), Y^v_\th(0))$ with respect to $\th$,
                    the boundness and  continuity
                    of the function $H(\th)$  (see (\ref{4.9999}))  with respect to $\th$,
                    the boundness and continuity of solutions to the  adjoint equations (\ref{4.36}) and (\ref{4.37}) with respect to $\th$,
                    and so on. Besides, owing to the presence of partial information and Markovian regime switching, we propose  new
                    adjoint equations, which are different from those of Hu and Wang \cite{Hu-Wang-20}, and  Wang, Wu, and Xiong \cite{Wang-Wu-Xiong-13}.

\end{itemize}

{
Note that the uncertainty of our model is different from that in the situation described in the work \cite{Maenhout-2004,Maenhout-2006,Yi-2013,Menoukeu-Momeya-2017,Oksendal-Sulem-2012}.
Indeed, Maenhout \cite{Maenhout-2004,Maenhout-2006} and Yi et al. \cite{Yi-2013} observed that decision-makers often adopt a ``reference model" while remaining cautious of its potential misspecification. Under certain suitable assumptions, this model uncertainty corresponds to uncertainties concerning the drift coefficient (see, for example, Maenhout \cite[Eq. (12)]{Maenhout-2006}).
Another aspect of model uncertainty is Knightian uncertainty, which refers to uncertainty regarding the underlying probability measure. In this framework, the cost functional is examined under the expectation $\dbE^{\dbQ}$, where the probability $\dbQ$ is absolutely continuous with respect to a given probability $\dbP$.
In the works of Menoukeu-Pamen and Momeya \cite{Menoukeu-Momeya-2017}, and  {\O}ksendal and Sulem \cite{Oksendal-Sulem-2012}, the authors assumed that the Radon-Nikodym derivative $M$ of the probability $\dbQ$ with respect to $\dbP$ is an It\^o process, parameterized by $\th$. It is important to note that in their studies, the coefficients of their state equations are treated as deterministic.
In contrast to their works, our model features not only uncertainties in the drift coefficients of the state equations and the coefficients of the observation processes (i.e., the Radon-Nikodym derivative), but also incorporates uncertainties in  diffusion coefficients, which are also parameterized by $\th$.
}

The paper is organized as follows. In  \autoref{Sec2}, we introduce some spaces and prove the existence and uniqueness of FBSDEs with regime switching.  \autoref{Sec3} discusses the equivalent problem of \textbf{Problem} (\textbf{PRUC}). The necessary maximum principle and sufficient conditions are presented in  \autoref{Sec4}. In  \autoref{Sec5}, we study  a risk-minimizing portfolio selection problem. Finally, some conclusions are drawn in  \autoref{Sec6}.

\section{Preliminaries} \label{Sec2}


By $\dbR^n$, we denote the $n$-dimensional real Euclidean space, and by $\dbR^{n\times d}$, we denote the set of all $n \times d$ real matrices. The set of natural numbers is represented by $\dbN$. The scalar product of $C = (c_{ij}), D = (d_{ij}) \in \dbR^{n \times d}$ is defined as $\langle C, D \rangle = \tr\{CD^\top\}$, where the superscript $\top$ indicates the transpose of vectors or matrices.
Additionally, we consider a continuous-time Markov chain $\a(\cd)$ with a finite state space $\cI = \{1, 2, \cdots, I\}$. The generator of the Markov chain $\a(\cd)$ is denoted by $\L = (\l_{ij}){i,j \in \cI}$, which satisfies $\l_{ij} \geq 0$ for $i \neq j \in \cI$ and $\sum_{j=1}^I \l_{ij} = 0$ for every $i \in \cI$. For each pair $(i,j) \in \cI \times \mathcal{I}$ with $i \neq j$, we define
\begin{equation*}\label{2.4}
	[M_{ij}](t)=\sum_{0\leq r\leq t} \mathbf{1}_{\{\a(r-)=i\}}\mathbf{1}_{\{\a(r)=j\}},\q
	\langle  M_{ij} \rangle(t)=\int_0^t \l_{ij} \mathbf{1}_{\{\a(r-)=i\}}dr,
\end{equation*}
where $\mathbf{1}_A$ denotes the indicator function of the set $A$. According to \cite{Donnelly-Heunis-2012, Li-Zheng-15, Nguyen-Yin-Nguyen-2021}, the process
$M_{ij}(t):=[M_{ij}](t)-\langle M_{ij}\rangle(t)$ is a discontinuous and square-integrable martingale with respect to the filtration $\sF^\a_t$, and it equals zero at the starting point.
Furthermore, the process $[M_{ij}]=([M_{ij}](t))_{t\in[0,T]}$ represents the optional quadratic variation, while $\langle M_{ij}\rangle=(\langle M_{ij}\rangle(t))_{t\in[0,T]}$  denotes the quadratic variation. By the definition of optional quadratic covariations, we have
\begin{equation*} \label{2.6}
	[M_{ij},W]=0,  \quad [M_{ij}, M_{mn}]=0,\  \text{as}\  (i,j)\neq(m,n).
\end{equation*}
For simplicity, we define $M_{ii}(t)=[M_{ii}](t)=\langle M_{ii}\rangle(t)=0$, for each $i\in \cI$ and for $t\in [0,T]$.
For $S=M,[M],\langle M\rangle,$  we denote
\begin{equation*}
	\begin{aligned}
		\int_0^tK(r)\bullet dS(r)&=\int_0^t\sum_{i,j=1,i\neq j}^IK_{ij}(r)dS_{ij}(r), \q
		K(r)\bullet dS(r)&=\sum_{i,j=1,i\neq j}^IK_{ij}(r)dS_{ij}(r).
	\end{aligned}
\end{equation*}

The following spaces are used frequently. For $\beta\geq2,$ we define:

\begin{itemize}
	\item
	$L^\b(\sF_t;\dbR^n)$ is the family of $\dbR^n$-valued $\sF_t$-measurable random variables $\xi$ with $\dbE|\xi|^\b<\infty.$
	\item ${\cS}_{\dbF}^{\b}(0,T;\dbR^n)$ is the family of $\dbR^n$-valued $\dbF$-adapted c\`{a}dl\`{a}g processes $(\psi(t))_{0\leq t\leq
		T}$  with  $$\|\psi(\cd)\|_{\cS_\dbF}=\dbE\[\mathop{\rm sup}\limits_{0\leq t\leq T}| \psi(t) |^{\b}\Big]^\frac{1}{\b}< +\infty.$$
	\item $\cH_{\dbF}^{1, \beta}(0,T;\dbR^{n})$ is the family of $\dbR^n$-valued $\dbF$-progressively measurable processes $(\psi(t))_{0\leq t\leq T}$ with       $$\dbE\[\(\int^{T}_{0} |\psi(t)|dt\)^\beta\]<+\infty.$$
	\item $\cH_{\dbF}^{2, \frac{\b}{2}}(0,T;\dbR^{n})$ is the family of $\dbR^n$-valued $\dbF$-progressively measurable processes $(\psi(t))_{0\leq
		t\leq T}$ with     $$\|\psi(\cd)\|_{\cH_\dbF}=\dbE\[\(\int^{T}_{0} |\psi(t)|^2dt\)^\frac{\b}{2}\]^{\frac{1}{\b}}<+\infty.$$
	\item $\cK_{\dbF}^{\b,1}(0,T;\dbR^n)$ is the family of $K(\cd)=(K_{ij}(\cd))_{i,j\in \cI}$ such that the $\dbF$-progressively
	measurable processes $K_{ij}(\cd)$ satisfies  $K_{ii}(t)=0, t\in[0,T]$ and   $$\dbE\[\int^T_0 \sum_{i,j=1, i\neq j}^I |K_{ij}(t)|^\b\l_{ij}\mathbf{1}_{\{\a(t-)=i\}}dt \]< \infty.$$
	\item $\cK_{\dbF}^{2,\frac{\b}{2}}(0,T;\dbR^n)$ is the family of $K(\cd)=(K_{ij}(\cd))_{i,j\in \cI}$ such that the $\dbF$-progressively
	measurable processes $K_{ij}(\cd)$ satisfies  $K_{ii}(t)=0, t\in[0,T]$ and
	$$\|K(\cd)\|_{\cK_\dbF}:=\dbE\[\(\int^T_0 \sum_{i,j=1, i\neq j}^I |K_{ij}(t)|^2\l_{ij}\mathbf{1}_{\{\a(t-)=i\}}dt \)
	^\frac{\b}{2}\]^{\frac{1}{\b}}< \infty.$$
\end{itemize}
%

\subsection{FBSDEs with Regime Switching}

In this subsection, we set $\b \geq 2$ and recall the notations $W$, $M$, $\a$, $K \bullet M$, and $\dbF^{W, \a}$ from  \autoref{Introduction}.

{To study the well-posedness of  SDEs with regime switching and BSDEs with regime switching, let us first introduce a lemma.
	Similar to the proof of Kunita's inequality (see  Kunita \cite[pp. 333-334]{Kunita-2004}), we have the following estimate with
	respect to  stochastic calculus related to the canonical martingale for  Markov chain.}

\begin{lemma}\label{le 6.0}\rm
  {
  	 For  $\p(\cd)=(\p_{ij}(\cd))_{i,j\in\cI}\in\cK_{\dbF^\a}^{\b,1}(0,T;\dbR^m)$, there exists a positive constant $C$ depending on  $T,\b,\sum_{i,j=1,i\neq j}^I\l_{ij}$ such that, for $0\leq t\leq T$,}
              \begin{align}\nonumber
              \begin{aligned}
            {  \dbE\bigg[ \sup\limits_{0<s\leq t}\(\int_0^s|\p(r)|\bullet d M(r) \)^\b\bigg]
              \leq C
               \dbE\bigg[ \int_0^t \sum_{i,j=1,i\neq j}^I|\p_{ij}(s)|^\b  \l_{ij}\mathbf{1}_{\{\a(s-)=i\}}ds \bigg] .}
              \end{aligned}
              \end{align}
\end{lemma}

              Consider the following SDE with regime switching:
               \begin{align}\label{6.111}
              \begin{aligned}
              X(t)= \ & x+\int_0^tb(s,X(s),\a(s-))ds+\int_0^t\si(s,X(s),\a(s-))dW(s)\\
                    & +\int_0^t\b(s,X(s-),\a(s-))\bullet d M(s),\ t\in[0,T].
              \end{aligned}
              \end{align}
%
             For each pair $(i_0,j_0)\in \cI \ts \cI$, let $b:[0,T]\ts \Omega\ts \dbR^n \ts \cI\ra\dbR^n,$ $\si:[0,T]  \ts \Omega \ts \dbR^{n} \ts \cI\ra \dbR^{n\ts d},$ $\b_{i_0j_0}:[0,T]\ts \Omega\ts \dbR^n \ts \cI\ra\dbR^n$ satisfy the following assumption.

            \begin{assumption}\label{assumption1} \rm
            	\begin{itemize}
            		\item [(i)] There exists a positive constant  $L$ such that, for $t\in[0,T]$, $x,x'\in \dbR^n$, $i\in \cI$,
            		\begin{align}\nonumber
            			|b(t,x,i)-b(t,x',i)|+|\sigma(t,x,i)-\sigma(t,x',i)|+|\beta_{i_0j_0}(t,x,i)-\beta_{i_0j_0}(t,x',i)|\leq L|x-x'|.
            		\end{align}
            		\item[(ii)]  For   $i\in \cI$, $b(\cd,0,i)\in \cH_{\dbF^{W,\a}}^{1,\b}(0,T;\dbR^n), $
                         $\beta(\cd,0,i)=(\beta_{i_0j_0}(\cd,0,i))_{i_0,j_0\in \cI} \in\cK_{\dbF^{W,\a}}^{\b,1}(0,T;\dbR^{n}),$ $\si(\cd,0,i)
            		\in\cH_{\dbF^{W,\a}}^{2,\frac{\b}{2}}(0,T;\dbR^{n\ts d}).$
            	\end{itemize}
             \end{assumption}
            With \autoref{le 6.0} in hand, following the standard argument for classical SDEs and BSDEs,
            we have the well-posedness of FBSDEs with regime switching.
\begin{theorem}\label{th 3.2}\rm
              Under \autoref{assumption1},  SDE with regime switching (\ref{6.111}) possesses a unique solution $X\in \cS_{\dbF^{W,\a}}^{\b}(0,T;\dbR^n).$    Furthermore, there exists a constant $C>0$
              depending on $L,T,\b,\sum\limits_{i,j=1,i\neq j}^I\l_{ij}$ such that
               \begin{align}\label{6.112}
               \begin{aligned}
                  &  \dbE\[\sup_{0\leq t\leq T}|X(t)|^\b \]\leq C\dbE\bigg[|x|^\b+\(\int_0^T|b(t,0,\a(t-))|dt\)^\b+\(\int_0^T|
                         \si(t,0,\a(t-))|^2dt\)^\frac{\b}{2}\\
                  & \qq\qq\qq\qq\q
                  +\(\int_0^T\sum_{i,j=1,i\neq j}^I|\b_{ij}(t,0,\a(t-))|^\b\l_{ij}\mathbf{1}_{\{\a(t-)=i\}}dt \)
                  \bigg].
                 \end{aligned}
               \end{align}
\end{theorem}

             Next, let us focus on the following  BSDE:
              \begin{equation}\label{6.1-1111}
                \begin{aligned}
                 Y(t)&=\xi+\int_t^TF(s,Y(s),Z(s),\sum_{i,j=1,i\neq j}^IK_{ij}(s)\l_{ij}\mathbf{1}_{\{\a(s-)=i\}},\a(s-))ds\\
                     &\q -\int_t^TZ(s)dW(s)-\int_t^TK(s)\bullet dM(s), \ t\in[0,T].
                \end{aligned}
             \end{equation}
            Similar to  (\ref{6.111}), the above equation   is called  a  \emph{BSDE with regime switching}.
           Assume $F: \Omega\ts [0,T] \ts \dbR^m  \ts   \dbR^{m\ts d}  \ts \dbR^m   \ts \cI\rightarrow  \dbR^m$ satisfies

             \begin{assumption}\label{assumption2} \rm
             		There exists some constant $L>0$ such that, for $t\in[0,T]$,   $y,y',k,k'\in \dbR^m$, $z,z'\in \dbR^{m\ts d}$, $i\in \cI$,
             		$$ |F(t,y,z,k,i)-F(t,y',z',k',i)|\leq L(|y-y'|+|z-z'|+|k-k'|)\ \  \hb{and}\ \
             		\dbE\[\int_0^T |F(t,0,0,0,i)|^2dt\]<\infty.$$
             \end{assumption}

\begin{theorem}\label{th 6.1}\rm
           Under \autoref{assumption2},  for $\xi\in L^2(\sF^{W,\a}_T; \dbR^m)$, BSDE (\ref{6.1-1111}) admits a unique solution $(Y(\cd),Z(\cd),K(\cd))\in \cS^2_{\dbF^{W,\a}}(0,T;\dbR^m) \ts \cH^{2,1}_{\dbF^{W,\a}}(0,T;\dbR^{m\ts d})
                  \ts {\cK}_{\dbF^{W,\a}}^{2,1}(0,T;\dbR^m).$ Moreover, there exists a constant $C>0$ depending on $L,T,\sum\limits_{i,j=1,i\neq j}^I\l_{ij}$  such that
           \begin{equation}\label{6.4-1}
              \begin{aligned}
              &\dbE\bigg[\sup_{t\in[0,T]}|Y(t)|^2+\int_0^T(|Z(t)|^2+\sum_{i,j=1,i\neq j}^I|K_{ij}(t)|^2\l_{ij}\mathbf{1}_{\{\a(t-)=i\}})dt\bigg] \\
              &\leq  C\dbE\bigg[|\xi|^2+\int_0^T |F(t,0,0,0,\a(t-))|^2dt \bigg].
              \end{aligned}
           \end{equation}
\end{theorem}

{
\begin{remark}\rm
  One can also establish the existence and uniqueness of BSDE (2.3) by Papapantoleon et al. \cite[Theorem 3.5]{Papapantoleon-2018}.
Indeed,  Papapantoleon et al.  \cite[Eq. (3.1)]{Papapantoleon-2018} reduces to the above Eq. (2.3)  if we set   $C(t)=t, \ X^o(t)=W(t)$, and $\int_{\dbR^n}K(t,e)\widetilde{\mu}^\natural(dt,de)=K(t)\bullet dM(t)$ with $N(t)\equiv0$.
Furthermore, since  the Lipschitz constant of $F$ with respect to $(y,z,k)$ does not depend on $(t,\o)$,
we can take $\Phi\equiv0$    according to Papapantoleon et al. \cite[(3.3) and (3.4)]{Papapantoleon-2018}.
Consequently, the function $M^\Phi_{\star}(\beta)$ defined in  Papapantoleon et al. \cite[Lemma 3.4]{Papapantoleon-2018}  yields the following estimate:
\begin{align*}
	M^0_{\star}(\beta)\leq \Pi^0_{\star}(\frac{\beta}{2},\b)=\frac{61}{\beta}.
\end{align*}
By choosing  $\widehat{\beta}=200$, it is evident that $M^0_{\star}(\widehat{\beta})<\frac{1}{2}.$
Hence, according to  Papapantoleon et al.  \cite[Theorem 3.5]{Papapantoleon-2018},
the desired result follows.
\end{remark}
}

\section{Equivalent Problem of {\bf Problem (PRUC)} }\label{Sec3}

Since the random noise $\overline{W}^v_\th(\cd)$ depends on the control $v(\cd)$, a coupled circle arises in handling  {\bf Problem (PRUC)}, i.e., the determination of the control $v(\cd)$ will depend on the observation $\overline{W}^v_\th(\cd)$, and  $\overline{W}^v_\th(\cd)$ also depends on the control $v(\cd)$, which brings an immediate difficulty. In order to solve this difficulty, let us look at the observation process $G(\cd)$ (see (\ref{4.2})) carefully. The main characteristics of $G(\cd)$ are that, on the one hand, it is independent of the control $v(\cd)$ and the parameter $\th$, and on the other hand, it is a Brownian motion under the probability $\dbP$. It needs to be pointed out that the probability $\dbP$ is independent of the control $v(\cd)$ and the parameter $\th$. This implies that we can consider an equivalent problem of {\bf Problem (PRUC)} under the probability $\dbP$.

               Subsequent analysis begins with an introduction to the following assumption. For each pair $(i_0,j_0)\in\cI\ts\cI$ and $\th\in \Th$,  let the maps
              \begin{align}\nonumber
                 \begin{aligned}
                 &   b_\th:[0,T] \ts \dbR \ts V \ts \cI  \ra \dbR,\q
                     (\si_\th,\bar{\si}_\th):[0,T]\ts \dbR \ts V  \ts \cI  \ra    \dbR,\\
                 &   \b_{\th,i_0j_0}: [0,T] \ts \dbR \ts V \ts \cI  \ra  \dbR, \q \F_\th: \dbR \ts \cI \ra\dbR,\\
                 &   f_\th:[0,T]\ts \dbR \ts \dbR\ts \dbR\ts \dbR \ts \dbR \ts V  \ts \cI \ra \dbR, \q
                     h_\th: [0,T]\ts \dbR \ts \cI \ra \dbR
                 \end{aligned}
              \end{align}
              satisfy
 \begin{assumption}\label{assumption3} \rm
 	\begin{itemize}
 		\item [(i)]
                 There exists a  positive constant $L$ such that, for $t\in[0,T]$, $x,x'\in \dbR$, $y,y',k,k'\in \dbR,
                         z,z',\bar{z},\bar{z}' \in \dbR,\ v,v'\in V,\ i\in \cI,\ \psi_\th=b_\th,\si_\th,$
                \begin{align}\nonumber
                   \begin{aligned}
                    &   |\psi_\th(t,x,v,i)-\psi_\th(t,x',v',i)|+|\bar{\si}_\th(t,x,v,i)-\bar{\si}_\th(t,x',v',i)|\leq L(|x-x'|+|v-v'|),\\
                    &   |\Phi_\th(x,i)-\Phi_\th(x',i)|+|f_\th(t,x,y,z,\bar{z},k,v,i)-f_\th(t,x',y',z',\bar{z}',k',v',i)|\\
                       \leq& L\(  (1+|x|+|x'|+|v|+|v'|)  (|x-x'|+|v-v'|)+|y-y'|+|z-z'|+|\bar{z}-\bar{z}'|+|k-k'|  \),\\
                    &   |\psi_\th(t,0,0,i)|+|f_\th(t,0,0,0,0,0,i)|+|\Phi_\th(0,i)|\leq L,\\
                    &\  |h_\th(t,x,i)-h_\th(t,x',i)|\leq L|x-x'|, \q |h_\th(t,x,i)|+|\bar{\si}_\th(t,x,v,i)|\leq L.
                   \end{aligned}
                \end{align}
             \item [(ii)]  There exists a positive constant $L$ such that, for $t\in[0,T],\ x,x'\in \dbR,\ v,v'\in V,\ i \in \cI$,
                \begin{align}\nonumber
                  \begin{aligned}
                    |\b_{\th,i_0j_0}(t,x,v,i)-\b_{\th,i_0j_0}(t,x',v',i)|\leq L(|x-x'|+|v-v'|)\q\ \hbox{and}\q\
                    |\b_{\th,i_0j_0}(t,0,0,i)|\leq L.
                  \end{aligned}
                \end{align}
            \item[(iii)]  The terms $b_\th,\si_\th,\bar{\si}_\th, \b_{\th,i_0j_0}, f_\th,\Phi_\th,h_\th$ are continuously differentiable in
                        $(x,y,z,\bar{z},k)$, and there exists a positive constant $L$ such that for $t\in[0,T]$, $x,x'\in \dbR,$
                        $y,y',k,k'\in \dbR,\ z,z',\bar{z},\bar{z}' \in \dbR,\ v,v'\in V,\ i\in \cI$,
                \begin{align}\nonumber
                  \begin{aligned}
                  &  |\phi_\th(t,x,v,i)-\phi_\th(t,x',v',i)| \leq L(|x-x'|+|v-v'|),\\
                  &  |\varphi_\th (t,x,y,z,\bar{z},k,v,i)-\varphi_\th(t,x',y',z',\bar{z}',k',v',i)|\\
                     \leq & L(|x-x'|+|y-y'|+|z-z'|+|\bar{z}-\bar{z}'|+|k-k'|+|v-v'|),
                  \end{aligned}
                \end{align}
                where $\th \in \Th, i\in \cI$,  $\phi_\th$ and $\varphi_\th$ denote the derivatives of $b_\th, \si_\th, \bar{\si}_\th, \b_{\th,i_0j_0},\Phi_\th,h_\th$ and $f_\th$ with respect to $(x,v)$ and  $(x,y,z,\bar{z},$ $k,v)$, respectively.
              \item [(iv)]   There exists a positive constant $L$ such that, for $t\in[0,T]$, $x\in \dbR,\ y,k\in \dbR,\ z,\bar{z}\in \dbR,$
                       $v\in V,\ i\in \cI,\ \th,\th'\in \Th$,
                \begin{align*}
                  & |\psi_\th(t,x,y,z,\bar{z},k,v,i)- \psi_{\th'}(t,x,y,z,\bar{z},k,v,i)|\leq L d(\th, \th'),\\
                  &{ |\Phi_\th(x,i)- \Phi_{\th'}(x,i)|+|\pa_x\Phi_\th(x,i)- \pa_x\Phi_{\th'}(x,i)|\leq L d(\th, \th'),}
                \end{align*}
                where $\psi_\th$ denotes $b_\th, \si_\th, \bar{\si}_\th, \b_{\th,i_0j_0},f_\th,h_\th$ and their derivatives in $(x,y,z,\bar{z},k,v)$.
              \item [(v)]   $\cQ$ is a weakly compact and convex set of probability measures on $(\Th, \sB(\Th))$.
 	\end{itemize}
\end{assumption}

               Now, let us retrospect Eq. (\ref{4.1}) and  Eq. (\ref{4.2}). Notice that
               the forward equation in (\ref{4.1}) involves the random noise $\overline{W}^v_\th(\cd)$,
                to undo the effect caused by   $\overline{W}^v_\th(\cd)$, we obtain by inserting (\ref{4.2}) into (\ref{4.1})
                \begin{align} \label{4.3}
                  \left  \{\begin{aligned}
                dX^v_\th(t) & =[b_\th(t,X^v_\th(t),v(t),\a(t-))-\bar{\si}_\th(t,X^v_\th(t),v(t),\a(t-))h_\th(t,X^v_\th(t),\a(t-))]dt\\
                            &\q +\si_\th(t,X^v_\th(t),v(t),\a(t-))dW(t)+\bar{\si}_\th(t,X^v_\th(t),v(t),\a(t-)) dG(t) \\
                            &\q +\b_\th(t,X^v_\th(t-),v(t),\a(t-))\bullet dM(t),\q t\in[0,T],\\
                dY^v_\th(t) &=  -f_\th(t,X^v_\th(t),Y^v_\th(t), Z^v_\th(t),\bar{Z}^v_\th(t),\sum_{i,j=1,i\neq j}^IK^v_{\th,ij}(t)\l_{ij}
                                 \mathbf{1}_{\{\a(t-)=i\}},v(t),\a(t-))dt\\
                            &\q  +Z^v_\th(t)dW(t)+\bar{Z}^v_\th(t)dG(t)+K^v_\th (t)\bullet dM(t),\ t\in[0,T],\\
                X^v_\th(0)  &=x,\q Y^v_\th(T)=\F_\th(X^v_\th(T), \a(T)).
                  \end{aligned}  \right.
                \end{align}
                  Due to $(W(\cd), G(\cd))$ being a 2-dimensional standard Brownian motion under the probability $\dbP$, one can get
                  the well-posedness of Eq. (\ref{4.3}) under the probability $\dbP$, see \autoref{le 4.1}  below.

\begin{lemma}\label{le 4.1}\rm
                Under $\mathrm{(i)}$ and $\mathrm{(ii)}$ of \autoref{assumption3}, for each $v(\cd)\in \cV_{ad}$, Eq. (\ref{4.3}) possesses a unique solution $(X^v_\th(\cd),Y^v_\th(\cd),Z^v_\th(\cd),\bar{Z}^v_\th(\cd) ,K^v_\th(\cd)) \in \cS_{\dbF}^{4}(0,T;\dbR)\ts \cS^2_{\dbF}(0,T;\dbR) \ts \cH^{2,1}_{\dbF}(0,T;\dbR)\ts \cH^{2,1}_{\dbF}(0,T;\dbR)\ts {\cK}_{\dbF}^{2,1}(0,T;\dbR).$ Furthermore, there exists a constant $C>0$ depending on $L,T,\sum\limits_{i,j=1,i\neq j}^I\l_{ij}$ such that
                \begin{align}\nonumber
                  \begin{aligned}
                  &   \dbE\bigg[\sup_{0\leq t\leq T}|X^v_\th(t)|^4+\sup_{0\leq t\leq T}|Y^v_\th(t)|^{2}
                      + \int_0^T (|Z^v_\th(t)|^2+|\cl{Z}^v_\th(t)|^2)dt\\
                     +&\int_0^T\sum_{i,j=1,i\neq j}^I|K^v_{\th,ij}(t)|^2\l_{ij}\mathbf{1}_{\{\a(t-)=i\}}dt\bigg]
                   \leq  C\bigg(|x|^4+\dbE \[ \int_0^T|v(t)|^4dt \] \bigg).
                  \end{aligned}
                \end{align}
\end{lemma}

              \begin{proof}
              According to   \autoref{th 3.2} and \autoref{th 6.1}, for each $v(\cd)\in \cV_{ad}$,  Eq. (\ref{4.3}) possesses a unique solution
              $(X_\th^v(\cd),Y_\th^v(\cd),Z_\th^v(\cd),\bar{Z}_\th^v(\cd),K_\th^v(\cd))\in \cS_{\dbF}^{4}(0,T;\dbR)\ts \cS^2_{\dbF}(0,T;\dbR) \ts \cH^{2,1}_{\dbF}(0,T;\dbR)\ts \cH^{2,1}_{\dbF}(0,T;\dbR)\ts {\cK}_{\dbF}^{2,1}(0,T;\dbR).$ Finally,  the above estimate  directly results from (\ref{6.112}) and (\ref{6.4-1}).
              \end{proof}
               The following lemma states the continuity of $(X^v_\th(\cd), Y^v_\th(\cd), Z^v_\th(\cd),\bar{Z}^v_\th(\cd),K^v_{\th}(\cd))$ with respect to $\th$.

\begin{lemma}\label{le 4.2}\rm
                Under $\mathrm{(i)}$, $\mathrm{(ii)},$ and $\mathrm{(iv)}$ of \autoref{assumption3},
                the maps $\th \mapsto X^v_\th(\cd), Y^v_\th(\cd), Z^v_\th(\cd),
                \bar{Z}^v_\th(\cd),K^v_{\th}(\cd)$ are continuous,
                i.e.,
                \begin{equation}\label{5.5}
                \begin{aligned}
                &  \lim_{\d\ra 0}\sup_{d(\th, \bar{\th})\leq\d }\dbE\bigg[\sup_{t\in[0,T]}
                    \(|X^v_\th(t)-X^v_{\bar{\th}}(t)|^4+|Y^v_\th(t)-Y^v_{\bar{\th}}(t)|^2\)+\int_0^T (|Z^v_\th(t)-Z^v_{\bar{\th}}(t)|^2\\
                & \qq\qq\qq\q+ |\bar{Z}^v_\th(t)-\bar{Z}^v_{\bar{\th}}(t)|^2+\sum_{i,j=1,i\neq j}^I |K^v_{\th,ij}(t)-K^v_{\bar{\th},ij}(t)|^2
                       \l_{ij}\mathbf{1}_{\{\a(t-)=i\}})dt\bigg]=0.
                \end{aligned}
                \end{equation}
\end{lemma}
\begin{proof}
  The proof is split into two steps.

             \emph{Step 1 ($X$-estimate).} Denote $\eta(s)=X_\th^v(s)-X_{\bar{\th}}^v(s)$ and, for $l=b,\si,\bar{\si},h,$
             \begin{equation*}
             \begin{aligned}
             A_l(s)&:=\int_0^1  \pa_xb_{\bar{\th}}(s,X^v_{\bar{\th}}(s)+\l(X^v_{\th}(s)-X^v_{\bar{\th}}(s)), v(s),\a(s-))d\l,\\
             B^{\th,\bar{\th}}_l(s)&:=l_\th(s,X^v_{\th}(s), v(s),\a(s-))-l_{\bar{\th}}(s,X^v_{\th}(s), v(s),\a(s-)), \\
             A_{\b,ij}(s)&:=\int_0^1  \pa_x\b_{\bar{\th},ij}(s,X^v_{\bar{\th}}(s)+\l(X^v_{\th}(s)-X^v_{\bar{\th}}(s)), v(s),\a(s-))d\l,\\
             B^{\th,\bar{\th}}_{\b,ij}(s)&:=\b_{\th,ij}(s,X^v_{\th}(s), v(s),\a(s-))-\b_{\bar{\th},ij}(s,X^v_{\th}(s), v(s),\a(s-)).
             \end{aligned}
             \end{equation*}
             Then we have
             \begin{equation*}
             \begin{aligned}
             \eta(t) & =\int_0^t\(\[A_b(s)+h^v_{\bar{\th}}(s)A_{\bar{\si}}(s)+\bar{\si}^v_{\th}(s)A_h(s)\]\eta(s)
                             +\[B^{\th,\bar{\th}}_b(s)+h^v_{\bar{\th}}(s)B^{\th,\bar{\th}}_{\bar{\si}}(s)
                             +\bar{\si}^v_{\th}(s)B^{\th,\bar{\th}}_h(s)\]\)ds\\
                   & \q +\int_0^t \(A_{\si}(s)\eta(s)+B^{\th,\bar{\th}}_{\si}(s)\)dW(s)+\int_0^t \(A_{\bar{\si}}(s)\eta(s)
                      +B^{\th,\bar{\th}}_{\bar{\si}}(s)\)dG(s) \\
                   &  \q + \int_0^t \(\sum_{i,j=1,i\neq j}^IA_{\b,ij}(s) \eta(s)\l_{ij} \mathbf{1}_{\{\a(s-)=i\}}
                          +\sum_{i,j=1,i\neq j}^IB^{\th,\bar{\th}}_{\b,ij}(s) \l_{ij} \mathbf{1}_{\{\a(s-)=i\}}\) dM_{ij}(s).
             \end{aligned}
             \end{equation*}
             It follows from \autoref{th 3.2} that
             \begin{equation*}
             \begin{aligned}
             \dbE\[\sup_{t\in[0,T]}|\eta(t)|^4\]
             &    \leq C\dbE\bigg[\(\int_0^T|B^{\th,\bar{\th}}_b(s)+h^v_{\bar{\th}}(s)B^{\th,\bar{\th}}_{\bar{\si}}(s)
             +\bar{\si}^v_{\th}(s)B^{\th,\bar{\th}}_h(s)|ds\)^4+\(\int_0^T|B^{\th,\bar{\th}}_{\si}(s)|^2ds\)^2\\
             &     \hskip 1.3cm+\(\int_0^T|B^{\th,\bar{\th}}_{\bar{\si}}(s)|^2ds\)^2+\int_0^T\sum_{i,j=1,i\neq j}^I|B^{\th,\bar{\th}}_{\b,ij}(s)|^4\l_{ij} \mathbf{1}_{\{\a(s-)=i\}}ds\bigg].
              \end{aligned}
             \end{equation*}
             Since $|B^{\th,\bar{\th}}_b|+|B^{\th,\bar{\th}}_{\si}|+|B^{\th,\bar{\th}}_{\bar{\si}}|+|B^{\th,\bar{\th}}_{\b,ij}|
             +|B^{\th,\bar{\th}}_h|\leq Ld(\th,\bar{\th})$ and $h^v_{\bar{\th}},\bar{\si}^v_{\th}$ are bounded by $L$, we have
             \begin{equation*}
             \begin{aligned}
             \dbE\[\sup_{t\in[0,T]}|\eta(t)|^4\]\leq Ld(\th,\bar{\th})^4.
              \end{aligned}
             \end{equation*}

             \emph{Step 2 ($Y$-estimate).} Denote $\xi(s):=Y_\th^v(s)-Y_{\bar{\th}}^v(s),\ \g(s):=Z_\th^v(s)-Z_{\bar{\th}}^v(s),\ \bar{\g}(s):=\bar{Z}_\th^v(s)-\bar{Z}_{\bar{\th}}^v(s),\z_{ij}(s):=\b^v_{\th,ij}(s)-\b^v_{\bar{\th},ij}(s)$ and
             denote, for $l=x,y,z,\bar{z},k$,
             \begin{equation*}
             \begin{aligned}
             C_\Phi(T)&=\int_0^1\pa_x\Phi_{\bar{\th}}(X^v_{\bar{\th}}(T)+\l(X_\th^v(T)-X^v_{\bar{\th}}(T)),\alpha(T))d\l,\\
             D_\Phi^{\th,\bar{\th}}(T)&=\Phi_\th(X^v_\th(T),\a(T))-\Phi_{\bar{\th}}(X^v_\th(T),\a(T)),\\
             \Pi^v_\th(s)&=(X^v_\th(s),Y^v_\th(s),Z^v_\th(s),\bar{Z}^v_\th(s),\sum_{i,j=1,i\neq j}^IK^v_{\th,ij}(s)
             \l_{ij} \mathbf{1}_{\{\a(s-)=i\}}   ),\\
             C_l(s)&=\int_0^1\pa_lf_{\bar{\th}}(s,\Pi^v_{\bar{\th}}(s)+\l(\Pi^v_{\th}(s)-\Pi^v_{\bar{\th}}(s)),v(s),\a(s-))ds,\\
             D_f^{\th,\bar{\th}}(s)&=f_\th(s,\Pi^v_\th(s), v(s),\a(s-))-f_{\bar{\th}}(s,\Pi^v_\th(s), v(s),\a(s-)).
              \end{aligned}
             \end{equation*}
             Consequently,  we have
             \begin{equation*}
             \begin{aligned}
             \xi(t) &  =C_\Phi(T)\eta(T)+D_\Phi^{\th,\bar{\th}}(T)+\int_t^T (C_x(s)\eta(s)+C_y(s)\xi(s)+C_z(s)\g(s)+C_{\bar{z}}(s)\bar{\g}(s)\\
                   &  \q +C_k(s)\sum_{i,j=1,i\neq j}^I\z_{ij}(s)\l_{ij} \mathbf{1}_{\{\a(s-)=i\}})ds-\int_t^T\g(s)dW(s)-\int_t^T\bar{\g}(s)dG(s)
                     -\int_t^T \z(s)\bullet dM(s).
             \end{aligned}
             \end{equation*}
             Thanks to \autoref{th 6.1}, one obtains
             \begin{equation*}
             \begin{aligned}
              &\dbE\bigg[\sup_{t\in[0,T]}|\xi(t)|^2+\int_0^T(|\g(t)|^2+|\bar{\g}(t)|^2
                  +\sum_{i,j=1,i\neq j}^I|\z_{ij}(t)|^2\l_{ij}\mathbf{1}_{\{\a(t-)=i\}})dt\bigg] \\
               \leq & C\dbE\bigg[|C_\Phi(T)\eta(T)+D_\Phi^{\th,\bar{\th}}(T)|^2+\int_0^T |C_x(t)\eta(t)|^2dt \bigg]\leq Ld(\th,\bar{\th})^2.
              \end{aligned}
             \end{equation*}
             This completes the proof.
\end{proof}

  Next, we make some  analysis on the process $R_\th^v(\cd)$ (see (\ref{4.7-1})) and on its inverse process  $(R_\th^v(\cd))^{-1}$,
  which are used in Bayes' formula (see, for example, (\ref{4.10}), (\ref{4.25-1})).
From It\^{o}'s formula  and the boundness of $h_\theta(\cd)$,
   one knows that
$R^v_\th(\cd)$ is an $(\dbF, \dbP)$-martingale  and

             \begin{align}\label{4.7-2}
             \begin{aligned}
               \sup_{\th\in\Th}\dbE\[\sup_{0\leq t \leq T}|R_\th^v(t)|^l\]<\infty,\ \forall\  l>1.
             \end{aligned}
             \end{align}
             Note that,  from (\ref{4.7-1}),  the inverse process $(R^v_\theta(\cdot))^{-1}$ satisfies
             \begin{equation*}
             \left\{
               \begin{aligned}
                 d(R^v_\theta(t))^{-1} & =(R^v_\theta(t))^{-1}h^2_\theta(t,X^v_\theta(t), \alpha(t-))dt
                                         +  (R^v_\theta(t))^{-1}h_\theta(t,X^v_\theta(t), \alpha(t-))dG(t),\  t\in[0,T],\\
                        (R^v_\theta(0))^{-1} & =1.
               \end{aligned}
               \right.
             \end{equation*}
             From the boundness of $h_\th(\cd)$, it follows
             \begin{equation*}\label{5.8}
             \sup_{\th\in\Th}\dbE\[\sup_{0\leq t \leq T}|(R_\th^v(t))^{-1}|^l\]<\infty,\ \forall\  l>1.
             \end{equation*}

           Similar to \autoref{le 4.2},  one can get  the continuity of the process $R^v_\th(\cd)$ with respect to $\th$.

   \begin{lemma}\label{le 5.3-0}\rm
                Under  $\mathrm{(i)}$, $\mathrm{(ii)}$ and $\mathrm{(iv)}$ of \autoref{assumption3}, it follows,  for $v(\cd)\in \cV_{ad}$,
                \begin{equation*}
                \begin{aligned}
                &  \lim_{\d\ra 0}\sup_{d(\th, \bar{\th})\leq\d }\dbE\bigg[\sup_{t\in[0,T]}|R^v_\th(t)-R^v_{\bar{\th}}(t)|^2\bigg]=0.
                \end{aligned}
                \end{equation*}
\end{lemma}

     With the process  $R^v_\th(\cd)$ and its inverse process  $(R_\th^v(\cd))^{-1}$ in hand, we can rewrite the cost functional
       (\ref{4.2222})  with the probability $\dbP$.  Before that, let us first introduce the assumption on $(\Psi_\th, \L_\th).$
 Let the maps $\Psi_\th:\dbR\ra \dbR$ and $\L_\th:\dbR\ra \dbR$ satisfy
\begin{assumption}\label{assumption4}\rm
	\begin{itemize}
		\item [(i)]
	         $\Psi_\theta(\cd)$ and $\L_\th(\cd)$ are continuously differentiable with respect to their respect variables.

             \item [(ii)]  There exists a constant $L>0$ such that for $x,x'\in \dbR$,\ $y,y'\in \dbR$,\ $\th\in \Th$,
              \begin{align}\nonumber
               \begin{aligned}
              & |\L_\th(y)-\L_\th(y')|\leq L|y-y'|,     \q |\Psi_\th(0)|+|\L_\th(0)|\leq L,\\
              &  |\Psi_\th(x)-\Psi_\th(x')|\leq L(1+|x|+|x'|)|x-x'|.
               \end{aligned}
              \end{align}
             \item [(iii)]  There exists a constant $L>0$ such that for $x\in \dbR$, $y\in\dbR$, $\th,\th'\in \Th$,
               \begin{align}\nonumber
                |\Psi_\th(x)-\Psi_{\th'}(x)|+| \L_\th(y)-\L_{\th'}(y)|\leq Ld(\th,\th').
              \end{align}
              \item [(iv)] There exists a constant $L>0$ such that for $x,x'\in \dbR$, $y,y'\in \dbR$, $\th\in \Th$,
               \begin{align}\nonumber
               |\pa_x\Psi_\th(x)-\pa_x\Psi_\th(x')|\leq L|x-x'|,\q  |\pa_y\L_\th(y)-\pa_y\L_\th(y')|\leq L|y-y'|.
               \end{align}
            \end{itemize}
\end{assumption}

               Thanks to Bayes' formula, the cost functional  (\ref{4.2222})  can be written as
               \begin{align}\label{4.10}
               J(v(\cd))=\sup_{Q\in \cQ}\int_{\Th}  \dbE[R^v_\th(T)\Psi_\th(X_\th^v(T))+\L_\th(Y^v_\th(0))] Q(d\th),
               \end{align}
               where $R^v_\th(\cd)$  is introduced in (\ref{4.7-1}).
              Note that the expectation $\dbE$ (or, say, the probability $\dbP$) does not depend on the control $v$ and the
              parameter $\th$, which is very crucial for later analysis.
                \textbf{Problem} (\textbf{PRUC})  is equivalent to minimize (\ref{4.10}) over $\mathcal{V}_{ad}$ subject to
                 (\ref{4.3}) and (\ref{4.7-1}).

\section{Maximum Principle and Sufficient Condition}\label{Sec4}
               This section is devoted to  the necessary maximum  principle and sufficient condition.
\subsection{Variational Equations}

               Let $\h v(\cd)$ be an optimal control and for each $\th\in \Th$,  let $(\h X_\th(\cd),\h Y_\th(\cd),\h Z_\th(\cd),\widehat{\bar{Z}}_\th(\cd),\h K_\th(\cd))$  and $\h R_\th(\cd)$  be the solutions  to Eq. (\ref{4.3}) and (\ref{4.7-1}) with $\h v(\cd)$, respectively. From the convexity of $\cV_{ad}$, for each $v(\cd)\in\cV_{ad}$ and $\e>0$, $v^\e(\cd):=\h v(\cd)+\e(v(\cd)-\h v(\cd))\in\cV_{ad}$. By $(X^\e_\th(\cd),Y^\e_\th(\cd),Z^\e_\th(\cd), \bar{Z}^\e_\th(\cd), K^\e_\th(\cd))$ and $R^\e_\th(\cd)$ we denote the solutions to Eqs. (\ref{4.3})  and (\ref{4.7-1}) with $v^\e(\cd)$, for each $\th\in \Th$.
              Denote, for  $\psi_\th=b_\th,\si_\th,\bar{\si}_\th,\b_\th,\Phi_\th, h_\th$, $\ell=x,y,z,\bar{z},k,v$,
              \begin{equation*}\nonumber
               \begin{aligned}
                \psi_\th(t) & :=\psi_\th(t,\h{X}_\th(t),\h v(t),\a(t-)), \qq\  \psi^\e_\th(t):=\psi_\th(t,X^\e_\th(t),v^\e(t),\a(t-)), \\
               \pa_x\psi_\th(t) &  :=\pa_x\psi_\th(t,\h{X}_\th(t),\h{v}(t),\a(t-)),\
                                     \pa_v\psi_\th(t):=\pa_v\psi_\th(t,\h{X}_\th(t),\h v(t),\a(t-)),\\
               f^\e_\th(t) & :=f_\th(t,X^\e_\th(t),Y^\e_\th (t),Z^\e_\th(t),\bar{Z}^\e_\th(t)
                                      \sum_{i,j=1,i\neq j}^IK^\e_{\th,ij}(t)\l_{ij}\mathbf{1}_{\{\a(t-)=i\}},v^\e(t),\a(t-)),\\
              f_\th(t) &  :=f_\th(t,\h{X}_\th(t),\h{Y}_\th(t),\h{Z}_\th(t),\h{\bar{Z}}_\th(t), \sum_{i,j=1,i\neq j}^I\h{K}_{\th,ij}(t)\l_{ij}
                           \mathbf{1}_{\{\a(t-)=i\}},\h{v}(t),\a(t-)),\\
              \pa_\ell f_\th(t) & :=\pa_\ell f_\th(t,\h{X}_\th(t),\h{Y}_\th(t),\h{Z}_\th(t),\h{\bar{Z}}_\th(t),
                                  \sum_{i,j=1,i\neq j}^I\h{K}_{\th,ij}(t)\l_{ij}\mathbf{1}_{\{\a(t-)=i\}}, \h{v}(t),\a(t-)).
%
              \end{aligned}
              \end{equation*}

               Consider the following variational equations:
               \begin{equation}\label{11.1}
               \left\{
               \begin{aligned}
               dR^{1}_\th(t)&=\( R^{1}_\th(t)h_\th(t) +\h{R}_\th(t)\pa_xh_\th(t)X^{1}_\theta(t)\)dG(t),\  t\in[0,T],  \\
               R^{1}_\th(0)&=0,
               \end{aligned}
               \right.
               \end{equation}
              and
              \begin{equation}\label{11.2}
               \left\{   \begin{aligned}
                 dX^{1}_\th(t) & =\bigg\{\[\pa_xb_\th(t)-\pa_x\bar{\si}_\th(t)h_\th(t)-\bar{\si}_\th(t)\pa_xh_\th(t)\]X^{1}_\th(t)\\
                               & \q\q + \[\pa_vb_\th(t) -\pa_v\bar{\si}_\th(t)h_\th(t)\](v(t)-\h{v}(t))\bigg\}dt\\
                               & \q\q  +\[\pa_x\si_\th(t)X^{1}_\th(t)+\pa_v\si_\th(t)(v(t)-\h{v}(t))\]dW(t)\\
                               & \q\q  +\[\pa_x\bar{\si}_\th(t)X^{1}_\th(t)+\pa_v\bar{\si}_\th(t)(v(t)-\h{v}(t))\]dG(t)\\
                               & \q\q  +\[\pa_x\b_\th(t)X^{1}_\th(t)+\pa_v\b_\th(t)(v(t)-\h{v}(t))\]\bullet dM(t),\q t\in[0,T],\\
                   X^{1}_\th(0)&  =0.
              \end{aligned}   \right.
               \end{equation}
              According to \autoref{th 3.2}, under \autoref{assumption3}, the above variational equations admit unique solutions $X^{1}_\th(\cd)\in {\cS}_{\dbF}^{4}(0,T;\dbR)$ and
              $R^{1}_\theta(\cd)\in {\cS}_{\dbF}^{2}(0,T;\dbR)$. Moreover, it follows
              \begin{equation}\label{4.12-1}
              \left\{ \begin{aligned}
               &  \dbE\bigg[\sup_{0\leq t\leq T}|X^{1}_\th(t) |^4\bigg]
                  \leq C\dbE\bigg[\int_0^T(|v(t)|^4+|\h{v}(t)|^4)dt\bigg],\\
               &  \dbE\bigg[\sup_{0\leq t\leq T}|R^{1}_\th (t)|^2 \bigg]
                   \leq C\bigg\{\dbE\[\int_0^T(|v(t)|^4+|\h{v}(t)|^4)dt\]\bigg\}^{\frac{1}{2}}.
              \end{aligned}  \right.
              \end{equation}

    For simplicity presentation, denote
              $$
                \d^{\e}X_\th(t):=\frac{1}{\e}\big(X^\e_\th(t)-\h{X}_\th(t)\big)-X^{1}_\th(t),\q \d^{\e}R_\th(t):=\frac{1}{\e}\big(R^\e_\th(t)-\h{R}_\th(t)\big)-R^{1}_\th(t).
               $$
\begin{lemma}\label{le 5.1}\rm
             Let $\mathrm{(i)}$-$\mathrm{(iii)}$ of \autoref{assumption3} hold true, then we have, for $\th\in \Th$,
          \begin{equation*}
                      \begin{aligned}
               \mathrm{ (i)}&\  \dbE\bigg[\sup\limits_{0\leq t\leq T}|\d^{\e}X_\th(t) |^4 \bigg]
               \leq C\dbE\bigg[\int_0^T(|v(t)|^4+|\h{v}(t)|^4)dt\bigg].\\
               \mathrm{(ii)}&\  \lim\limits_{\e\ra0}\sup\limits_{\th\in\Th}\dbE\bigg[\sup\limits_{0\leq t\leq T}|\d^{\e}X_\th(t)|^4\bigg]=0.\\
              \mathrm{ (iii)}&\  \dbE\bigg[\sup\limits_{0\leq t\leq T}|\d^{\e}R_\th(t)|^2\bigg]
               \leq C\dbE\bigg[\int_0^T(|v(t)|^4 +|\h{v}(t)|^4)dt\bigg]^\frac{1}{2}.\\
               \mathrm{(iv)}&\  \lim\limits_{\e\ra0}\sup\limits_{\th\in\Th} \dbE\bigg[\sup\limits_{0\leq t\leq T}|\d^{\e}R_\th (t) |^2 \bigg]=0.
          \end{aligned}
           \end{equation*}
\end{lemma}

\begin{proof}
According to the definition  of  $\d^{\e}X_\th$, we have
                \begin{align}\label{4.12}
                \left\{ \begin{aligned}
                d\d^{\e}X_\th(t) & =\bigg\{\frac{1}{\e}\([b^\e_\th(t)-\bar{\si}^\e_\th(t)h^\e_\th(t)]
                                     - [b_\th(t)-\bar{\si}_\th(t)h_\th(t)] \)\\
                                & \qq  -\(\big[\pa_xb_\th(t)-\pa_x\bar{\si}_\th(t)h_\th(t)-\bar{\si}_\th(t)\pa_xh_\th(t)\big]X^{1}_\th(t)\\
                                &  \qq  +\big[\pa_vb_\th(t)-\pa_v\bar{\si}_\th(t)h_\th(t)\big](v(t)-\h{v}(t))\)\bigg\}dt\\
                                &\qq+\bigg\{\frac{1}{\e}\(\si^\e_\th(t)-\si_\th(t)\)-\[\pa_x\si_\th(t)X^{1}_\th(t)
                                     +\pa_v\si_\th(t)(v(t)-\h{v}(t))\]\bigg\}dW(t)\\
                                &\qq +\bigg\{\frac{1}{\e}\(\bar{\si}^\e_\th(t)-\bar{\si}_\th(t)\)-
                                     \[\pa_x\bar{\si}_\th(t)X^{1}_\th(t)+\pa_v\bar{\si}_\th(t)(v(t)-\h{v}(t))\]\bigg\}dG(t)\\
                                &\qq  +\bigg\{\frac{1}{\e}\(\b^\e_\th(t)-\b_\th(t)\)- \[\pa_x\b_\th(t)X^{1}_\th(t)
                                     +\pa_v\b_\th(t)(v(t)-\h{v}(t))\]\bigg\}\bullet dM(t),\ t\in[0,T],\\
                 \d^\e X_\th(0) & =0.
                \end{aligned}   \right.
                \end{align}
                For simplicity, we denote, for  $\psi_\th=b_\th,\si_\th,\bar{\si}_\th,\b_\th,h_\th$,
                \begin{equation}\nonumber
                \begin{aligned}
                \pa_x\psi^\e_\th(t): & =\int_0^1\pa_x\psi_\th(t,\h{X}_\th(t)+\lambda\e(X^1_\th(t)+\d^\e X_\th(t)),
                                      \h{v}(t)+\l\e (v(t)-\h{v}(t)), \a(t-))d\l, \\
                A_{1,\th}^\e(t): & =\[(\pa_xb^\e_\th(t)-\pa_xb_\th(t))+(\pa_x\bar{\si}^\e_\th(t)-\pa_x\bar{\si}_\th(t))h_\th(t)
                                       +(\pa_xh^\e_\th(t)-\pa_xh_\th(t))\bar{\si}_\th(t)\\
                                 &\qq      +\pa_x h_\th^\e(t)(\bar{\si}_\th^\e (t)-\bar{\si}_\th(t))\Big] X^{1}_\th(t)\\
                                 & \qq +\[(\pa_vb^\e_\th(t)-\pa_vb_\th(t))+(\pa_v\bar{\si}^\e_\th(t)-\pa_v\bar{\si}_\th(t))
                                           h_\th(t)\](v(t)-\h{v}(t)),\\
                A_{2,\th}^\e(t): & =\[\pa_x\si^\e_\th(t)-\pa_x\si_\th(t)\] X^{1}_\th(t)
                                      +\[\pa_v\si^\e_\th(t)-\pa_v\si_\th(t)\](v(t)-\h{v}(t)),\\
                 A_{3,\th}^\e(t): & =\[\pa_x\bar{\si}^\e_\th(t)-\pa_x\bar{\si}_\th(t)\] X^{1}_\th(t)
                                    +\[\pa_v\bar{\si}^\e_\th(t)-\pa_v\bar{\si}_\th(t)\](v(t)-\h{v}(t)),\\
                 A_{4,\th,ij}^\e(t):& =\[\pa_x\b^\e_{\th,ij}(t)-\pa_x\b_{\th,ij}(t)\] X^{1}_\th(t)
                                    +\[\pa_v\b^\e_{\th,ij}(t)-\pa_v\b_{\th,ij}(t)\](v(t)-\h{v}(t)).\\
                \end{aligned}
                \end{equation}
Then Eq. (\ref{4.12}) can be written as the following linear SDE
                \begin{equation*}
                \left\{  \begin{aligned}
                d\d^{\e}X_\th(t) & =\[\big(\pa_xb^\e_\th(t)+\pa_x\bar{\si}^\e_\th(t)h_\th(t)+\pa_x h^\e_\th(t)\bar{\si}_\th(t)\\
                                 & \qq + \pa_x h_\th^\e(t)(\bar{\si}_\th^\e(t)-\si_\th(t))\big)\d^{\e}X_\th(t)
                                     +A_{1,\th}^\e(t)\]dt\\
                                 &\q +\[\pa_x\si^\e_\th(t)\d^{\e}X_\th(t)+A_{2,\th}^\e(t)\]dW(t)
                                     +\[\pa_x\bar{\si}^\e_\th(t)\d^{\e}X_\th(t)+A_{3,\th}^\e(t)\Big]dG(t)\\
                                 &\q +\[\pa_x\b^\e_\th(t)\d^{\e}X_\th(t)+A_{4,\th}^\e(t)\]\bullet dM(t),\\
                 \d^{\e}X_\th(0) & =0.
                \end{aligned} \right.
                \end{equation*}
                From \autoref{th 3.2} and (\ref{4.12-1}), there exists a constant $C>0$ depending on $L,T,\sum\limits_{i,j=1,i\neq j}^I \lambda_{ij}$ such that
              \begin{equation*}
               \begin{aligned}
               \dbE\Big[\sup_{0\leq t\leq T}|\d^{\e}X_\th(t)|^4 \Big]
              & \leq C\dbE\bigg[\(\int_0^T|A_{1,\th}^\e(t)|dt\)^4+\(\int_0^T|A_{2,\th}^\e(t)|^2dt\)^2+\(\int_0^T|A_{3,\th}^\e(t)|^2dt\)^2 \\
              &  \q  +\int_0^T\sum_{i,j=1,i\neq j}^I|A_{4,\th,ij}^\e(t)|^4 \l_{ij}\mathbf{1}_{\{\a(t)=i\}}dt\bigg] \\
              & \leq C\dbE\bigg[\int_0^T\big(|A_{1,\th}^\e(t)|^4+|A_{2,\th}^\e(t)|^4+|A_{3,\th}^\e(t)|^4
              +\sum_{i,j=1,i\neq j}^I|A_{4,\th,ij}^\e(t)|^4\l_{ij} \big)dt\bigg] \\
              &  \leq C\dbE\bigg[\int_0^T(|v(t)|^4+|\h{v}(t)|^4) dt\bigg].
              \end{aligned}
              \end{equation*}
             As  for  item  $\mathrm{(ii)}$, it follows from  dominated convergence theorem that
             \begin{align}\nonumber
             \dbE\bigg[\int_0^T\big(|A_{1,\th}^\e(t)|^4+|A_{2,\th}^\e(t)|^4+|A_{3,\th}^\e(t)|^4
              +\sum_{i,j=1,i\neq j}^I|A_{4,\th,ij}^\e(t)|^4\l_{ij} \big)dt\bigg]\ra 0, \ \text{as}\ \e\ra0.
             \end{align}
             Thereby, we obtain  $\mathrm{(ii)}$.
 We now focus on  item $\mathrm{(iii)}$. From the definition of $\delta^{\varepsilon}R_\theta$, we know
             \begin{equation}\label{4.17}
             \left\{ \begin{aligned}
                d\d^{\e}R_\th(t) &  =\bigg\{\frac{1}{\e}\(R^\e_\th(t)h^\e_\th(t)-\h{R}_\th(t)h_\th(t)\)- \(R^{1}_\th(t)h_\th(t)
                                     +\h{R}_\th(t)\pa_xh_\th(t)X^{1}_\theta(t)\)\bigg\}dG(t),\\
                  \d^{\e}R_\th(0) &  =0.
              \end{aligned}  \right.
             \end{equation}
             Denote
             \begin{equation*}
              \begin{aligned}
              A_{5,\th}^\e(t): & =(\pa _xh^\e_\th(t)-\pa_xh_\th(t))\h{R}_\th(t) X^{1}_\th(t)-\pa_xh^\e_\th(t)\h{R}_\th(t)\d^\e X_\th(t)
                                 +R_\th^{1}(t)(h^\e_\theta(t)-h_\th(t)).
              \end{aligned}
             \end{equation*}
              (\ref{4.17}) can be rewitten  as
            \begin{equation*}
               d\d^{\e}R_\theta(t)=\bigg\{h^\e_\th(t)\d^{\e}R_\th(t)+A_{5,\th}^\e(t)\bigg\}dG(t),\q \d^{\e}R_\th(0)=0.
            \end{equation*}
            Thanks to \autoref{th 3.2}, (\ref{4.7-2}), (\ref{4.12-1}), and item $\mathrm{(i)}$, it follows from H\"{o}lder inequality
            \begin{equation*}
            \begin{aligned}
            & \dbE\bigg[\sup_{0\leq t\leq T}|\d^{\e}R_\th(t)|^2 \bigg]\leq C\dbE\bigg[\int_0^T|A_{5,\th}^\e(t)|^2dt\bigg]\\
             \leq &C\dbE\bigg[\int_0^T(|\h{R}_\th(t) X^{1}_\th(t)|^2+|\h{R}_\th(t)\d^\e X_\th(t)|^2+|R_\th^{1}(t)|^2)dt\bigg]\\
             \leq &C\bigg\{\dbE \bigg[\int_0^T(|v(t)|^4+|\bar{v}(t)|^4)dt\bigg]\bigg\}^\frac{1}{2}.
            \end{aligned}
           \end{equation*}
           Finally,  item $\mathrm{(iv)}$ is a direct result of   item $\mathrm{(ii)}$ and dominated convergence theorem.
\end{proof}

           Next, let us investigate the variational BSDE: for each $\th\in  \Th$,
           \begin{equation}\label{11.3}
           \left\{  \begin{aligned}
           d Y_\th^{1}(t) & =-\bigg\{ \pa_xf_\th(t)X_\th^{1}(t)+ \pa_yf_\th(t)Y_\th^{1}(t)+ \pa_zf_\th(t)Z_\th^{1}(t)
                               + \pa_{\bar{z}}f_\th(t)\bar{Z}_\th^{1}(t)\\
                          & \q\q\q+ \pa_kf_\th(t)\sum_{i,j=1,i\neq j}^IK^{1}_{\theta,ij}(t)\l_{ij}\mathbf{1}_{\{\a(t-)=i\}}
                              +\langle\pa_v f_\th(t), v(t)-\h{v}(t)\rangle \bigg\}dt\\
                          & \q\q\q+Z_\th^{1}(t)dW(t)+\bar{Z}_\th^{1}(t)dG(t)+K_\th^{1}(t)\bullet dM(t),\  t\in[0,T],\\
            Y_\th^{1}(T) & =\pa_x\Phi_\th(T)X_\th^{1}(T).
            \end{aligned}  \right.
           \end{equation}
            From \autoref{assumption3}, $\pa_yf_\th,\pa_zf_\th,\pa_{\bar{z}}f_\th,\pa_kf_\th$ are uniformly bounded and $\pa_xf_\th, \pa_v f_\th,\pa_x\Phi_\th$ are bounded by $L(1+|x|+|v|)$. Thanks to \autoref{th 6.1}, the above equation admits
            a unique solution $(Y_\th^1(\cd),Z_\th^1(\cd),\bar{Z}_\th^1(\cd),K_\th^1(\cd))\in \cS^2_{\dbF}(0,T;\dbR) \ts \cH^{2,1}_{\dbF}(0,T;\dbR)\ts \cH^{2,1}_{\dbF}(0,T;\dbR)\ts {\cK}_{\dbF}^{2,1}(0,T;\dbR)$. Moreover, there exists a constant $C>0$ depending on $L,T,\sum\limits_{i,j=1,i\neq j}^I\l_{ij}$ such that
            \begin{equation}\label{4.21-1}
             \begin{aligned}
              &  \dbE\bigg[\sup_{t\in[0,T]}|Y_\th^{1}(t)|^2+\int_0^T(|Z_\th^{1}(t)|^2+|\bar{Z}_\th^{1}(t)|^2
                 +\sum_{i,j=1,i\neq j}^I|K^{1}_{\th,ij}(t)|^2\l_{ij}\mathbf{1}_{\{\a(t-)=i\}})dt\bigg]\\
                \leq &C\dbE \bigg[|x|^4+\int_0^T(|v(t)|^4+|\h{v}(t)|^4) dt\bigg].
            \end{aligned}
            \end{equation}
            In addition, $(X^1_\th(\cd),R^1_\th(\cd), Y^1_\th(\cd),Z^1_\th(\cd),\bar{Z}^1_\th(\cd),K^1_\th(\cd))$ are also continuous with respect to  $\th$, as shown below.
\begin{lemma}\label{le 6.3-0}\rm
             Under $\mathrm{(i)}$-$\mathrm{(iv)}$ of \autoref{assumption3}, it follows
            \begin{equation*}
            \begin{aligned}
            &  \lim_{\d\ra 0}\sup_{d(\th, \bar{\th})\leq\d }\dbE\bigg[\sup_{t\in[0,T]}\(|X^1_\th(t)-X^1_{\bar{\th}}(t)|^4
                    +|R^1_\th(t)-R^1_{\bar{\th}}(t)|^2+|Y^1_\th(t)-Y^1_{\bar{\th}}(t)|^2\)\\
            & \  +\int_0^T(|Z^1_\th(t)-Z^1_{\bar{\th}}(t)|^2+|\bar{Z}^1_\th(t)-\bar{Z}^1_{\bar{\th}}(t)|^2
            +\sum_{i,j=1,i\neq j}^I|K^1_{\th,ij}(t)-K^1_{\bar{\th},ij}(t)|^2\l_{ij}\mathbf{1}_{\{\a(t-)=i\}})dt\bigg]=0.
            \end{aligned}
            \end{equation*}
\end{lemma}

\begin{proof}
 The proof is split into three steps.

             \emph{Step 1 ($X$-estimate).}\ Denote $\eta^1=X_\th^1-X^1_{\bar{\th}},\xi^1=Y_\th^1-Y^1_{\bar{\th}},\g^1=Z_\th^1-Z^1_{\bar{\th}},
             \bar{\g}^1=\bar{Z}_\th^1-\bar{Z}^1_{\bar{\th}}, \z^1_{ij}=K^1_{\th,ij}- K^1_{\bar{\th},ij}$ and denote for $l=\si,\bar{\si},$
             \begin{equation*}
             \begin{aligned}
               E_b(t)&=\bigg[\(\pa_xb_\th(t)-\pa_x\bar{\si}_\th(t)h_\th(t)-\bar{\si}_\th(t)\pa_xh_\th(t)\)
                         -\(\pa_xb_{\bar{\th}}(t)-\pa_x\bar{\si}_{\bar{\th}}(t)h_{\bar{\th}}(t)
                         -\bar{\si}_{\bar{\th}}(t)\pa_xh_{\bar{\th}}(t)\)\bigg]X_\th^1(t) \\
             &  \qq +\bigg[(\pa_vb_\th(t) -\pa_v\bar{\si}_\th(t)h_\th(t))-(\pa_vb_{\bar{\th}}(t)
                -\pa_v\bar{\si}_{\bar{\th}}(t)h_{\bar{\th}}(t))\bigg](v(t)-\bar{v}(t)),\\
               E_l(t)&=\bigg[ \pa_x l_\th(t)-\pa_x l_{\bar{\th}}(t)  \bigg]X_\th^1(t)
                   +\bigg[ \pa_v l_\th(t)-\pa_v l_{\bar{\th}}(t)  \bigg](v(t)-\bar{v}(t)),\\
               E_{\b,ij}(t)&=\bigg[ \pa_x\b_{\th,ij}(t)-\pa_x\b_{\bar{\th},ij}(t)  \bigg]X_\th^1(t)
                   +\bigg[ \pa_v\b_{\th,ij}(t)-\pa_v\b_{\bar{\th},ij}(t)  \bigg](v(t)-\bar{v}(t)).
             \end{aligned}
             \end{equation*}
             According to \autoref{th 3.2},  we obtain
             \begin{equation*}
             \begin{aligned}
             & \dbE\[\sup_{t\in[0,T]}|X^1_\th(t)-X^1_{\bar{\th}}(t)|^4  \]
              \leq C\dbE\[\int_0^T\(|E_b(t)|^4+|E_\si(t)|^4+|E_{\bar{\si}}(t)|^4
                  +\sum_{i,j=1,i\neq j}^I|E_{\b,ij}(t)|^4\l_{ij}\mathbf{1}_{\{\a(t-)=i\}} \)dt\].
             \end{aligned}
             \end{equation*}
             Since $\pa_xb_\th, \pa_x\si_\th,\pa_x\bar{\si}_\th,\pa_x\b_{\th,ij},\pa_x h_\th, h_\th$ and $\pa_vb_\th,\pa_v\si_\th,\pa_v\bar{\si}_\th,\pa_v\b_{\th,ij}$
             are Lipschitz with respect to $(x,\th)$, as well as $h_\th,\bar{\si}_\th,\pa_x\bar{\si}_\th$ are bounded by the constant $L$,
             we have that $\dbE\[\sup\limits_{t\in[0,T]}|X^1_\th(t)-X^1_{\bar{\th}}(t)|^4  \]\leq Ld(\th,\bar{\th})^4.$

             \emph{Step 2 ($R$-estimate).}\ Denote
             $
             F_R(t)=R^1_{\bar{\th}}(t)(h_\th(t)-h_{\bar{\th}}(t))+\h{R}_\th(t)\pa_xh_\th(t)X_\th^1(t)
             -\h{R}_{\bar{\th}}(t)\pa_xh_{\bar{\th}}(t)X_{\bar{\th}}^1(t).
             $
            According to   \cite[Lemma 2.3]{Hu-Wang-20}, we obtain
            $
            \dbE\[\sup\limits_{t\in[0,T]}|R^1_\th(t)-R^1_{\bar{\th}}(t)|^2  \]\leq C\dbE\[\int_0^T|F_R(t)|^2dt\].
           $
           In addition, since
            \begin{equation*}
            \begin{aligned}
           & h_\th(t)-h_{\bar{\th}}(t)=h_\th(t,\bar{X}_\th(t),\a(t-))-h_{\bar{\th}}(t,\bar{X}_{\bar{\th}}(t),\a(t-) )\\
            =&h_\th(t,\bar{X}_\th(t),\a(t-))-h_{\bar{\th}}(t,\bar{X}_\th(t),\a(t-) )+
            h_{\bar{\th}}(t,\bar{X}_\th(t),\a(t-))-h_{\bar{\th}}(t,\bar{X}_{\bar{\th}}(t),\a(t-) ),
            \end{aligned}
            \end{equation*}
            and
            \begin{equation*}
            \begin{aligned}
            & \h{R}_\th(t)\pa_xh_\th(t)X_\th^1(t)-\h{R}_{\bar{\th}}(t)\pa_xh_{\bar{\th}}(t)X_{\bar{\th}}^1(t)  \\
            =&(\h{R}_\th(t)- \h{R}_{\bar{\th}}(t))\pa_xh_\th(t)X_\th^1(t)+\h{R}_{\bar{\th}}(t)(\pa_xh_\th(t)
               -\pa_xh_{\bar{\th}}(t)   )X_\th^1(t)
            +\h{R}_{\bar{\th}}(t)\pa_xh_{\bar{\th}}(t)(X_\th^1(t)-X_{\bar{\th}}^1(t)),
            \end{aligned}
            \end{equation*}
           it follows from the Lipschitz property of $h_\th, \pa_x h_\th$ with respect to $(x,\th)$, the boundness of $\pa_x h_\th$
           and   \autoref{le 4.1}, \autoref{le 4.2},
           \autoref{le 5.3-0}, (\ref{4.7-2}), (\ref{4.12-1}) that
           $
           \dbE\[\sup\limits_{t\in[0,T]}|R^1_\th(t)-R^1_{\bar{\th}}(t)|^2  \]\leq Cd(\th,\bar{\th})^2.
           $

            \emph{Step 3 ($Y$-estimate).} Denote for $l=x,y,z,\bar{z},k,$
            \begin{equation*}
            \begin{aligned}
            I^{1,1}&  =  \pa_x\Phi(\h{X}_{\bar{\th}}(T))\cl{\eta}(T),\  I^{1,2}=\(\pa_x\Phi_\th(\h{X}_\th(T))-\pa_x\Phi_{\bar{\th}}
                        (\h{X}_{\bar{\th}}(T))\)\h{X}_\th(T),\\
            \h{\Pi}_{\th}(t)& =(\h{X}_\th(t),\h{Y}_\th(t),\h{Z}_\th(t),\h{\bar{Z}}_\th(t),\sum_{i,j=1,i\neq j}^I\h{K}_{\th,ij}(t)\l_{ij}\mathbf{1}_{\{\a(t-)=i\}}),\\
            G_l(t) & =  \pa_lf_{\bar{\th}}(t,\h{\Pi}_{\th}(t), \h{v}(t),\a(t-)),\\
            H(t)   & =  \[\pa_vf_\th(t,\h{\Pi}_{\th}(t), \h{v}(t),\a(t-))-\pa_vf_{\bar{\th}}(t,\h{\Pi}_{\th}(t), \h{v}(t),\a(t-))\]
                        (v(t)-\bar{v}(t))\\
                   &\q +\[ \pa_xf_\th(t,\h{\Pi}_{\th}(t), \h{v}(t),\a(t-))-G_x(t)  \]X^1_\th(t)
                    +\[ \pa_yf_\th(t,\h{\Pi}_{\th}(t), \h{v}(t),\a(t-))-G_y(t)  \]Y^1_\th(t)\\
                   & \q+\[ \pa_zf_\th(t,\h{\Pi}_{\th}(t), \h{v}(t),\a(t-))-G_z(t)  \]Z^1_\th(t)
                    +\[ \pa_{\bar{z}}f_\th(t,\h{\Pi}_{\th}(t), \h{v}(t),\a(t-))-G_{\bar{z}}(t)  \]\bar{Z}^1_\th(t)\\
                   &\q + \sum_{i,j=1,i\neq j}^I\[\pa_kf_\th(t,\h{\Pi}_{\th}(t), \h{v}(t),\a(t-))-G_k(t)\]
                   K_{\th,ij}^1(t)\l_{ij}(t)\mathbf{1}_{\{\a(t-)=i\}}.
            \end{aligned}
            \end{equation*}
           Thanks to \autoref{th 6.1}, we have
           \begin{equation*}
            \begin{aligned}
            &    \dbE\[\sup_{t\in[0,T]}|\xi^1(t)|^2+\int_0^T(|\g^1(t)|^2+|\bar{\g}^1(t)|^2
                      +\sum_{i,j=1,i\neq j}^I|\zeta^1_{ij}(t)|^2\l_{ij}\mathbf{1}_{\{\a(t-)=i\}})dt   \]\\
               \leq& C\dbE\[|I^{1,1}|^2+|I^{1,2}|^2+\int_0^T(|G_x(t)\eta(t)|^2+|H(t)|^2)dt\].
            \end{aligned}
            \end{equation*}
           On the one hand,  \autoref{assumption3}   allows to show
           $
           \dbE\[|I^{1,1}|^2 +\int_0^T|G_x(t)\eta(t)|^2 dt\]\leq Cd(\th, \bar{\th})^2.
           $
           On the other hand, in analogy to  $X$-estimate, one has
           $
           \dbE\[|I^{1,2}|^2+\int_0^T|H(t)|^2dt\]\leq Cd(\th, \bar{\th})^2.
           $
           This finishes the proof.
\end{proof}

           { In the following,  for $\Xi_\th= Y_\th, Z_\th, \bar{Z}_\th, K_{\th,ij}$, we define
            $\d^{\e}\Xi_\th(t)\deq \frac{1}{\e}(\Xi^\e_\th(t)-\h{\Xi}_\th(t))-\Xi^{1}_\th(t).$}
            The following lemma states the boundness and the continuity of $(\d^\e Y_\th(\cd), \d^\e Z_\th(\cd), \d^\e \cl Z_\th(\cd),\d^\e K_{\th}(\cd))$ with respect to $\th$.

\begin{lemma}\label{le 5.2}\rm
            Under  \autoref{assumption3}, there exists a constant $C>0$ depending on $L,T,\sum\limits_{i,j=1,i\neq j}^I\l_{ij}$ such that
            \begin{equation*}
             \begin{aligned}
              \mathrm{(i)}\ & \dbE\bigg[\sup_{t\in[0,T]}|\d^\e Y_\th(t)|^2+\int_0^T(|\d^\e Z_\th(t)|^2+|\d^\e \cl Z_\th(t)|^2
                            +\sum\limits_{i,j=1,i\neq j}^I|\d^\e K_{\th ,ij}(t)|^2\l_{ij}\mathbf{1}_{\{\a(t-)=i\}})dt\bigg]\\
              & \leq C\dbE \bigg[|x|^4+\int_0^T(|v(t)|^4+|\h{v}(t)|^4) dt\bigg].\\
              \mathrm{(ii)}\ & \lim_{\e\ra0}\sup_{\th\in\Th} \dbE\bigg[\sup_{t\in[0,T]}|\d^\e Y_\th(t)|^2 +\int_0^T(|\d^\e Z_\th(t)|^2+|\d^\e
                                \bar{Z}_\th(t)|^2
                             +\sum\limits_{i,j=1,i\neq j}^I|\d^\e K_{\th ,ij}(t)|^2\l_{ij}\mathbf{1}_{\{\a(t-)=i\}})dt\bigg]=0.
              \end{aligned}
            \end{equation*}
\end{lemma}

\begin{proof}
 First, notice that $(\d^\e Y_\th(\cd),\d^\e Z_\th(\cd), \d^\e \bar{Z}_\th(\cd),\d^\e K_{\th}(\cd))$ solves the following equation
            \begin{equation}\label{4.22}
            \left\{ \begin{aligned}
              d \d^\e Y_\th(t) & =-\bigg\{ \frac{1}{\e}(f^\e_\th(t)-f_\th(t))-[\pa_xf_\th(t)X_\th^{1}(t)+ \partial_yf_\theta(t)Y_\theta^{1}(t)
                                 + \pa_zf_\th(t)Z_\theta^{1}(t)\\
                               & \hskip 1.1cm + \partial_{\bar{z}}f_\theta(t)\bar{Z}_\theta^{1}(t) +\pa_kf_\th(t)\sum_{i,j=1,i\neq j}^IK^{1}_{\th,ij}(t)
                                  \l_{ij}\mathbf{1}_{\{\a(t-)=i\}}+\langle \pa_v f_\th(t), v(t)-\h{v}(t)\rangle]\bigg\}dt\\
                               &\hskip 1.1cm  +\d^\e Z_\th(t)dW(t)+\d^\e \bar{Z}_\th(t)dG(t)+\d^\e K_{\th}(t)\bullet dM(t),\q t\in[0,T],\\
                 \d^\e Y_\th(T) & =\frac{1}{\e}(\Phi^\e_\th(T)-\Phi_\th(T))-\pa_x\Phi_\th(\h X_\th(T))X_\th^{1}(T).
             \end{aligned}   \right.
            \end{equation}
            Denote, for $\ell=x,y,z,\bar{z},k,v$,
            \begin{equation*}
             \begin{aligned}
             \Pi^\e_\th(t) & =(\h{X}_\th(t)+\l\e(\d^\e X_\th(t)+X^1_\th(t)), \h{Y}_\th(t)+\l\e(\d^\e Y_\th(t)+Y^1_\th(t)),
                               \h{Z}_\th(t)+\l\e(\d^\e Z_\th(t)+Z^1_\th(t)), \\
                           &  \q\, \h{\bar{Z}}_\th(t)+\l\e(\d^\e \cl{Z}_\th(t)+\cl{Z}^1_\th(t)),
                              \sum_{i,j=1,i\neq j}^I\(\h{K}_{\th,ij}(t)+\l\e(\d^\e K_{\th,ij}(t)+K^1_{\th,ij}(t))\)\l_{ij}\mathbf{1}_{\{\a(t-)=i\}},\\
                           & \q \h v(t)+\l\e(v(t)-\h v(t))),\\
                            \pa_\ell f^\e_\th(t) & =\int_0^1\pa_\ell f_\th(t,\Pi^\e_\th(t),  \a(t-))d\l,\\
                        \end{aligned}
                    \end{equation*}
 and
           \begin{equation*}
             \begin{aligned}
             B_{1,\th}^\e(T) & =\int_0^1\pa_x\Phi_\th(\h{X}_\th(T) +\l\e(X^{1}_\theta(T)+\d^\e X_\th(T)))d\l \d^\e X_\th(T),\\
             B_{2,\th}^\e(T) & =\int_0^1[\pa_x\Phi_\th(\h{X}_\th(T)+\l\e(X^{1}_\th(T)+\d^\e X_\th(T))-\pa_x\Phi_\th(\h{X}_\th(T))]d
              \l X^{1}_\th(T),\\
             C_{1,\th}^\e(t) & =\int_0^1[\pa_x f^\e_\th(t)-\pa_x f_\th(t)]X^1_\th(t)+[\pa_y f^\e_\th(t)-\pa_y f_\th(t)]Y^1_\th(t)
                                +[\pa_z f^\e_\th(t)-\pa_z f_\th(t)]Z^1_\th(t)\\
                             & \q  +[\pa_{\bar{z}} f^\e_\theta(t)-\pa_{\bar{z}} f_\th(t)] \cl{Z}^1_\th(t)
                                +[\pa_k f^\e_\th(t)-\pa_k f_\th(t)]  \sum_{i,j=1,i\neq j}^IK^1_{\th,ij}(t)\l_{ij}\mathbf{1}_{\{\a(t-)=i\}}\\
                             & \q +\langle\pa_v f^\e_\th(t)-\pa_v f_\th(t)), v(t)-\h{v}(t)\rangle.
            \end{aligned}
            \end{equation*}
            Hence,  Eq. (\ref{4.22}) can be rewritten as the following linear BSDE
             \begin{equation*}
             \left\{  \begin{aligned}
               d \d^\e Y_\th(t) &  =-\bigg\{ \pa_xf^\e_\th(t)\d^\e X_\th(t) +\pa_yf^\e_\th(t)\d^\e Y_\th(t)+\pa_zf^\e_\th(t)\d^\e Z_\th(t)\\
                                &  \qq\q+\pa_{\bar{z}}f^\e_\th(t)\d^\e \bar{Z}_\th(t)+ \pa_kf^\e_\th(t)\sum_{i,j=1,i\neq j}^I\d^\e K_{\th,ij}(t)
                                   \l_{ij}\mathbf{1}_{\{\a(t-)=i\}}+C_{1,\th}^\e(t)\bigg\}dt\\
                                & \qq \q+\d^\e Z_\th(t)dW(t)+\d^\e \bar{Z}_\th(t)dG(t)+\d^\e K_{\th}(t)\bullet dM(t),\ t\in[0,T],\\
                 \d^\e Y_\th(T) & =B_{1,\th}^\e(T)+B_{2,\th}^\e(T).
             \end{aligned} \right.
             \end{equation*}
             From \autoref{th 6.1},  item (i) of  \autoref{assumption3}  and (\ref{4.12-1}), (\ref{4.21-1}), item $\mathrm{(i)}$ of
             \autoref{le 5.1},  there exists a constant $C>0$ depending on $L,T,\sum_{i,j=1,i\neq j}^I\l_{ij}$ such that
             $$ \begin{aligned}
              &  \dbE\bigg[\sup_{t\in[0,T]}|\d^\e Y_\th(t)|^2+\int_0^T(|\d^\e Z_\th(t)|^2+|\d^\e \bar{Z}_\th(t)|^2+\sum_{i,j=1,i\neq j}^I
                 |\d^\e K_{\th ,ij}(t)|^2\l_{ij}\mathbf{1}_{\{\a(t-)=i\}})dt\bigg]\\
               \leq &C\dbE \bigg[|B_{1,\th}^\e(T)+B_{2,\th}^\e(T)|^2+ \int_0^T(| \pa_xf^\e_\th(t)\d^\e X_\th(t) |^2+   |C_{1,\th}^\e(t)|^{2})dt\bigg]\\
                \leq &C\dbE\bigg[|x|^4+\int_0^T(|v(t)|^4+|\bar{v}(t)|^4) dt\bigg].
              \end{aligned} $$
             The proof of item $\mathrm{(ii)}$ comes from dominated convergence theorem.
\end{proof}

Now, we define
              \begin{align}\nonumber
              \cQ^v=\bigg\{Q\in \cQ\Big|J(v(\cd))=\int_{\Th}  \dbE[R^v_\th(T)\Psi_\th(X_\th^v(T))+\L_\th(Y^v_\th(0))] Q(d\th)\bigg\}.
              \end{align}
             \noindent Note that $\cQ^v$ is nonempty.  Indeed, according to the definition of $J(v(\cd))$, there exists a sequence $Q^L\in \cQ$ such that $\int_{\Th}\dbE[R^v_\th(T)\Psi_\th(X_\th^v(T))+\L_\th(Y^v_\th(0))] Q^L(d\th)\geq J(v(\cdot))-\frac{1}{L}$. Notice that $\cQ$ is weakly compact, one can find a $Q^v\in \cQ$ such that a subsequent of $Q^L$ (if necessary) converges weakly to $Q^v$. Thanks to \autoref{le 4.1}, \autoref{le 4.2} and \autoref{le 5.3-0}, (\ref{4.7-2}) and  \autoref{assumption4},
             the map $\th\mapsto\dbE[R^v_\th(T)\Psi_\th(X_\th^v(T))+\L_\th(Y^v_\th(0))]$ is   bounded  and  continuous. Hence, we obtain
             $$
             \begin{aligned}
             &J(v(\cd))\geq\int_{\Th}\dbE[R^v_\th(T)\Psi_\th(X_\th^v(T))+\L_\th(Y^v_\th(0))] Q^v(d\th)\\
             =&\lim_{L\ra\infty}\int_\Th \dbE[R^v_\th(T)\Psi_\th(X_\th^v(T))+\L_\th(Y^v_\th(0))] Q^L(d\th)\geq J(v(\cd)).
             \end{aligned}
             $$
             Thereby, $\cQ^v$ is nonempty. Now, we can show the following variational inequality.
\begin{theorem}\label{le 5.3}\rm
             Let   \autoref{assumption3}  and  \autoref{assumption4}  hold true. Then
             \begin{equation*}
             \begin{aligned}
             0&\le \lim_{\e\ra0}\frac{1}{\e}\(J(v^\e(\cd))-J(\h{v}(\cd)) \)\\
             &=\sup_{Q\in \cQ^{\h{v}}}\int_\Th\dbE[R^{1}_\th(T)\Psi_\th(\h{X}_\th(T))+\h{R}_\th(T)\pa_x\Psi_\th(\h{X}_\th(T))X^{1}_\th(T)
                +\pa_y\L_\th(\h{Y}_\th(0)) Y^{1}_\th(0)  ]Q(d\th).
             \end{aligned}
             \end{equation*}
\end{theorem}
             \begin{proof}
              The proof is split into four steps.

              \emph{Step 1.}  Notice that for arbitrary $Q\in \mathcal{Q}^{\h{v}}$, we have
             \begin{align}\nonumber
             \begin{aligned}
               J(v^\e(\cd)) &  \geq\int_{\Th}  \dbE[R^\e_\th(T)\Psi_\theta(X_\theta^\varepsilon(T))
                          +\Lambda_\theta(Y^\varepsilon_\theta(0))] Q(d\theta),\\
              J(\h{v}(\cd)) &  =\int_{\Th}  \dbE[\h{R}_\theta(T)\Psi_\th(\h{X}_\th(T))+\L_\th(\h{Y}_\th(0))] Q(d\th).
             \end{aligned}
             \end{align}
             From  the fact $\frac{1}{\e}\(R^\e_\th(T)-\h R_\th(T)\)=\d^\e R_\th(T)+R^1(T), \frac{1}{\e}\(X^\e_\th(T)-\h X_\th(T)\)=\d^\e X_\th(T)+X^1(T),\\
             \frac{1}{\e}\(Y^\e_\th(T)-\h Y_\th(T)\)=\d^\e Y_\th(T)+Y^1(T),$ one has
             \begin{equation}\label{4.25-1}
              \begin{aligned}
              &  \frac{1}{\e}\( J(v^\e(\cd))-J(\h{v}(\cd)) \)\\
                \geq&\frac{1}{\e}\int_{\Th}  \dbE[R^\e_\th(T)\Psi_\th(X_\th^\e(T))-\h{R}_\th(T)\Psi_\th(\h{X}_\th(T))
                  +\L_\th(Y^\e_\th(0))-\L_\th(\h{Y}_\th(0))] Q(d\th)\\
                =&\int_{\Th}  \dbE\[\Psi_\th(X_\th^\e(T))(\d^\e R_\th(T)+R_\th^{1}(T))
                   +\h{R}_\th(T)\pa_x\Psi^\e_\th(T)(\d^\e X_\th(T)+X_\th^{1}(T))\\
                    &  \qq\q+\pa_x\L_\th^\e(0)(\d^\e Y_\th(0)+Y_\th^{1}(0))\] Q(d\th)\\
%
              =&\int_{\Th}  \dbE\[\Psi_\th(X_\th^\e(T))\d^\e R_\th(T)+\h{R}_\th(T)\pa_x\Psi^\e_\th(T)\d^\e X_\th(T)
                 +\pa_y\L_\th^\e(0)\d^\e Y_\th(0)\] Q(d\th)\\
             & \q +\int_{\Th}  \dbE\[\Psi_\th(\h{X}_\th(T))R_\th^{1}(T)+\h{R}_\th(T)\pa_x\Psi_\th(\h{X}_\th(T))X^{1}_\th(T)
               +\pa_y\L_\th(\h{Y}_\th(0)) Y^{1}_\th(0) \] Q(d\th),\\
             & \q+\int_{\Th}  \dbE\[(\Psi_\th(X_\th^\e(T))-\Psi_\th(\h{X}_\th(T))) R_\theta^{1}(T)
               + \h{R}_\th(T)(\pa_x\Psi^\e_\th(T)-\pa_x\Psi_\th(\h{X}_\th(T)))X_\th^{1}(T)\\
             & \q+(\pa_y\L_\th^\e(0)-\pa_y\L_\th(\h{Y}_\th(0)))Y_\th^{1}(0)\] Q(d\th),\\
             \end{aligned}
             \end{equation}
              where
              \begin{align}\label{6.21}
              \begin{aligned}
              \pa_x\Psi^\e_\th(T)&=\int_0^1\pa_x\Psi_\th( \h{X}_\th(T))+\l\e(\d^\e X_\th(T)+X_\th^{1}(T))d\l,\\
              \pa_x\L_\th^\e(0)&=\int_0^1 \pa_x\L_\th(\h{Y}_\th (0))+\l\e(\d^\e Y_\th(0)+Y_\th^{1}(0))d\l.
              \end{aligned}
              \end{align}
             Since $|\Psi_\th(X_\th^\e(T))|\leq L(1+|X_\th^\e(T)|^2)$,  $|\pa_x\Psi^\e_\th(T)|\leq L(1+|\h{X}_\th(T)|+|\d^\e X_\th(T)|+|X_\th^{1}(T)|)$ and $\pa_y\L^\e_\th(0)$ is uniformly bounded, we have
             \begin{equation*}
              \begin{aligned}
               &\Big|\dbE\big[\Psi_\th(X_\th^\e(T))\d^\e R_\th(T)+\h{R}_\th(T)\pa_x\Psi^\e_\th(T)\d^\e X_\th(T) +\pa_y\L_\th^\e(0)\d^\e Y_\th(0)\big]\Big|\\
              \leq& L\dbE\[(1+|X_\th^\e(T)|^2)|\d^\e R_\th(T)|
              +(1+|\h{X}_\th(T)|+|\d^\e X_\th(T)|+|X_\th^{1}(T)|) |\h{R}_\th(T)||\d^\e X_\th(T)|
              +|\d^\e Y_\th(0)|\].
              \end{aligned}
             \end{equation*}
             It follows from  \autoref{le 5.1} and \autoref{le 5.2} that
             \begin{equation}\label{4.27}
             \begin{aligned}
             \lim_{\e\ra0}\sup_{\th\in \Th}\dbE\[\Psi_\th(X_\th^\e(T))\d^\e R_\th(T) +\h R_\th(T)\pa_x\Psi^\e_\th(T)\d^\e X_\th(T)
               +\pa_y\L_\th^\e(0)\d^\e Y_\th(0)\]=0.
             \end{aligned}
             \end{equation}
Thanks to  \autoref{assumption4},  it follows
            \begin{align}\nonumber
            \begin{aligned}
              &  \Big| \dbE\big[(\Psi_\th(X_\th^\e(T))-\Psi_\th(\h{X}_\th(T)))R_\th^{1}(T)
                 +\h{R}_\th(T)(\pa_x\Psi^\e_\th(T)-\pa_x\Psi_\th(\h{X}_\th(T)))X_\th^{1}(T)\\
              &  +(\pa_y\L_\th^\e(0)-\pa_y\L_\th(\h{Y}_\th(0)))Y_\th^{1}(0)\big]\Big|\\
                 \leq& L\e  \dbE\[ (1+|X_\th^\e(T)|+|\h{X}_\th(T)|)|X^{1}_\th(T)+\d^\e  X_\th(T)||R^1_\th(T)|\\
              &  +|X^{1}_\th(T)+\d^\e  X_\th(T)||\h{R}_\th(T)||X^{1}_\th(T)|
                  +|Y^{1}_\th(0)+\d^\e  Y_\th(0)||Y^{1}_\th(0)|\],
            \end{aligned}
            \end{align}
            which combining  \autoref{le 4.1}, \autoref{le 5.1}, \autoref{le 5.2}, and (\ref{4.7-2}), (\ref{4.12-1}),  (\ref{4.21-1}) yields
            \begin{align}\label{4.28}
               \begin{aligned}
               & \lim_{\e\ra0}\sup_{\th\in\Th}\dbE\[(\Psi_\th(X_\th^\e(T))-\Psi_\th(\h{X}_\th(T)))R_\th^{1}(T)
                 +\h{R}_\th(T)(\pa_x\Psi^\e_\th(T)-\pa_x\Psi_\th(\h{X}_\th(T)))X_\th^{1}(T)\\
               &  \qq\qq\q  +(\pa_y\L_\th^\e(0)-\pa_y\L_\th(\h{Y}_\th(0)))Y_\th^{1}(0)\]=0.
               \end{aligned}
            \end{align}
           Insert (\ref{4.27}) and (\ref{4.28}) into (\ref{4.25-1}) we deduce
           \begin{align}\label{4.29}
             \begin{aligned}
             & \liminf_{\e\ra0}\frac{1}{\e}\(  J(v^\e(\cd))-J(\h{v}(\cd)) \)\\
             \geq&\int_{\Th}  \dbE\[\Psi_\th(\h{X}_\th(T))R_\th^{1}(T)+\h{R}_\th(T)\pa_x\Psi_\th(\h{X}_\th(T))X^{1}_\th(T)
               +\pa_y\L_\th(\h{Y}_\th(0)) Y^{1}_\th(0) \] Q(d\th).
             \end{aligned}
           \end{align}

           \emph{Step 2.} Next, let us choose a subsequence $\varepsilon_{_M}\ra0$  such that
           $$
             \limsup_{\e\ra0}\frac{1}{\e}\(  J(v^\e(\cd))-J(\h{v}(\cd)) \)=\lim_{M\ra\infty}\frac{1}{\e_{_M}}\( J(v^{\e_{_M}}(\cd))-J(\h{v}(\cd)) \).
           $$
            For every $M\geq1$, since $\cQ^{v^{\e_{_{_M}}}}$ is nonempty, there exists a probability measure
            $Q^{{\e_{_M}}}\in \cQ^{v^{\e_{_M}}}$ such that
            \begin{align}\label{6.7-1}
              \begin{aligned}
               J(v^{\e_{_M}}(\cd)) &  =\int_{\Th}\dbE\[R^{{\e_{_M}}}_\th(T)\Psi_\th(X_\theta^{{\e_{_M}}}(T))
                                 +\L_\th(Y^{{\e_{_M}}}_\th(0))\] Q^{{\e_{_M}}}(d\th),\\
                     J(\h{v}(\cd)) &  \geq\int_{\Th}  \dbE\[\h{R}_\th(T)\Psi_\th(\h{X}_\th(T))+\L_\th(\h{Y}_\th(0))\] Q^{{\e_{_M}}}(d\th).
               \end{aligned}
            \end{align}
            Similar to (\ref{4.25-1}), one gets
             \begin{equation*}
             \begin{aligned}
               & \frac{1}{\e_{_M}}\( J(v^{\e_{_M}}(\cd))-J(\h{v}(\cd)) \)\\
                \leq &\frac{1}{\e_{_M}}\int_{\Th}  \dbE[R^{{\e_{_M}}}_\th(T)\Psi_\th(X_\th^{{\e_{_M}}}(T))-\h{R}_\th(T)\Psi_\th(\h{X}_\th(T))
                 +\L_\th(Y^{{\e_{_M}}}_\th(0))-\L_\th(\h{Y}_\th(0))] Q^{{\e_{_M}}}(d\th)\\
              =&\int_{\Th}  \dbE\big[\Psi_\th(X_\th^{\e_{_M}}(T))\d^{\e_{_M}} R_\th(T) +\h{R}_\th(T)\pa_x\Psi^{\e_{_M}}_\th(T)\d^{\e_{_M}}
                    X_\th(T)+\pa_y\L_\th^{\e_{_M}}(0)\d^{\e_{_M}} Y_\th(0)\big] Q^{{\e_{_M}}}(d\th)\\
              & \q +\int_{\Th}  \dbE\big[\Psi_\th(\h{X}_\th(T))R_\th^{1}(T)+\h{R}_\th(T)\pa_x\Psi_\th(\h{X}_\th(T))X_\th^{1}(T)
                  +\pa_y\L_\th(\h{Y}_\th(0))Y_\th^{1}(0)\big] Q^{{\e_{_M}}}(d\th)\\
              & \q +\int_{\Th}  \dbE\big[(\Psi_\th(X_\th^{\e_{_M}}(T))-\Psi_\th(\h{X}_\th(T)))R_\theta^{1}(T)
                +(\pa_x\Psi^{\e_{_M}}_\th(T)-\pa_x\Psi_\th(\h{X}_\th(T)))\h{R}_\th(T)X_\th^{1}(T)\\
              & \q+(\pa_y\L_\th^{\e_{_M}}(0)-\pa_y\L_\th(\h{Y}_\th(0)))Y_\th^{1}(0)\big] Q^{{\e_{_M}}}(d\th),
             \end{aligned}
             \end{equation*}
             where  $\pa_x\Psi^{\e_{_M}}_\th(T), \pa_y\L_\th^{\e_{_M}}(0)$ are given in (\ref{6.21}).
             We analyze the above terms one by one.  In analogous to (\ref{4.27}), we arrive at
             \begin{align}\nonumber
               \begin{aligned}
                & \lim_{M\ra\infty}\sup_{\th\in \Th}\dbE[\Psi_\th(X_\th^{\e_{_M}}(T))\d^{\e_{_M}} R_\th(T) +\h{R}_\th(T)\pa_x\Psi^{\e_{_M}}_\th(T)\d^{\e_{_M}}
                    X_\th(T)+\pa_y\L_\th^{\e_{_M}}(0)\d^{\e_{_M}} Y_\th(0)]=0.
               \end{aligned}
             \end{align}
             Since $\cQ$ is weakly compact, there exists a probability measure $Q^{\#}\in\cQ$ such that a subsequence of $(Q^{\e_{_M}})_{M\geq1}$,  still denoted by $(Q^{\e_{_M}})_{M\geq1}$, converges to $Q^{\#}$. Define
             \begin{equation}\label{4.9999}
             H(\th):=\dbE[\Psi_\th(\h{X}_\th(T))R_\th^{1}(T)+\h{R}_\th(T)\pa_x\Psi_\th(\h{X}_\th(T))X_\th^{1}(T)
             +\pa_y\L_\th(\h{Y}_\th(0))Y_\th^{1}(0)].
             \end{equation}
             Since  $|\Psi_\th(\h{X}_\th(T))|\leq L(1+|\h{X}_\th(T)|^2), |\pa_x\Psi_\th(\h{X}_\th(T))|\leq L(1+|\h{X}_\th(T)|)$, from  \autoref{le 4.1}, (\ref{4.7-2}), (\ref{4.12-1}), (\ref{4.21-1}), the map $\th\mapsto H(\th)$ is bounded. In addition, according to  \autoref{assumption4}, (\ref{4.12-1}) and \autoref{le 6.3-0},    the map $\th\mapsto H(\th)$ is continuous. Hence, we have
              \begin{equation*}
              \begin{aligned}
               &   \lim_{M\ra\infty}\int_{\Th}  \dbE\[\Psi_\th(\h{X}_\th(T))R_\th^{1}(T)+\h{R}_\th(T)\pa_x\Psi_\th(\h{X}_\th(T))X_\th^{1}(T)
                  +\pa_y\L_\th(\h{Y}_\th(0))Y_\th^{1}(0)\] Q^{{\e_{_M}}}(d\th)\\
                 =&\int_{\Th}\dbE\[\Psi_\th(\h{X}_\th(T))R_\th^{1}(T)+\h{R}_\th(T)\pa_x\Psi_\th(\h{X}_\th(T))X_\th^{1}(T)
                  +\pa_y\L_\th(\h{Y}_\th(0))Y_\th^{1}(0)\] Q^{\#}(d\theta).
               \end{aligned}
              \end{equation*}
             Finally, it follows from  \autoref{assumption4}, \autoref{le 5.1} and \autoref{le 5.2},  (\ref{4.7-2}), (\ref{4.12-1}), (\ref{4.21-1})
             that
            $$
            \begin{aligned}
              &   \lim _{M\rightarrow\infty}   \Big|\int_{\Th}  \dbE\[(\Psi_\th(X_\th^{\e_{_M}}(T))-\Psi_\th(\h{X}_\th(T)))R_\theta^{1}(T)
                 +(\pa_x\Psi^{\e_{_M}}_\th(T)-\pa_x\Psi_\th(\h{X}_\th(T)))\h{R}_\th(T)X_\th^{1}(T)\\
              & \qq+(\pa_y\L_\th^{\e_{_M}}(0)-\pa_y\L_\th(\h{Y}_\th(0)))Y_\th^{1}(0)\] Q^{{\e_{_M}}}(d\th)\Big|\\
                \leq& L\lim_{M\ra\infty}  \e_{_M} \int_{\Th}  \dbE\[
                  (1+|X_\th^{\e_{_M}}(T)|+|\h{X}_\th(T)|)|X^{1}_\th(T)+\d^{\e_{_M}} X_\th(T)| |R^{1}_\th(T)| \\
              & \q     + |X^{1}_\th(T)+\d^{\e_{_M}} X_\th(T)| |\h{R}_\th(T)||X^{1}_\th(T)|
              +|Y^{1}_\th(0)+\d^{\e_{_M}}   Y_\th(0)||Y^{1}_\th(0)|\]Q^{{\e_{_M}}}(d\th)=0.
              \end{aligned}
            $$
            Consequently,  we have
             \begin{equation*}
               \begin{aligned}
                & \limsup_{\e\ra0}\frac{1}{\e}\(  J(v^\e(\cd))-J(\h{v}(\cd)) \)=\lim_{M\ra\infty}\frac{1}{\e_{_M}}\( J(v^{\e_{_M}}(\cd))-J(\h{v}(\cd))\)\\
                \leq&\int_{\Th}\dbE\[\Psi_\th(\h{X}_\th(T))R_\th^{1}(T)+\h{R}_\th(T)\pa_x\Psi_\th(\h{X}_\th(T))X_\th^{1}(T)
                  +\pa_y\L_\th(\h{Y}_\th(0))Y_\th^{1}(0)\] Q^{\#}(d\th).
               \end{aligned}
             \end{equation*}

              \emph{Step 3.}   We claim that $Q^{\#}\in \cQ^{\h{v}}$.  Indeed, on the one hand, from the definition of $J(v(\cd))$ and (\ref{6.7-1}),
              \begin{equation*}
                 \begin{aligned}
              &   |J(v^{\e_{_M}}(\cd))-J(\h{v}(\cd))|=J(v^{\e_{_M}}(\cd))-J(\h{v}(\cd))\\
                 \leq& \int_{\Th}\dbE\[|R^{\e_{_M}}_\th(T)\Psi_\th(X_\th^{\e_{_M}}(T))-\h{R}_\th(T)\Psi_\th(\h{X}_\th(T))|
                  +|\L_\th(Y^{\e_{_M}}_\th(0))-\L_\th(\h{Y}_\th(0))|\]Q^{\e_{_M}}(d\th)\\
                  \leq &\sup_{\th\in \Th}\dbE\[|R^{\e_{_M}}_\th(T)\Psi_\th(X_\th^{\e_{_M}}(T))-\h{R}_\th(T)\Psi_\th(\h{X}_\th(T))|
                  +|\L_\th(Y^{\e_{_M}}_\th(0))-\L_\th(\h{Y}_\th(0))|\]   \\
                 \leq& L\e_{_M}\sup_{\th\in \Th}\dbE\bigg[|R^{1}(T)+\d^{\e_M} R_\th(T)||\Psi_\th(X_\th^{\e_{_M}}(T))|+|Y^{1}_\th(0)+\d^{\e_M} Y_\th(0)|\\
              &\hskip 1.8cm+(1+|X_\th^{\e_{_M}}(T)|+|\h{X}_\th(T)|)|X^{1}_\th(T)+\d^{\e_M} X_\th(T)|  |\h{R}_\th(T)|\bigg],
              \end{aligned}
              \end{equation*}
              which and the assumption $|\Psi_\th(X_\th^{\e_{_M}}(T))|\leq L(1+|X_\th^{\e_{_M}}(T)|^2)$,  \autoref{le 5.1}, \autoref{le 5.2},  (\ref{4.7-2}), (\ref{4.12-1}) and (\ref{4.21-1})  allow to show
              \begin{equation}\label{4.34111}
               \begin{aligned}
               \lim_{M\ra\infty}|J(v^{\e_{_M}}(\cd))-J(\h{v}(\cd))|=0.
                \end{aligned}
              \end{equation}
              On the other hand,
              \begin{equation}\label{4.35112}
                \begin{aligned}
                & \lim_{M\ra\infty}\int_{\Th}  \Big|\dbE[R^{{\e_{_M}}}_\th(T)\Psi_\th(X_\th^{{\e_{_M}}}(T))+\L_\th(Y^{{\e_{_M}}}_\th(0))
                  -(\h{R}_\th(T)\Psi_\th(\h{X}_\th(T))+\L_\th(\h{Y}_\th(0)))]\Big| Q^{{\e_{_M}}}(d\th)\\
                \leq& L\lim_{M\ra\infty}\e_{_M} \int_{\Th}  \dbE\[ |R^{1}(T)+\d^\e R_\th(T)||\Psi_\th(X_\th^{\e_{_M}}(T))|
                 +|Y^{1}_\th(0)+\d^\e Y_\th(0)|
               \\
               &\hskip 2cm  +(1+|X_\th^{\e_{_M}}(T)|+|\h{X}_\th(T)|)|X^{1}_\theta(T)+\d^\e X_\th(T)|  |\h{R}_\th(T)| \]Q^{{\e_{_M}}}(d\th)=0.
              \end{aligned}
              \end{equation}
               Hence,  we can derive from (\ref{4.34111}) and (\ref{4.35112}), the boundness and continuity of the map
               $\th\mapsto\dbE[\h{R}_\th(T)\Psi_\th(\h{X}_\th(T))+\L_\th(\h{Y}_\th(0))]$ that
               \begin{align}\nonumber
                 \begin{aligned}
                  J(\h{v}(\cd))&=\lim_{M\ra\infty}J(v^{\e_{_M}}(\cd))=\lim_{M\ra\infty}\int_{\Th}\dbE[R^{{\e_{_M}}}_\th(T)\Psi_\th(X_\th^{{\e_{_M}}}(T))
                   +\L_\th(Y^{{\e_{_M}}}_\th(0))] Q^{{\e_{_M}}}(d\th) \\
                 & =\lim_{M\ra\infty}\int_{\Th}\dbE[\h{R}_\th(T)\Psi_\th(\h{X}_\th(T))+\L_\th(\h{Y}_\th(0))] Q^{{\e_{_M}}}(d\th)\\
                 &  =\lim_{M\ra\infty}\int_{\Th}\dbE[\h{R}_\th(T)\Psi_\th(\h{X}_\th(T))+\L_\th(\h{Y}_\th(0))] Q^{\#}(d\th),
                \end{aligned}
               \end{align}
               which means $Q^{\#}\in \cQ^{\h{v}}$.

              \emph{Step 4.}   Taking  $Q=Q^{\#}$ in  (\ref{4.29}), it follows
              \begin{equation*}
                \begin{aligned}
                & \lim_{\e\ra0}\frac{1}{\e}\Big( J(v^{\e_{_M}}(\cd))-J(\h{v}(\cd)) \Big)\\
                =&\int_{\Th}\dbE\[\Psi_\th(\h{X}_\th(T))R_\th^{1}(T)+\h{R}_\th(T)\pa_x\Psi_\th(\h{X}_\th(T))X_\th^{1}(T)
                  +\pa_y\L_\th(\h{Y}_\th(0))Y_\th^{1}(0)\] Q^{\#}(d\theta)\\
                =&\sup_{Q\in \cQ^{\h{v}}}\int_{\Th}\dbE\[\Psi_\th(\h{X}_\th(T))R_\th^{1}(T)+\h{R}_\th(T)\pa_x\Psi_\th(\h{X}_\th(T))X_\th^{1}(T)
                  +\pa_y\L_\th(\h{Y}_\th(0))Y_\th^{1}(0)\]Q(d\theta).
                 \end{aligned}
              \end{equation*}
              \end{proof}

               The above variational inequality implies  the following result.
\begin{theorem}\label{th 5.4}\rm
              Let   \autoref{assumption3}  and  \autoref{assumption4}  be in force, then there exists some probability $\overline{Q} \in \cQ^{\h{v}}$ such that, for all
              $v(\cd)\in \cV_{ad}$,
              \begin{equation*}
              \begin{aligned}
             \int_\Th \dbE\[\Psi_\th(\h{X}_\th(T))R_\th^{1;v}(T)+\h{R}_\th(T)\pa_x\Psi_\th(\h{X}_\th(T))X_\th^{1;v}(T)
                  +\pa_y\L_\th(\h{Y}_\th(0))Y_\th^{1;v}(0)\]\cl{Q}(d\theta)\geq0,
              \end{aligned}
              \end{equation*}
              where $ X_\th^{1;v}(\cd),R^{1;v}_\theta(\cd)$ and  $(Y_\theta^{1;v}(\cd),Z_\theta^{1;v}(\cd),\bar{Z}_\theta^{1;v}(\cd),K_\theta^{1;v}(\cd))$ denote the solutions to the  variational SDE (\ref{11.2}), the  variational observation process (\ref{11.1})
              and the variational BSDE (\ref{11.3}), respectively.
\end{theorem}

                \begin{proof}
                Define
               \begin{equation*}
                S(\theta,v):=\dbE\[\Psi_\th(\h{X}_\th(T))R_\th^{1;v}(T)+\h{R}_\th(T)\pa_x\Psi_\th(\h{X}_\th(T))X_\th^{1;v}(T)
                             +\pa_y\L_\th(\h{Y}_\th(0))Y_\th^{1;v}(0)\].
                \end{equation*}
               From \autoref{le 5.3}, $\lim\limits_{\e\ra0}\frac{1}{\e}\(  J(v^\e(\cd))-J(\h{v}(\cd))\) =\sup\limits_{Q\in \cQ^{\h{v}}} \int_\Th S(\th,v)Q(d\th)\geq0.$ Consequently, we  obtain
              \begin{equation*}
                \begin{aligned}
                 \inf_{v(\cd)\in \cV_{ad}}\sup_{Q\in \cQ^{\h{v}}}\int_\Th S(\th,v)Q(d\th)\geq0.
                \end{aligned}
              \end{equation*}
              Next, we are ready to use Sion's theorem. To this end, we need to  show that $S(\theta,v)$ is convex and continuous with respect to $v$. In fact, on the one hand, from the fact that,  for $0\leq\rho\leq1$, $v(\cd),\tilde{v}(\cd)\in \mathcal{V}_{ad}$,
              $$
               \begin{aligned}
                & R^{1;\rho v+(1-\rho)\ti{v}}_\th(T)=\rho R^{1; v}_\th(T)+(1-\rho)R^{1;  \ti{v}}_\th(T),\\
                & X^{1;\rho v+(1-\rho)\ti{v}}_\th(T)=\rho X^{1; v}_\th(T)+(1-\rho)X^{1;  \ti{v}}_\th(T),\\
                & Y^{1;\rho v+(1-\rho)\ti{v}}_\th(0)=\rho Y^{1; v}_\theta(0)+(1-\rho)Y^{1;  \tilde{v}}_\th(0),\\
              \end{aligned}
              $$
             it yields, for  all  $\th\in \Th$, $S(\th,\rho v+(1-\rho)\ti{v})=\rho S(\th, v)+(1-\rho)S(\th, \ti{v})$.
               On the other hand,  we have from \autoref{th 3.2} that for $v(\cd),\tilde{v}(\cd)\in \mathcal{V}_{ad}$,
               $$
                \begin{aligned}
                 & \dbE\[|R^{1;v}_\th(T)-R^{1;\ti{v}}_\th(T)|^2\]\leq C \dbE\[\int_0^T|v(t)-\ti{v}(t)|^4dt   \]^\frac{1}{2},\\
                 & \dbE\[|X^{1;v}_\th(T)-X^{1;\ti{v}}_\th(T)|^4\]\leq C\dbE\[\int_0^T|v(t)-\ti{v}(t)|^4dt   \],\\
                \end{aligned}
                $$
               and from \autoref{th 6.1} that
               $$
               \begin{aligned}
                 &{|Y^{1;v}_\th(0)-Y^{1;\ti{v}}_\th(0)|^2}\leq C \dbE\[\int_0^T|v(t)-\ti{v}(t)|^4dt\]^\frac{1}{2},
                \end{aligned}
               $$
               where  the constant $C>0$ depends on $L, T,\sum\limits_{i,j=1,i\neq j}^I\l_{ij}$. Hence, $S(\theta,v)$ is continuous in $v$, for each $\theta$. Applying {Sion's theorem  (see \cite{Sion-1958})}, we derive
               \begin{equation*}
                \begin{aligned}
                  \sup_{Q\in \cQ^{\h{v}}}\inf_{v(\cd)\in \cV_{ad}}\int_\Th S(\th,v)Q(d\th)
                   =\inf_{v(\cd)\in \cV_{ad}}\sup_{Q\in \cQ^{\h{v}}}\int_\Th S(\th,v)Q(d\th)\geq0.
                  \end{aligned}
               \end{equation*}
               Then, for arbitrary $\d>0$, there exists a probability $Q^\d\in \cQ^{\h{v}}$ such that
               $$
                 \inf_{v(\cd)\in \cV_{ad}}\int_\Th S(\th,v)Q^\d(d\th) \geq-\d.
               $$
                From the compactness of $\cQ^{\h{v}}$, there exists a subsequence   $\d_M\rightarrow0$ such that $Q^{\d_M}$ converges weakly to a probability $\overline{Q}\in\cQ^{\h{v}}$. Then, one has,
                $$\int_\Th S(\th,v)\cl{Q}(d\th)=\lim_{\d_M\ra0}\int_\Th S(\th,v)Q^{\d_M}(d\th)\geq0,\q \forall v(\cd)\in \cV_{ad}.$$
                This completes the proof.
              \end{proof}

\subsection{Adjoint Equations and SMP}

             For simplicity, we denote $\cl W_\th(\cd)=\cl W_\th^{\h{v}}(\cd),$  for each $\th\in \Th$. Let us introduce the following adjoint equations:
             \begin{equation}\label{4.36}
               \left\{  \begin{aligned}
                 dp^A_\th(t) & =q^A_\th(t) dW(t)+\bar{q}^A_\th(t) d\cl{W}_\th(t)+k^A_\th(t)\bullet dM(t),\ t\in[0,T],\\
               p^A_\th(T)    & =\Psi_\th(\h{X}_\th(T)),
             \end{aligned}  \right.
             \end{equation}
             and
             \begin{equation}\label{4.37}
             \left\{   \begin{aligned}
               dp^F_\th(t) & =\pa_yf_\th(t)p^F_\th(t)dt+\pa_zf_\th(t)p^F_\th(t)dW(t)+[\pa_{\bar{z}}f_\th(t)-h_\th(t)]p^F_\th(t)d\cl{W}_\th(t)\\
                           &\q + \pa_kf_\th(t)p^F_\th(t)\bullet dM(t),\ t\in[0,T],\\
                dp^B_\th(t)& =-\bigg\{[\pa_xb_\th(t)-\bar{\si}_\th(t)\pa_xh_\th(t)]p^B_\th(t)
                              +\pa_x\si_\th(t)q^B_\th(t)+\pa_x\bar{\si}_\th(t)\bar{q}^B_\th(t)\\
                          &  \qq\q+\sum_{i,j=1,i\neq j}^I\pa_x\b_{\th,ij}(t)k^B_{ij}(t)\l_{ij}\mathbf{1}_{\{\a(t-)=i\}}
                             +\pa_xh_\th(t)\bar{q}^A_\th(t)-\pa_xf_\th(t)p^F_\th(t)\bigg\}dt\\
                          &  \qq\q+q^B_\th(t)dW(t)+\bar{q}^B_\th(t)d\cl{W}_\th(t)+k^B_\th(t)\bullet dM(t),\ t\in[0,T],\\
                p^F_\th(0)&  =-\pa_x\L_\th(\h Y_\th(0)),\q p^B_\th(T)=\pa_x\Psi_\th(\h{X}_\th(T))-\pa_x\Phi_\th(\h{X}_\th(T))p^F_\th(T).
             \end{aligned}  \right.
             \end{equation}
             According to  \autoref{th 3.2} and \autoref{th 6.1},   there exists a constant $C>0$ such that
             \begin{equation}\label{6.44}
             \begin{aligned}
             &  \mathrm{(i)}\ \dbE\[\sup_{t\in[0,T]}|p^F_\th(t)|^4\]\leq C,\\
             &  \mathrm{(ii)}\ \dbE\[\sup_{t\in[0,T]}|p^A_\th(t)|^2+\int_0^T(|q^A_\th(t)|^2+|\bar{q}^A_\th(t)|^2
                +\sum_{i,j=1,i\neq j}^I|k^A_{\th,ij}(t)|^2\l_{ij}\mathbf{1}_{\{\a(t-)=i\}})dt\]\\
             & \qq \leq  C\(|x|^4+\dbE\int_0^T|\h v(t)|^4dt\),\\
             &  \mathrm{(iii)}\  \dbE\[\sup_{t\in[0,T]}|p^B_\th(t)|^2+\int_0^T(|q^B_\th(t)|^2+|\bar{q}^B_\th(t)|^2
                  +\sum_{i,j=1,i\neq j}^I|k^B_{\th,ij}(t)|^2\l_{ij}\mathbf{1}_{\{\a(t-)=i\}})dt\]\\
             &\qq \leq C\(|x|^4+\dbE\int_0^T|\h v(t)|^4dt\).
             \end{aligned}
             \end{equation}

             Besides, using the method in \autoref{le 6.3-0}, we can show that
             $p^F_\th(\cd)$, $ (p^A_\th(\cd),q^A_\th(\cd),\bar{q}^A_\th(\cd),\b^A_{\th}(\cd))$, and
       $      (p^B_\th(\cd),q^B_\th(\cd),\bar{q}^B_\th(\cd),\b^B_{\th}(\cd))$ are continuous with respect to  $\th$.
             \begin{lemma}\label{le 6.6}\rm
             Under   \autoref{assumption3}  and  \autoref{assumption4}, we have
            \begin{equation*}
            \begin{aligned}
            &  \lim_{\d\ra 0}\sup_{d(\th, \bar{\th})\leq\d }\dbE\bigg[\sup_{t\in[0,T]}\(|p^F_\th(t)-p^F_{\bar{\th}}(t)|^4
                    +|p^A_\th(t)-p^A_{\bar{\th}}(t)|^2+|p^B_\th(t)-p^B_{\bar{\th}}(t)|^2\)\\
            &  +\int_0^T(|q^A_\th(t)-q^A_{\bar{\th}}(t)|^2+|\bar{q}^A_\th(t)-\bar{q}^A_{\bar{\th}}(t)|^2
                   +\sum_{i,j=1,i\neq j}^I|\b^A_{\th,ij}(t)-\b^A_{\bar{\th},ij}(t)|^2\l_{ij}\mathbf{1}_{\{\a(t-)=i\}})dt \\
            & +\int_0^T(|q^B_\th(t)-q^B_{\bar{\th}}(t)|^2+|\bar{q}^B_\th(t)-\bar{q}^B_{\bar{\th}}(t)|^2
                   +\sum_{i,j=1,i\neq j}^I|\b^B_{\th,ij}(t)-\b^B_{\bar{\th},ij}(t)|^2\l_{ij}\mathbf{1}_{\{\a(t-)=i\}})dt\bigg]=0.
            \end{aligned}
            \end{equation*}
           \end{lemma}

             Define the Hamiltonian as follows: for $i,j,i_0\in \cI$ and  for $t\in[0,T]$, $x\in \dbR$, $y,k_{ij}\in \dbR,$ $z,\bar{z}\in \dbR$,
             $p^F\in \dbR,$  $p^B\in \dbR,$  $q^B,\bar{q}^B\in \dbR,$ $k^B_{ij},$ $\bar{q}^A\in \dbR$,\ $v\in V$,
             \begin{equation}\label{4.38}
               \begin{aligned}
               & H_\th(\o,t,x,y,z,\bar{z},k,v,i_0,p^F,p^B,q^B,\bar{q}^B,k^B,\bar{q}^A)\\
                = & b_\th(\o,t,x,v,i_0)p^B   +\si_\th(\o,t,x,v,i_0) q^{B}
               +\bar{\si}_\th(\o,t,x,v,i_0)\bar{q}^{B}\\
               & +\sum_{i,j=1,i\neq j}^I  \b_{\th,ij}(\o,t,x,v,i_0)k^B_{ij}  \l_{ij}\mathbf{1}_{\{\a(t-)=i\}}
                +  h_\th(\o,t,x,i_0)\bar{q}^A   \\
               & -\(f_\th(\o,t,x,y,z,\bar{z},\sum_{i,j=1,i\neq j}^Ik_{ij}\l_{ij}\mathbf{1}_{\{\a(t-)=i\}},v,i_0)-h_\th(t,x,i_0)\bar{z}\)p^F.
               \end{aligned}
             \end{equation}
             Here $k^B=(k^B_{ij})_{i,j\in \mathcal{I}}$ and $k=(k_{ij})_{i,j\in \mathcal{I}}$.

             Recall that $\h{v}(\cd)$ is an optimal control. By  $\bar{\dbP}^{\h{v}}_\th$ we denote the probability $\bar{\dbP}^{v}_\th$
             given in (\ref{4.6-111}) but with $\h{v}(\cd)$ instead of $v(\cd)$.
             Before proving the stochastic maximum principle, we first give a lemma.
             Define
             \begin{equation*}
             \begin{aligned}
            \Pi_\th(\o, t):=\pa_vH_\th \big(\o,t,\h{X}_\th(t),\h{Y}_\th(t),
                   \h{Z}_\th(t),\h{\bar{Z}}_\th(t),\h{K}_{\th}(t),\h{v}(t),\a(t-), p^F_\th(t),p^B_\th(t),q^B_\th(t),
                \cl{q}^B_\th(t),k^B_{\th}(t),\cl{q}^A_\th(t)\big).
                \end{aligned}
             \end{equation*}

             \begin{lemma}\label{le 7.1}\rm
             Under  \autoref{assumption3}  and  \autoref{assumption4}, the map $(\th,t,\o)\mapsto\dbE\[\Pi_\th(\o, t)|\cG_t\]$ is a $\dbG$-progressively measurable
             process; in other words, for each $t\in[0,T]$, the function $\dbE\[\Pi_\th(\o, t)|\cG_t\]: \Th\ts[0,t]\ts\Omega\ra\dbR$  is $\sB(\Th)\ts
             \sB([0,t])\ts\cG_t$-measurable.
             \end{lemma}

             \begin{proof}    For convenience,  we suppress $\o$ in   $\Pi_\th(\o, t)$.
              Since $\Th$ is a Polish space, for each $M>1$, there exists  a compact subset $C^M\subset \Th$ such that
             $\cl Q(\th\notin C^M)\leq \frac{1}{M}$. Then  we can find a subsequence of open neighborhoods $\(B(\th_l, \frac{1}{2M})\)_{l=1}^{L_M}$ such that $C^M\subset \bigcup_{l=1}^{L_M}\(B(\th_l, \frac{1}{2M})\)$.
             From the  locally compact property of $\Th$ and using partitions of unity, there exists a sequence of continuous functions
             $\k_l:\Th\ra \dbR$ with values in $[0,1]$ such that
             \begin{equation*}
              \k_l(\th)=0,\ \text{if}\ \th\notin B\(\th_l, \frac{1}{2M}\),\  l=1,
              \cds, L_M,\ \text{and}\ \sum_{l=1}^{L_M}\k_l(\th)=1,\ \text{if}\ \th\in C^M.
             \end{equation*}
             We set $\th_l^*$ satisfying  $\k_l(\th_l^*)>0$ and define
             \begin{equation*}
               \Pi^M_\th(t):=\sum_{l=1}^{L_M}\Pi_{\th_l^*}(t)\kappa_l(\th)\mathbf{1}_{\{\th\in C^M\}}.
             \end{equation*}
             Notice that from (\ref{4.7-2}) and  (\ref{6.44}), one has
             \begin{equation*}
             \begin{aligned}
             &     \dbE^{\bar{\dbP}^{\h{v}}_\th}\[\int_0^T |\Pi_\th(t)|dt\]=\dbE\[\h R_{\th}(T)\int_0^T |\Pi_\th(t)|dt\]\leq L.
             \end{aligned}
             \end{equation*}
             Hence, we obtain from Bayes's formula that
             \begin{equation*}
             \begin{aligned}
             &\int_\Th\dbE^{\bar{\dbP}^{\h{v}}_\th}\[\int_0^T\Big|\dbE^{\bar{\dbP}^{\h{v}}_\th}[\Pi_\th^M(t)-\Pi_\th(t)|\cG_t ]\Big|dt\]\cl Q(d\th)\\
             \leq &\int_\Th\dbE^{\bar{\dbP}^{\h{v}}_\th}\[\int_0^T\Big|\Pi_\th^M(t)-\Pi_\th(t)\Big|dt\]\cl Q(d\th)\\
             \leq& \int_\Th \sum_{l=1}^{L_M}\dbE^{\bar{\dbP}^{\h{v}}_\th}\bigg[\int_0^T |\Pi_\th^M(t)-\Pi_\th(t)|dt\bigg]\k_l(\th) \mathbf{1}_{\{\th\in C^M\}}
                  +\dbE^{\bar{\dbP}^{\h{v}}_\th}\bigg[\int_0^T|\Pi_\th(t)|dt\bigg]\k_l(\th) \mathbf{1}_{\{\th\notin C^M\}} \cl Q(d\th)\\
              \leq&  \sup_{d(\th,\bar{\th})\leq \frac{1}{M}}\dbE^{\bar{\dbP}^{\h{v}}_\th}\bigg[\int_0^T |\Pi_{\bar{\th}}(t)-\Pi_\th(t)|dt\bigg]
             +L\cl Q(\th\notin C^M) \\
              \leq & \sup_{d(\th,\bar{\th})\leq \frac{1}{M}}\dbE\bigg[\h R_{\th}(T)\cd\int_0^T |\Pi_{\bar{\th}}(t)-\Pi_\th(t)|dt\bigg] +\frac{L}{M}.
             \end{aligned}
             \end{equation*}
             Here we use the fact  that  $\k_l(\th)=0$ whenever $d(\th,\bar{\th})\geq \frac{1}{2N}$.
             Next, let us show
             \begin{equation*}
              \lim_{M\ra\infty}  \sup_{d(\th,\bar{\th})\leq \frac{1}{M}}\dbE\bigg[\h R_{\th}(T)\cd
              \int_0^T |\Pi_{\bar{\th}}(t)-\Pi_\th(t)|dt\bigg]=0.
              \end{equation*}
             Notice that
             \begin{equation*}
             \begin{aligned}
             \Pi_\th(t) & =\pa_vb_\th(t,\h X_\th(t),\h v(t), \a(t-))p_\th^B(t)+\pa_v\si_\th(t,\h X_\th(t),\h v(t), \a(t-))q_\th^B(t)\\
                           &\q +\pa_v\bar{\si}_\th(t,\h X_\th(t),\h v(t), \a(t-))\bar{q}_\th^B(t)\\
                           & \q+\sum_{i,j=1,i\neq j}^I\pa_v\b_{\th,ij}(t,\h X_\th(t),\h v(t),
                              \a(t-))k_{\th,ij}^B(t)\l_{ij}\mathbf{1}_{\{\a(t-)=i\}}\\
                           &\q -\pa_vf(t,\h X_\th(t), \h Y_\th(t),  \h Z_\th(t), \h {\bar{Z}}_\th(t),
                           \sum_{i,j=1,i\neq j}^I \h K_{\th,ij}(t)\l_{ij}\mathbf{1}_{\{\a(t-)=i\}},\a(t-))p_\th^F(t).
             \end{aligned}
             \end{equation*}
             We just show
             \begin{equation*}
             \begin{aligned}
             & \lim_{M\ra\infty}\sup_{d(\th,\bar{\th})\leq \frac{1}{M}}\dbE\bigg[\h R_{\th}(T)\cd
              \int_0^T \sum_{i,j=1,i\neq j}^I|\pa_v\b_{\th,ij}(t,\h X_\th(t))k_{\th,ij}^B(t)\\
             &\qq\qq\qq\qq\qq  -\pa_v\b_{\bar{\th},ij}(t,\h X_{\bar{\th}}(t))k_{\bar{\th},ij}^B(t)|\l_{ij}\mathbf{1}_{\{\a(t-)=i\}}dt\bigg]=0,
             \end{aligned}
             \end{equation*}
             since the other terms can be estimated similarly. Here and after we use $\pa_v\b_{\th,ij}(t,\h X_\th(t))$ instead of
             $\pa_v\b_{\th,ij}(t,\h X_\th(t),\h v(t),\a(t-))$  for short. Since
             \begin{equation*}
             \begin{aligned}
              &  |\pa_v\b_{\th,ij}(t,\h X_\th(t))k_{\th,ij}^B(t)-\pa_v\b_{\bar{\th},ij}(t,\h X_{\bar{\th}}(t))k_{\bar{\th},ij}^B(t)|\\
               \leq& |\pa_v\b_{\th,ij}(t,\h X_\th(t))k_{\th,ij}^B(t)-\pa_v\b_{\bar{\th},ij}(t,\h X_\th(t))k_{\th,ij}^B(t)|\\
              & +|\pa_v\b_{\bar{\th},ij}(t,\h X_\th(t))k_{\th,ij}^B(t)-\pa_v\b_{\bar{\th},ij}(t,\h X_{\bar{\th}}(t))k_{\th,ij}^B(t)|\\
              & +|\pa_v\b_{\bar{\th},ij}(t,\h X_{\bar{\th}}(t))k_{\th,ij}^B(t)-\pa_v\b_{\bar{\th},ij}(t,\h X_{\bar{\th}}(t))k_{\bar{\th},ij}^B(t)|,
             \end{aligned}
             \end{equation*}
             from  \autoref{assumption3}  we have
             \begin{equation*}
             \begin{aligned}
              &  |\pa_v\b_{\th,ij}(t,\h X_\th(t))k_{\th,ij}^B(t)-\pa_v\b_{\bar{\th},ij}(t,\h X_{\bar{\th}}(t))k_{\bar{\th},ij}^B(t)|\\
                \leq & Ld(\th,\bar{\th})|k_{\th,ij}^B(t)|+L|\h X_\th(t)-\h X_{\bar{\th}}(t)||k_{\th,ij}^B(t)|
              +L|k_{\th,ij}^B(t)-k_{\bar{\th},ij}^B(t)|.
             \end{aligned}
             \end{equation*}
             Hence, according to (\ref{5.5}), (\ref{4.7-2}), (\ref{6.44}), and \autoref{le 6.6}, it follows from Cauchy-Schwarz inequality
             and H\"{o}lder inequality that
             \begin{equation*}
             \begin{aligned}
             &  \dbE\bigg[\h R_{\th}(T)\cd\int_0^T \sum_{i,j=1,i\neq j}^I|\pa_v\b_{\th,ij}(t,\h X_\th(t))k_{\th,ij}^B(t)
                 -\pa_v\b_{\bar{\th},ij}(t,\h X_{\bar{\th}}(t))k_{\bar{\th},ij}^B(t)|\l_{ij}\mathbf{1}_{\{\a(t-)=i\}}dt\bigg]\\
               \leq& L\Big(\sum_{i,j=1,i\neq j}^I\l_{ij}\Big) \cd\Bigg(  \dbE\bigg[\int_0^T \sum_{i,j=1,i\neq j}^I|k_{\th,ij}^B(t)|^2
                 \l_{ij}\mathbf{1}_{\{\a(t-)=i\}}dt   \bigg]^\frac{1}{2}\cd
                 \dbE\bigg[(\h R_{\th}(T) )^2   \bigg]^\frac{1}{2}\cd d(\th,\bar{\th})\\
             &   \q+\dbE\bigg[\int_0^T \sum_{i,j=1,i\neq j}^I|k_{\th,ij}^B(t)|^2
                 \l_{ij}\mathbf{1}_{\{\a(t-)=i\}}dt   \bigg]^\frac{1}{2}\cd  \dbE\bigg[(\h R_{\th}(T))^4   \bigg]^\frac{1}{4}\cd
             \dbE\bigg[\sup_{t\in[0,T]}|\h X_\th(t)-\h X_{\bar{\th}}(t)|^4\bigg]^\frac{1}{4}\\
             &  \q +  \dbE\bigg[\int_0^T \sum_{i,j=1,i\neq j}^I|k_{\th,ij}^B(t)-k_{\bar{\th},ij}^B(t) |^2
                 \l_{ij}\mathbf{1}_{\{\a(t-)=i\}}dt   \bigg]^\frac{1}{2}\cd
                  \dbE\bigg[(\h R_{\th}(T))^2   \bigg]^\frac{1}{2}\Bigg)\\
                 \ra &0,\ \text{as}\ d(\th,\bar{\th})\ra 0.
             \end{aligned}
             \end{equation*}
             \end{proof}

             Now, we are in the position to give the main result--stochastic maximum principle.
\begin{theorem}\label{th 5.1}\rm
              Under   \autoref{assumption3}  and  \autoref{assumption4}, let $\h{v}(\cd)$ be an optimal control and $(\h{X}_\th(\cd),\h{Y}_\th(\cd),
              \h{Z}_\th(\cd),$  $\h{\bar{Z}}_\th(\cd),\h{K}_\th(\cd))$ the solution to Eq. (\ref{4.1}) with $\h{v}(\cd)$. Then there exists a probability
              $\cl{Q}\in \cQ^{\h{v}}$ such that $dt\ts d\bar{\dbP}^{\h{v}}_\th$-a.s.,  $v\in V$,
              \begin{align}\nonumber
                \begin{aligned}
                &  \int_\Th \dbE^{\bar{\dbP}^{\h{v}}_\th}\[ \langle\pa_vH_\th(t,\h{X}_\th(t),\h{Y}_\th(t),\h{Z}_\th(t),\h{\bar{Z}}_\th(t),\h{K}_{\th},
                    \h{v}(t),\a(t-), \\
                &  \qq\qq\qq p^F_\th(t),p^B_\th(t),q^B_\th(t),\cl{q}^B_\th(t),k^B_{\th}(t),\cl{q}^A_\th(t)),v-\h{v}(t)\rangle\Big|\cG_t\]\cl{Q}(d\th)\geq0,
                \end{aligned}
              \end{align}
             where $(p^F_\th(\cd),p^B_\th(\cd),q^B_\th(\cd),\cl{q}^B_\th(\cd),k^B_{\th}(\cd))$ and
             $(\bar{p}^A_\th(\cd),q^A_\th(\cd),\bar{q}^A_\th(\cd),k^A_\th(\cd))$ are the solutions to the adjoint equations (\ref{4.37}) and (\ref{4.36}), respectively.
\end{theorem}

             \begin{proof}
              From It\^{o}'s formula, one has
              \begin{equation*}
              \left\{
              \begin{aligned}
              d[(\h{R} _\th(t))^{-1}R^{1} _\th(t)]&=X_\th^1(t)\pa_xh_\th(t)d\cl{W}_\th(t),\ t\in[0,T],\\
              [(\h{R} _\th(0))^{-1}R^{1} _\th(0)]&=0.
              \end{aligned}
              \right.
              \end{equation*}
              Applying It\^{o}'s formula to $[(\h{R} _\th(\cd))^{-1}R^{1} _\th(\cd)]p^A_\th(\cd)$, $p^F_\th(\cd)Y^{1}_\th(\cd)$,  $p^B_\th(\cd) X^{1}_\th(\cd)$, respectively, we have
              \begin{equation*}
               \begin{aligned}
              &  \dbE^{\bar{\dbP}^{\h{v}}_\th}\[[(\h{R} _\th(T))^{-1}R^{1} _\th(T)]\Psi_\th(\h{X}_\th(T))\]
                =\dbE^{\bar{\dbP}^{\h{v}}_\th}\[\int_0^T\bar{q}^A_\th(t)\pa_xh_\th(t)X_\th^1(t)dt\],\\
              &  \dbE^{\bar{\dbP}^{\h{v}}_\th}\[p_\th^F(T)\pa_x\Phi_\th(\h X_\th(T))X_\th^1(T)+\pa_x\L_\th(Y_\th(0))Y^1_\th(0)\]\\
                =  &-\dbE^{\bar{\dbP}^{\h{v}}_\th}\[\int_0^T[\pa_xf_\th(t)X_\th^1(t)+\langle\pa_vf_\th(t),v(t)-\h{v}(t)\rangle]p^F_\th(t)dt\],\\
              \end{aligned}
              \end{equation*}
              and
              \begin{equation*}
               \begin{aligned}
              &  \dbE^{\bar{\dbP}^{\h{v}}_\th}\[\pa_x\Psi_\th(\h X_\th(T))X_\th^1(T)-p_\th^F(T)\pa_x\Phi_\th(\h X_\th(T))X_\th^1(T)\]\\
               =  &  \dbE^{\bar{\dbP}^{\h{v}}_\th}\[\int_0^T(\pa_xf_\th(t)X_\th^1(t)p_\th^F(t)-\bar{q}_\th^A(t)\pa_xh_\th(t)X_\th^1(t))dt\]\\
              & +\dbE^{\bar{\dbP}^{\h{v}}_\th}\[\int_0^T\langle \pa_vb_\th(t)p^B_\th(t)+\pa_v\si_\th(t)q^B_\th(t)+\pa_v\bar{\si}_\th(t)\bar{q}^B_\th(t)\\
              &\hskip 1.5cm  +\sum_{i,j=1,i\neq j}^I\partial_v \beta_{\th,ij}(t)k^B_{\th,ij}(t)\l_{ij}\mathbf{1}_{\{\a(t-)=i\}},v(t)-\h v(t)\rangle dt\].
              \end{aligned}
              \end{equation*}
              Hence, it yields
              \begin{equation*}
               \begin{aligned}
                &  \dbE^{\bar{\dbP}^{\h{v}}_\th}\[[(\h{R} _\th(T))^{-1}R^{1} _\th(T)]\Psi_\th(\h{X}_\th(T))+\pa_x\L_\th(\h Y_\th(0))Y^1_\th(0)
                   +\pa_x\Psi_\th(\h X_\th(T))X_\th^1(T)\]\\
                =&\dbE^{\bar{\dbP}^{\h{v}}_\th}\bigg[\int_0^T\langle \pa_vb_\th(t)p^B_\th(t)
                       + \pa_v\si_\th(t) q^{B}_\th(t)
                       +\pa_v\bar{\si}_\th(t)\bar{q}^{B}_\th(t)\\
                &\hskip 1.2cm     +\sum_{i,j=1,i\neq j}^I\partial_v \beta_{\th,ij}(t)k^B_{\th,ij}(t)  \l_{ij}\mathbf{1}_{\{\a(t-)=i\}}
                        -p^F_\theta(t) \partial_v f_\theta(t), v(t)-\h{v}(t)\rangle dt\bigg]\\
                &  =\dbE^{\bar{\dbP}^{\h{v}}_\th}\bigg[\int_0^T\langle \pa_vH_\th(t,\h{X}_\th(t),\h{Y}_\th(t),\h{Z}_\th(t),\h{\bar{Z}}_\th(t),
                \h{K}_{\th},\h{v}(t),\a(t-), p^F_\th(t),p^B_\th(t),q^B_\th(t),\\
                &  \qq\qq\qq\cl{q}^B_\th(t),k^B_{\th}(t),\cl{q}^A_\th(t)), v(t)-\h{v}(t)\rangle dt\bigg].\\
                 \end{aligned}
              \end{equation*}
             According to \autoref{th 5.4} and the above equality,  we know, for $v(\cd)\in \mathcal{V}_{ad}$,
             \begin{equation*}
             \begin{aligned}
             0  &   \leq \int_\Th \dbE\[R_\th^{1}(T)\Psi_\th(\h{X}_\th(T))+\h{R}_\th(T)\pa_x\Psi_\th(\h{X}_\th(T))X_\th^{1}(T)
                    +\pa_x\L_\th(\h{Y}_\th(0))Y_\th^{1}(0)\]\cl{Q}(d\theta)\\
                &  =\int_\Th \dbE^{\bar{\dbP}^{\h{v}}_\th}[(\h{R} _\th(T))^{-1}R^{1} _\th(T)\Psi_\th(\h{X}_\th(T))+\pa_x\Psi_\th(\h{X}_\th(T))X_\th^{1}(T)
                    +\pa_x\L_\th(\h{Y}_\th(0))Y_\th^{1}(0) ]\cl{Q}(d\th)\\
                &  =\int_\Th \dbE^{\bar{\dbP}^{\h{v}}_\th}\bigg[\int_0^T\langle\pa_vH_\th(t,\h{X}_\th(t),\h{Y}_\th(t),\h{Z}_\th(t),\h{\bar{Z}}_\th(t),
                   \h{K}_{\th}(t),\h{v}(t),\a(t-), p^F_\th(t),p^B_\th(t),q^B_\th(t),\\
                &  \qq\qq\qq\cl{q}^B_\th(t),k^B_{\th}(t),\cl{q}^A_\th(t)),v(t)-\h{v}(t)\rangle dt\bigg]\cl{Q}(d\theta)\\
                &  =\int_\Th \dbE^{\bar{\dbP}^{\h{v}}_\th}\bigg[\int_0^T\dbE^{\bar{\dbP}^{\h{v}}_\th}\[\langle \pa_vH_\th(t,\h{X}_\th(t),\h{Y}_\th(t),
                   \h{Z}_\th(t),\h{\bar{Z}}_\th(t),\h{K}_{\th},\h{v}(t),\a(t-), p^F_\th(t),p^B_\th(t),q^B_\th(t),\\
                &  \qq\qq\qq
                \cl{q}^B_\th(t),k^B_{\th}(t),\cl{q}^A_\th(t)),v(t)-\h{v}(t)\rangle|\cG_t\]dt\bigg]\cl{Q}(d\th).
             \end{aligned}
             \end{equation*}
             Note that in the above equality, we suppress the superscript $v$ in $\mathbb{R}^{1;v}_\th, X^{1;v}_\th,Y^{1;v}_\th$ and
             $\o$ in $\pa_v H_\th(\o,\cd)$
             for convenience.
             According to \autoref{le 7.1}, one can know that the map
             $$
             \begin{aligned}
                &  (\th,t,\o)\mapsto\dbE^{\bar{\dbP}^{\h{v}}_\th}\[\langle \pa_vH_\th(t,\h{X}_\th(t),\h{Y}_\th(t),
                   \h{Z}_\th(t),\h{\bar{Z}}_\th(t),\h{K}_{\th},\h{v}(t),\a(t-), p^F_\th(t),p^B_\th(t),q^B_\th(t),\\
                &  \qq\qq\qq
                \cl{q}^B_\th(t),k^B_{\th}(t),\cl{q}^A_\th(t)),v(t)-\h{v}(t)\rangle|\mathcal{G}_t\](\o)
             \end{aligned}
              $$
             is a $\cG_t$-progressively measurable process. Fubini's Theorem allows us to show, for arbitrary $v(\cd)\in \cV_{ad}$,
            \begin{equation*}
            \begin{aligned}
            & \dbE^{\bar{\dbP}^{\h{v}}_\th}\bigg[\int_0^T\int_\Th \dbE^{\bar{\dbP}^{\h{v}}_\th}\[ \langle\pa_vH_\th(t,\h{X}_\th(t),\h{Y}_\th(t),
                   \h{Z}_\th(t),\h{\bar{Z}}_\th(t),\h{K}_{\th},\h{v}(t),\a(t-), p^F_\th(t),p^B_\th(t),q^B_\th(t),\\
            &  \qq\qq\qq \cl{q}^B_\th(t),k^B_{\th}(t),\cl{q}^A_\th(t)),v(t)-\h{v}(t)\rangle|\cG_t\]\cl{Q}(d\th)dt\bigg]\geq0,
            \end{aligned}
             \end{equation*}
            which implies $dt\ts d{\bar{\dbP}^{\h{v}}_\th}$-a.s.,
            \begin{equation*}
            \begin{aligned}
            & \int_\Th \dbE^{\bar{\dbP}^{\h{v}}_\th}\[\langle \pa_vH_\th(t,\h{X}_\th(t),\h{Y}_\th(t),
                   \h{Z}_\th(t),\h{\bar{Z}}_\th(t),\h{K}_{\th},\h{v}(t),\a(t-), p^F_\th(t),p^B_\th(t),q^B_\th(t),\\
            &  \qq\qq\qq \cl{q}^B_\th(t),k^B_{\th}(t),\cl{q}^A_\th(t)),v-\h{v}(t)\rangle|\cG_t\]\cl{Q}(d\th)\geq0,\
             v\in V.
            \end{aligned}
            \end{equation*}
            \end{proof}
\subsection{Sufficient Condition}

            This subsection concerns the sufficient condition for the optimal control $\h v(\cd)$. We go on using the notation $\cl W_\th(\cd)=\cl W_\th^{\h{v}}(\cd),$  for each $\th\in \Th$. Denote  $h^v_\th(t):=h_\th(t,X^v_\th(t), \a(t-))$  and
            $h_\th(t):= h^{\h v}_\th(t), t\in[0,T].$
            For $p^B\in \dbR, h\in \dbR, p^F\in \dbR$, define two maps $L:[0,T]\ts\dbR\ts V\ts \cI\ra\dbR,\ S:\dbR\ts \cI\ra\dbR$ by
            \begin{equation*}
            \begin{aligned}
            L(t,x,v,i_0;h,p^B)  & := p^Bh\bar{\si}_\th(t,x,v,i_0),\q  S(x,i_0;p^F)  :=p^F\Phi(x,i_0).
            \end{aligned}
            \end{equation*}

              Next, we give   sufficient conditions for optimal control.
\begin{theorem}\label{th 5.6}\rm
           Under  \autoref{assumption3}  and  \autoref{assumption4}, let $\h{v}(\cd)\in\cV_{ad}$, $\cl{Q}\in \cQ^{\h{v}}$, let $(\h{X}_\th(\cd),\h{Y}_\th(\cd),$
           $\h{Z}_\th(\cd),$ $\h{\bar{Z}}_\th(\cd),\h{K}_{\th}(\cd))$ be the solution to Eq. (\ref{4.1}) with $\h{v}(\cd)$, and let $(p^F_\th(\cd),p^B_\th(\cd),
           q^B_\th(\cd),\bar{q}^B_\th(\cd),k^B_{\th}(\cd))$ and $(p^A_\th(\cd),q^A_\th(\cd),\bar{q}^A_\th(\cd), k^A_\th(\cd))$ be the solutions to Eqs. (\ref{4.37}) and (\ref{4.36}), respectively.
           Assume that
           \begin{itemize}
           	\item [i)]
             the function $H_\th(t,x,y,z,\bar{z},k,v,i_0,p^F,p^B,q^B,\bar{q}^B,k^B,\bar{q}^A)$ defined in (\ref{4.38}) is convex with respect to $x,y,z,\cl{z},k,v$;
       \item [ii)]  $x\mapsto L(\cd,x,\cd,\cd;\cd,\cd)$ is convex and $x\mapsto S(x,\cd;\cd)$ is concave;
            \item [iii)]  for any $v(\cd)\in \mathcal{V}_{ad},t\in[0,T],$ ${\bar{\dbP}^{\h{v}}_\th}$-a.s.,
            \begin{equation}\label{4.43-2}
            \left\{   \begin{aligned}
            & R^v_\th(T)\Psi_\th(X_\th^v(T))-\h R_\th(T)\Psi_\th(\h X_\th(T))-\Psi_\th( X_\th(T))(R^v_\th(T)-\h R_\th(T))\\
            & \q\q\q\q\q\q\q\q\q\q -\h R_\theta(T)\partial_x\Psi_\theta(\h X_\theta (T))(X_\theta^v(T)-\h X_\theta(T))\geq0,\\
            & p_\th^B(t)\Big(\bar{\si}^v_\th(t)h_\th^v(t)-\bar{\si}_\th(t)h_\th(t)-h_\th(t)\pa_x\bar{\si}_\th(t)(X_\th^v(t)-X_\th(t))\\
             &\qq\qq\qq\qq\qq\qq\qq  -\bar{\si}_\th(t) \pa_xh_\th(t)(X_\th^v(t)-X_\th(t))\Big)\geq0,\\
            & \bar{q}_\th^A(t)(\h R_\th(t))^{-1}\Big(R^v_\th(t)h_\th^v(t)-\h R_\th(t)h_\th(t)-h_\th(t)(R^v_\th(t)-\widehat{R}_\th(t))\\
            &\qq\qq\qq\qq\qq\qq\qq -\widehat{R}_\th(t)\pa_xh_\th(t)(X_\th^v(t)-X_\th(t))\Big)\geq0.
            \end{aligned}  \right.
            \end{equation}
             \end{itemize}
            If $dt\ts d{\bar{\dbP}^{\h{v}}_\th}$-a.s.,
              \begin{align}\nonumber
                \begin{aligned}
                &  \int_\Th \dbE^{\bar{\dbP}^{\h{v}}_\th}\[\langle\pa_vH_\th(t,\h{X}_\th(t),\h{Y}_\th(t),\h{Z}_\th(t),\h{\bar{Z}}_\th(t),\h{K}_{\th},\h{v}(t),\a(t-), \\
                &  \qq\qq p^F_\th(t),p^B_\th(t),q^B_\th(t),\cl{q}^B_\th(t),k^B_{\th}(t),\cl{q}^A_\th(t)),v-\h{v}(t)\rangle|\cG_t\]\cl{Q}(d\th)\geq0,
                \end{aligned}
              \end{align}
              then the control $\h{v}(\cd)$ is optimal.
\end{theorem}
              \begin{remark} \rm
         If we know  full information, i.e., $h(\cd)=\bar{\si}(\cd)=0$, then $L(\cd)\equiv0, \h{R}_\th(\cd)=R^v_\th(\cd)\equiv1, v(\cd)\in \cV_{ad}$, and (\ref{4.43-2}) reduces to the condition that  $\Psi_\theta(x)$ is convex with respect to $x$.
If  $\Lambda_\theta(x)=x,\Psi_\th(\cd)\equiv0$, then $J(v(\cd))=\sup_{Q\in \cQ}\int_{\Th}Y^v_\th(0)Q(d\th).$
             Moreover, our FBSDE is only driven by Brownian motion  (without the canonical martingales for  Markov chain), then $p^F_\theta(\cdot)$ satisfies
              \begin{equation*}
              \left\{  \begin{aligned}
              dp^F_\theta(t)    & =\pa_yf_\th(t)p^F_\th(t)dt+\pa_zf_\th(t)p^F_\th(t)dW(t),\ t\in[0,T],\\
                p^F_\theta(0) & =-1.
               \end{aligned}  \right.
             \end{equation*}
              Set  $m_\th(t)=-p^F_\th(t), t\in[0,T].$ It is clear that  $m_\th(\cd)$ satisfies
              \begin{equation*}
               \left\{ \begin{aligned}
                dm_\th(t)   &=\pa_yf_\th(t)m_\th(t)dt+\pa_zf_\th(t)m_\th(t)dW(t),\ t\in[0,T],\\
               m_\theta(0)  &=1,
              \end{aligned} \right.
              \end{equation*}
              which is just Eq. (20) of Hu and Wang \cite{Hu-Wang-20}. In the meanwhile, from the relation $S(x,i_0;p^F)=p^F\Phi(x,i_0)=(-m)\cd\Phi(x,i_0)$, we know that the concavity of $S(x)$ with respect to $x$ is equivalent to the convexity of   $\Phi(x)$ with respect to $x$.
              In this case, our \autoref{th 5.6} reduces to the sufficient condition for the optimal control problem with full information under model uncertainty. See  Hu and Wang \cite[Theorem 3.11]{Hu-Wang-20}.
              \end{remark}
              \begin{proof}[Proof of \autoref{th 5.6}]
               For simplicity, for $v(\cd)\in \cV_{ad}$ and $\th\in \Th$, by $(X_\th^v(\cd),Y_\th^v(\cd),Z_\th^v(\cd),\bar{Z}_\th^v(\cd),K_\th^v(\cd))$ and
               $R_\th^v(\cd)$ we denote the solutions to Eqs. (\ref{4.3}) and (\ref{4.7-1}), respectively. Denote
               $(\eta_\th,\xi_\th,\g_\th,\bar{\g}_\th,\z_\th,\chi_\th)=(X_\th^v-\h{X}_\th,Y_\th^v-\h{Y}_\th,Z_\th^v
                -\h{Z}_\th,\bar{Z}_\th^v-\h{\bar{Z}}_\th, K_\th^v-\h{K}_\th,R_\th^v-\h{R}_\th)$. Then we have
               \begin{equation*}
               \left\{  \begin{aligned}
                  d\eta_\th(t)& =([\pa_xb_\th(t)-h_\th(t)\pa_x\bar{\si}_\th(t)-\bar{\si}_\th(t) \pa_xh_\th(t)     ]\eta_\th(t)+A_{1,\th}(t))dt\\
                            & \q+(\pa_x\si_\th(t)\eta_\th(t)+A_{2,\th}(t))dW(t)+(\pa_x\bar{\si}_\th(t)\eta_\th(t)+A_{3,\th}(t))dG(t)\\
                            & \q+(\pa_x\b_\th(t)\eta_\th(t)+A_{4,\th}(t))\bullet dM(t),\\
                  \eta_\th(0) & =0,
                \end{aligned} \right.
               \end{equation*}
               where
              $$  \begin{aligned}
               A_{1,\th}(t) & =b_\th^v(t)-b_\th(t)-\pa_xb_\th(t)\eta_\th(t)\\
                            & \q-\(\bar{\si}^v_\th(t)h^v_\th(t)-\bar{\si}_\th(t)h_\th(t)-h_\th(t)\pa_x\bar{\si}_\th(t)
                               \eta_\th(t)-\bar{\si}_\th(t) \pa_xh_\th(t)\eta_\th(t)\),\\
                A_{2,\th}(t)& =\si_\th^v(t)-\si_\th(t)-\pa_x\si_\th(t)\eta_\th(t),\q
                A_{3,\th}(t) =\bar{\si}_\th^v(t)-\bar{\si}_\th(t)-\pa_x\bar{\si}_\th(t)\eta_\th(t),\\
                A_{4,\th}(t)& =\b_\th^v(t)-\b_\th(t)-\pa_x\b_\th(t)\eta_\th(t).
              \end{aligned} $$
              Recall that $dG(t)=h_\theta(t)dt+d\cl W_\th(t)$, then
             \begin{equation*}
             \left\{ \begin{aligned}
             d\eta_\th(t) & =([\pa_xb_\th(t)- \bar{\si}_\th(t) \pa_xh_\th(t) ]\eta_\th(t)+A_{1,\th}(t)+h_\th(t)A_{3,\th}(t))dt\\
                        & \q+(\pa_x\si_\th(t)\eta_\th(t)+A_{2,\th}(t))dW(t)+(\pa_x\bar{\si}_\th(t)\eta_\th(t)+A_{3,\th}(t))d\cl{W}_\th(t)\\
                        & \q+(\pa_x\b_\th(t)\eta_\th(t)+A_{4,\th}(t))\bullet dM(t),\\
               \eta_\th(0)& =0.
             \end{aligned} \right.
             \end{equation*}
              Applying It\^{o}'s formula to $p^B_\th(\cd)\eta_\th(\cd)$, we derive
              \begin{equation}\label{4.45}
              \begin{aligned}
               & \dbE^{\bar{\dbP}^{\h{v}}_\th}\[\pa_x\Psi_\th(X_\th(T))\eta_\th(T)-\pa_x\Phi_\th(\h{X}_\th(T))p^F_\th(T)\eta_\th(T)\]\\
                 =&\dbE^{\bar{\dbP}^{\h{v}}_\th}\[\int_0^T p_\th^F(t)\pa_xf_\th(t)\eta_\th(t)
               -\bar{q}_\th^A(t)\pa_xh_\th(t)\eta_\th(t) dt\]\\
               & +\dbE^{\bar{\dbP}^{\h{v}}_\th}\[\int_0^T(A_{1,\th}(t)+h_\th(t)A_{3,\th}(t))p_\th^B(t)+A_{2,\th}(t)q_\th^B(t) +A_{3,\th}(t)\bar{q}_\th^B(t)\\
               & \hskip 1.8cm+\sum_{i,j=1,i\neq j}^IA_{4,\th,ij}(t)k_{\th,ij}^B(t)\l_{ij}\mathbf{1}_{\{\a(t-)=i\}}dt\].
               \end{aligned}
               \end{equation}
              Notice
             \begin{equation*}
              \left\{ \begin{aligned}
              d\xi_\th(t) & =-\bigg\{\pa_yf_\th(t)\xi_\th(t)+\pa_zf_\th(t)\g_\th(t)+\pa_{\bar{z}}f_\th(t)\bar{\g}_\th(t)
                             + \sum_{i,j=1,i\neq j}^I\pa_{k}f_\theta(t)\z_{\th,ij}(t)\l_{ij}\mathbf{1}_{\{\a(t-)=i\}}\\
                           & \qq\q-\bar{\g}_\th(t)h_\th(t)+A_{5,\th}(t)\bigg\}dt+\g_\th(t)dW(t)+\bar{\g}_\th(t)d\cl{W}_\th(t)
                               +\z_\th(t)\bullet dM(t),\\
               \xi_\th(T) & =\F_\th(X^v_\th(T), \a(T))-\F_\th(\h X_\th(T), \a(T)),
             \end{aligned}  \right.
             \end{equation*}
            where
            $$
            \begin{aligned}
            A_{5,\th}(t)  =f_\th^v(t)-f_\th(t)-\pa_yf_\th(t)\xi_\th(t)-
                           \pa_zf_\th(t)\g_\th(t) -  \pa_{\bar{z}}f_\th(t)\bar{\g}_\th(t)
                         - \sum_{i,j=1,i\neq j}^I\pa_{k}f_\th(t)\z_{\th,ij}(t)\l_{ij}\mathbf{1}_{\{\a(t-)=i\}}.
            \end{aligned}
            $$
           Applying It\^{o}'s formula to  $p^F_\th(t)\xi_\th(t)$, it follows
           \begin{equation}\label{4.46-1}
            \begin{aligned}
             \dbE^{\bar{\dbP}^{\h{v}}_\th}[p^F_\th(T)(\F_\th(X^v_\th(T), \a(T))-\F_\th(\h X_\th(T), \a(T))) +\pa_y\L_\th(Y_\th(0))\xi_\th(0)]
             =-\dbE^{\bar{\dbP}^{\h{v}}_\th}\[\int_0^TA_{5,\th}(t)p_\th^F(t)dt\].
           \end{aligned}
           \end{equation}
           In addition,
           \begin{equation*}
           \begin{aligned}
          \chi_\th(0)   =0\q\hbox{and}\q    d\chi_\th(t) =\[h_\th(t)\chi_\th(t)+\h R_\th(t)\pa_xh_\th(t)\eta_\th(t)+A_{6,\th}(t) \]dG(t),\q t\in[0,T],
           \end{aligned}
           \end{equation*}
            where
           $$
            A_{6,\th}(t)=R^v_\th(t)h^v_\th(t)-\widehat{R}_\th(t)h_\th(t)-h_\th(t)\chi_\th(t)-\widehat{R}_\th(t)\pa_xh_\th(t)\eta_\th(t).
           $$
           It follows from  It\^{o}'s formula to $p^A_\th(t)[(R_\th(t))^{-1}\chi_\th(t)]$ on $[0,T]$ that
           \begin{equation}\label{4.47}
            \begin{aligned}
             &\dbE^{\bar{\dbP}^{\h{v}}_\th}\[ \Psi_\th(\h{X}_\th(T))[(\widehat{R}_\th(T))^{-1}\chi_\th(T)]   \]
              =\dbE^{\bar{\dbP}^{\h{v}}_\th}\[\int_0^T\bar{q}_\th^ A(\pa_xh_\th(t)\eta_\th(t)+(\widehat{R}_\th(t))^{-1}A_{6,\th}(t))dt\].
             \end{aligned}
           \end{equation}
            Since $\L_\th(x)$ is convex with respect to $x$, $S(x,\cd;\cd)$ is concave with respect to $x$
            as well as the first inequality of (\ref{4.43-2}), we arrive  at from the definition of $\cQ^{\h{v}}$ that
           \begin{equation*}
             \begin{aligned}
             & J(v(\cd))-J(\h{v}(\cd))\\
               \geq &\int_\Th\dbE\[R^v_\th(T)\Psi_\th(X_\th^v(T))-\h R_\th(T)\Psi_\th(\h{X}_\th(T))+\L_\th(Y^v_\th(0))-\L_\th(\h{Y}_\th(0))\] \cl{Q}(d\th)\\
              \geq &\int_\Th\dbE\[\Psi_\th(\h{X}_\th(T))\chi_\th(T)\] \cl{Q}(d\th)
                +\int_\Th\dbE[\h R_\th(T)\pa_x\Psi_\th(\h X_\th (T))\eta_\th(T)] \cl{Q}(d\th)\\
            &  \q +\int_\Th\dbE[\pa_x\L_\th(\h{Y}_\th(0))\xi_\th(0)] \cl{Q}(d\th)\\
             =&\int_\Th\dbE^{\bar{\dbP}^{\h{v}}_\th}\[(\widehat{R}_\th(T))^{-1}\Psi_\th(\h{X}_\th(T))\chi_\th(T)\] \cl{Q}(d\th)
                  +\int_\Th\dbE^{\bar{\dbP}^{\h{v}}_\th}[\pa_x\Psi_\th(X_\th (T))\eta_\th(T)] \cl{Q}(d\th)\\
            &  \q+\int_\Th\dbE^{\bar{\dbP}^{\h{v}}_\th}[\pa_x\L_\th(\h{Y}_\th(0))\xi_\th(0)] \cl{Q}(d\th)\\
              \geq& \int_\Th\bigg\{\dbE^{\bar{\dbP}^{\h{v}}_\th}\[(\widehat{R}_\th(T))^{-1}\Psi_\th(\h{X}_\th(T))\chi_\th(T)\]
               +\dbE^{\bar{\dbP}^{\h{v}}_\th}\[\pa_x\Psi_\th(X_\th(T))\eta_\th(T)-\pa_x\Phi_\th(\h{X}_\th(T))p^F_\th(T)\eta_\th(T)\]\\
             &\hskip 0.8cm + \dbE^{\bar{\dbP}^{\h{v}}_\th}\[p^F_\th(T)(\F_\th(X^v_\th(T), \a(T))-\F_\th(\h X_\th(T), \a(T)))\] +\dbE^{\bar{\dbP}^{\h{v}}_\th}\[\pa_y\L_\th(\widehat{Y}_\th(0))\xi_\th(0)\]\bigg\} \cl{Q}(d\th) .
            \end{aligned}
           \end{equation*}
           Insert  (\ref{4.45}), (\ref{4.46-1}) and (\ref{4.47}) into the above inequality and from the definitions of
           $A_{i,\th}, i=1,2,\cdots, 6$, one can see
                       \begin{equation*}
             \begin{aligned}
             &  J(v(\cd))-J(\h{v}(\cd))\\
%
               \geq&\int_\Th \Bigg\{ \dbE^{\bar{\dbP}^{\h{v}}_\th}\bigg[\int_0^T\langle\pa_vH_\th(t,\h{X}_\th(t),\h{Y}_\th(t),
                 \h{Z}_\th(t),\h{\bar{Z}}_\th(t),\h{K}_{\th},\h{v}(t),\a(t-),\\
             &   \qq\qq\qq\qq\qq\qq
                 p^F_\th(t),p^B_\th(t),q^B_\th(t),\bar{q}^B_\th(t),k^B_{\th}(t),\bar{q}^A_\th(t)),v(t)-\widehat{v}(t)\rangle dt\bigg]\\
             &    \qq+\dbE^{\bar{\dbP}^{\h{v}}_\th}\bigg[\int_0^Tp_\th^B(t)\(\bar{\si}^v_\th(t)h_\th^v(t)-\bar{\si}_\th(t)h_\th(t)
                   -h_\th(t)\pa_x\bar{\si}(t)\eta_\th(t)-\bar{\si}_\th(t)\pa_xh_\th(t)\eta_\th(t)\)dt\bigg]\\
             &    \qq+\dbE^{\bar{\dbP}^{\h{v}}_\th}\bigg[\int_0^T\bar{q}_\th^A(t)(\h R_\th(t))^{-1}\(R^v_\th(t)h_\th^v(t)-\h R_\th(t)h_\th(t)
                  -h_\th(t)\chi_\th(t) -\h{R}_\th(t)\pa_xh_\th(t)\eta_\th(t)\)dt\bigg]\\
             &    \qq+\dbE^{\bar{\dbP}^{\h{v}}_\th}\bigg[\int_0^T(L(t,X_\th^v(t),v(t),\a(t-);h_\th(t),p^B_\th(t))-L_\th(t)-\pa_xL_\th(t)\eta_\th(t) )dt\bigg]
                  \Bigg\}\cl{Q}(d\th),
             \end{aligned}
            \end{equation*}
        where
             $$
              L_\th(t)=L(t,X_\th(t),\h{v}(t),\a(t-);h_\th(t),p^B_\th(t)),\q
              \pa_xL_\th(t)=\pa_xL(t,X_\th(t),\h{v}(t),\a(t-);h_\th(t),p^B_\th(t)).
             $$
           According to   items $\mathrm{ii)}$ and   $\mathrm{iii)}$ of (\ref{4.43-2}) and  the fact that $L(x)$ is convex with respect to $x$, it follows
            \begin{equation*}
             \begin{aligned}
               J(v(\cd))-J(\h{v}(\cd))
             &    \geq\dbE^{\bar{\dbP}^{\h{v}}_\th}\bigg[\int_\Th\dbE^{\bar{\dbP}^{\h{v}}_\th}\[\langle\pa_vH_\th\big(t,\h{X}_\th(t),\h{Y}_\th(t),\h{Z}_\th(t),
                   \h{\bar{Z}}_\th(t),\h{K}_{\th},\h{v}(t),\a(t-),p^F_\th(t),p^B_\th(t),\\
             &     \qq\qq\qq\qq   q^B_\th(t),\bar{q}^B_\th(t),k^B_{\th}(t),\bar{q}^A_\th(t)\big),v(t)-\h{v}(t)\rangle  |\cG_t\]
                     \cl{Q}(d\th)\bigg]\geq0.
             \end{aligned}
            \end{equation*}
             The proof is complete.
             \end{proof}
\begin{remark}\rm
Compared to Menoukeu-Pamen \cite{Menoukeu-2017}, our work has four significant differences: 1) the partial information filtration is generated by the process $G(\cd)$ in (\ref{4.2}), i.e., $\mathcal{G}_s=\si\{G(r): 0\leq r\leq  s\}$; 2) our original cost functional involves probability $\bar{\dbP}^v_\th$, but not $\mathbb{P}$;
                     3) the   adjoint equation is an FBSDE with Markovian regime switching; 4) we adopt  the linearization method and weak convergence technique to study the  sufficient maximum principle.
\end{remark}

\section{A Risk-minimizing Portfolio Selection Problem}\label{Sec5}

In the section, we review and solve \autoref{ex 1.1} in the introduction.
Recall that  $\th\in\Th=\{1,2\}$ and  $\cQ=\{Q^\l: \l\in[0,1]\}.$
Here $Q^\l$ is the  probability  such that $Q^\l(\{1\})=\l,\   Q^\l(\{2\})=1-\l.$
Define $ \widetilde{\mu}(t)=\dbE[\mu(t)|\cG_t]$. Making use of the separation principle to Eq. (\ref{1.2}) (see \cite[Theorem 3.1]{Xiong-Zhou-2007}), we obtain
\begin{equation*}
\left\{
\begin{aligned}
 dx^v_\th(t)&=r(t)(x_\th^v(t)-v(t))dt+\widetilde{\mu}(t)v(t)dt+\si(t)v(t)d\nu(t),\\
  x_\th^v(0)&=0,
  \end{aligned}
\right.
\end{equation*}
where $\nu(\cd)$ is the innovation process satisfying
\begin{equation*}
\begin{aligned}
 d \nu(t)=\frac{1}{\si(t)}d \log S_{1}(t)-\frac{1}{\si(t)}\(\widetilde{\mu}(t)-\frac{1}{2}\si^2(t)\)dt.
  \end{aligned}
\end{equation*}
From the definition of $G(t)$  (see (\ref{1.3})), we know
\begin{equation}\label{5.3}
\begin{aligned}
 d \nu(t)= dG(t)-\frac{1}{\si(t)}\(\widetilde{\mu}(t)-\frac{1}{2}\si^2(t)\)dt.
  \end{aligned}
\end{equation}
In order to obtain an explicit solution for the optimal control, we consider a special BSDE. Precisely,
   let $f_{1,\th}(\cd), f_{2,\th}(\cd):[0,T]\rightarrow\dbR$ be two bounded functions and let $f_{2,\th}(t)>0, t\in[0,T], \th=1,2$. Consider the following
BSDE with regime switching
\begin{equation*}
\left\{
\begin{aligned}
 dy^v_\th(t)&= -\(f_{1,\th}(t)y^v_\th(t) +\frac{1}{2}f_{2,\th}(t)(v(t))^2\)dt  +\bar{z}_\th^v(t) dG(t)+k^v_{\th}(t)\bullet dM(t),\\
  y_\th^v(T)&=x_\th^v(T).
  \end{aligned}
\right.
\end{equation*}
%
According to \autoref{th 6.1}, the above equation exists a unique solution $(y^v_\th(\cd),\bar{z}^v_\th(\cd),k^v_\th(\cd))$.
The objective is to minimize
\begin{equation*}
\begin{aligned}
 J(v(\cdot))& =\sup_{Q^\l\in \cQ}\int_\Th y_\th^v(0) Q^\l(d\th)=
\sup_{\l\in[0,1]}\(\l y_1^v(0)+(1-\l)y_2^v(0)\)=\max\Big\{y_1^v(0), y_2^v(0)\Big\}.
\end{aligned}
\end{equation*}
In this case, the Hamiltonian is of the form
\begin{equation*}
\begin{aligned}
&H_\th(t,x_\th(t),y_\th(t),\bar{z}_\th(t),v(t), p^F_\th(t),p^B_\th(t),\bar{q}^B_\th(t))\\
 =&\[r(t)(x_\th(t)-v(t))+\mu(t) v(t)\]p^B_\th(t)+\si(t)v(t)\bar{q}^B_\th(t)\\
&-\[f_{1,\th}(t)y_\th(t)+\frac{1}{2}f_{2,\th}(v(t))^2-\frac{1}{\si(t)}\(\mu(t)-\frac{1}{2}\si^2(t)\)\bar{z}_\th(t)
\]p^F_\th(t),
\end{aligned}
\end{equation*}
where $(p^F_\th(\cd),p_\th^B(\cd), \bar{q}_\th^B(\cd),k_\th^B(\cd))$ is the solution to the following adjoint equation
\begin{equation*}
\left\{
\begin{aligned}
 dp_\th^F(t)&=\(f_{1,\th}(t)+\frac{1}{\si^2(t)}(\mu(t)-\frac{1}{2}\si^2(t))^2\)p_\th^F(t)dt
 -\frac{1}{\si(t)}\(\mu(t)-\frac{1}{2}\si^2(t)\)p_\th^F(t)dG(t),\\
 dp_\th^B(t)&=-\(r(t)p_\th^B(t)+ \frac{1}{\si(t)}\(\mu(t)-\frac{1}{2}\si^2(t)\)\bar{q}_\th^B(t)\)dt+\bar{q}_\th^B(t)dG(t)+k_\th^B(t)\bullet dM(t),\\
 p_\th^F(0)&=-1, \q  p_\th^B(T)=-p_\th^F(T).
\end{aligned}
\right.
\end{equation*}
According to (\ref{5.3}), we can rewrite the above equation as
\begin{equation*}
\left\{
\begin{aligned}
 dp_\th^F(t)&=\(f_{1,\th}(t)+\frac{1}{\si^2(t)}(\mu(t)-\frac{1}{2}\si^2(t))(\mu(t)-\tilde{\mu}(t))\)p_\th^F(t)dt\\
 &\q
 -\frac{1}{\si(t)}\(\mu(t)-\frac{1}{2}\si^2(t)\)p_\th^F(t)d\nu(t),\  t\in[0,T],\\
 dp_\th^B(t)&=-\(r(t)p_\th^B(t)-\frac{1}{\si(t)}(\mu(t)-\tilde{\mu}(t))\bar{q}_\th^B(t)\)dt+\bar{q}_\th^B(t)d\nu(t)+k_\th^B(t)\bullet dM(t),\\
 p_\th^F(0)&=-1, \q  p_\th^B(T)=-p_\th^F(T).
\end{aligned}
\right.
\end{equation*}

Let $\h{v}(\cd)$ be an optimal control.   According to \autoref{th 5.1}, there exists a constant $\h{\l}\in[0,1]$
such that
$\max\{\h{y}_1(0),\h{y}_2(0) \}=\h{\l}\h{y}_1(0)+(1-\h{\l})\h{y}_2(0)$  and
\begin{equation*}
\begin{aligned}
&\h{\l}\[ (r(t)-\widetilde{\mu}(t))\widetilde{p}_1^B(t)-\si(t) \widetilde{\bar{q}}_1^B(t)+f_{2,1}(t)\widetilde{p}_1^F(t)\h v(t)\]\\
&+(1-\h{\l})\[ (r(t)-\widetilde{\mu}(t))\widetilde{p}_2^B(t)-\si(t) \widetilde{\bar{q}}_2^B(t)+f_{2,2}(t)\widetilde{p}_2^F(t)\h v(t)\]=0,
\end{aligned}
\end{equation*}
where $(\widetilde{p}_\th^F(\cd),\widetilde{p}_\th^B(\cd), \widetilde{\bar{q}}_\th^B(\cd))$ is the unique solution to the equation:
\begin{equation*}
\left\{
\begin{aligned}
 d\widetilde{p}_\th^F(t)&=\(f_{1,\th}(t)+\frac{1}{\si^2(t)}(\mu(t)-\frac{1}{2}\si^2(t))(\mu(t)-\tilde{\mu}(t))\)\widetilde{p}_\th^F(t)dt\\
 &\q
 -\frac{1}{\si(t)}(\mu(t)-\frac{1}{2}\si^2(t))\wt p_\th^F(t)d\nu(t),\  t\in[0,T],\\
 d\wt p_\th^B(t)&=-\(r(t)\wt p_\th^B(t)-\frac{1}{\si(t)}(\mu(t)-\tilde{\mu}(t) )\wt{\bar{q}}_\th^B(t)\)dt+\wt{\bar{q}}_\th^B(t)d\nu(t),\\
 \wt p_\th^F(0)&=-1, \q  \wt p_\th^B(T)=-\wt p_\th^F(T).
\end{aligned}
\right.
\end{equation*}
For simplicity of editing,   denote
                \begin{equation*}
                      \begin{aligned}
                       \widetilde{p}^F(t)&=
                        \begin{pmatrix}
                       \widetilde{p}^F_1(t)\\
                       \widetilde{p}^F_2(t)
                        \end{pmatrix},\q
                 %
                 %
                  f_1(t)=
                        \begin{pmatrix}
                        f_{1,1}(t) &0  \\
                        0 &   f_{1,2}(t)
                        \end{pmatrix}, \q
                                           \h{X}(t)=
                        \begin{pmatrix}
                       \h x_1(t)\\
                       \h x_2(t)
                        \end{pmatrix}, \\
                       \widetilde{p}^B(t)&=
                        \begin{pmatrix}
                       \widetilde{p}^B_1(t)\\
                       \widetilde{p}^B_2(t)
                        \end{pmatrix},\q
                 \widetilde{\bar{q}}^B(t)=
                        \begin{pmatrix}
                       \widetilde{\bar{q}}^B_1(t)\\
                       \widetilde{\bar{q}}^B_2(t)
                        \end{pmatrix},\q
                 %
                      %
                      %
                  %
                 \widehat{\Lambda}=
                        \begin{pmatrix}
                         \widehat{\lambda} &0  \\
                        0 &   1-\widehat{\lambda}
                        \end{pmatrix},\\
                  %
                        %
                        %
                        %
                        f_2^{\h{\l}}(t)&= \h{\l}f_{2,1}(t)\widetilde{p}_1^F(t)+(1-\h{\l})f_{2,2}(t)\widetilde{p}_2^F(t).
                      \end{aligned}
                      \end{equation*}
                Consequently, we have
                \begin{equation}\label{5.4}
                \left\{
                \begin{aligned}
                &\max\{\h{y}_1(0),\h{y}_2(0) \}=\h{\l}\h{y}_1(0)+(1-\h{\l})\h{y}_2(0),\\
                &  (r(t)-\widetilde{\mu}(t))I_{1\times 2} \widehat{\Lambda}\widetilde{p}^B(t)
                -\si(t)I_{1\ts 2}\widehat{\Lambda}\widetilde{\bar{q}}^B(t)+f^{\h{\l}}_2(t) \h{v}(t)=0,
                \end{aligned}
                \right.
                \end{equation}
                and
             \begin{equation*}
                \left\{
                \begin{aligned}
                d\widehat{X}(t)&=\(r(t)\widehat{X}(t)-  (r(t)-\widetilde{\mu}(t))I_{2\times 1}\widehat{v}(t)\)dt+\si(t)\widehat{v}(t)I_{2\times 1}d\nu(t),\\
                \widehat{X}(0)&=x_0I_{2\times 1}.
                \end{aligned}
                \right.
                \end{equation*}
From the second equality of (\ref{5.4}), one has
 \begin{equation*}
\h{v}(t)=\(f^{\h{\l}}_2(t)\)^{-1}\(  \si(t)I_{1\ts 2} \widehat{\Lambda}\widetilde{\bar{q}}^B(t)
-(r(t)-\widetilde{\mu}(t))I_{1\ts 2}   \widehat{\Lambda}\widetilde{p}^B(t) \),
\end{equation*}
where $\widetilde{p}^F(\cd)$  and $(\widetilde{p}^B(\cd), \widetilde{\bar{q}}^B(\cd))$ are the unique solutions to the following two equations, respectively,
\begin{equation*}
\left\{
\begin{aligned}
 d\widetilde{p}^F(t)&=\(f_{1}(t)+\frac{1}{\si^2(t)}(\mu(t)-\frac{1}{2}\si^2(t))(\mu(t)-\tilde{\mu}(t))I_{2\ts 2}\) \widetilde{p}^F(t)dt\\
                    & - \frac{1}{\si(t)}(\mu(t)-\frac{1}{2}\si^2(t)) \widetilde{p}^F(t)d\nu(t), \q t\in[0,T],\\
 d\widetilde{p}^B(t)&=-\(r(t)\widetilde{p}^B(t)-\frac{1}{\si(t)}(\mu(t)-\tilde{\mu}(t) )\widetilde{\bar{q}}^B(t)\)dt+\widetilde{\bar{q}}^B(t)d\nu(t),\\
 \widetilde{p}^F(0)&=-I_{2\times 1}, \q  \widetilde{p}^B(T)=-\widetilde{p}^F(T).
\end{aligned}
\right.
\end{equation*}
In the meanwhile, the optimal state process $\h X$ satisfies
             \begin{equation*}
                \left\{
                \begin{aligned}
                d\widehat{X}(t)&=\[r(t)\widehat{X}(t)-
                 (r(t)-\widetilde{\mu}(t))I_{2\times 1}
                \(f^{\h{\l}}_2(t)\)^{-1}\(  \si(t)I_{1\ts 2} \widehat{\Lambda}\widetilde{\bar{q}}^B(t)
-(r(t)-\widetilde{\mu}(t))I_{1\ts 2}   \widehat{\Lambda}\widetilde{p}^B(t)  \)\]dt\\
                &\q +\si(t) \(f^{\h{\l}}_2(t)\)^{-1}\(  \si(t)I_{1\ts 2} \widehat{\Lambda}\widetilde{\bar{q}}^B(t)
-(r(t)-\widetilde{\mu}(t))I_{1\ts 2}   \widehat{\Lambda}\widetilde{p}^B(t)  \)  I_{2\ts 1}  d\nu(t),\\
                \widehat{X}(0)&=x_0I_{2\times 1},
                \end{aligned}
                \right.
                \end{equation*}
             which is a linear SDE. Hence, it possesses a unique solution $\widehat{X}(\cd)\in {\cS}_{\dbG}^{2}(0,T;\dbR^2)$.

\section{Conclusions}\label{Sec6}

{This paper consists of two main parts.
The first part addresses a stochastic optimal control problem for Markovian regime-switching systems with partial information under model uncertainty. We derive both necessary and sufficient maximum principles.
Additionally, we apply these theoretical results to tackle a risk-minimizing portfolio selection problem. The second part posits that the coefficients of the state equation, cost functional, and observation process are influenced by the market state parameter $\theta$, effectively illustrating the investor's demand for model selection under varying market conditions.
Our model uncertainty is distinct from the uncertainty arising from perturbing the drift coefficient (see Maenhout \cite{Maenhout-2004, Maenhout-2006}) and is not equivalent to Knightian uncertainty (see Menoukeu-Pamen and  Momeya \cite{Menoukeu-Momeya-2017}, {\O}ksendal and Sulem \cite{Oksendal-Sulem-2012}, and Yi et al. \cite{Yi-2013}).
 Within the current framework, the classical variational approach is invalid. To address this, we employ the weak convergence method to obtain a new variational inequality.}

\end{document}